\documentclass[reqno,11pt,titlepage]{amsart}
\usepackage[foot]{amsaddr}
\usepackage{amssymb}
\usepackage{amsthm}
\usepackage{amsmath}
\usepackage[utf8]{inputenc}
\usepackage[margin=1.0in]{geometry}
\usepackage{pst-node}
\usepackage{tikz-cd} 
\usepackage{mathabx}
\usepackage{mathrsfs}
\usepackage{mathscinet}

\newtheorem{theorem}{Theorem}[section]
\newtheorem{lemma}[theorem]{Lemma}
\newtheorem{proposition}[theorem]{Proposition}
\newtheorem{corollary}[theorem]{Corollary}

\newtheorem{definition}[theorem]{Definition}

\theoremstyle{remark}
\newtheorem{remark}[theorem]{Remark}

\usepackage[T1]{fontenc}
\catcode`@=11
\def\underarrow#1{\mathop{\vtop{\m@th\ialign{##\crcr
$\hfil\displaystyle{#1}\hfil$\crcr
\noalign{\kern3pt\nointerlineskip}
\hfil$\uparrow$\hfil\crcr\noalign{\kern3pt}}}}\limits}
\catcode`@=12

\usepackage[english]{babel}
\usepackage[utf8]{inputenc}

\usepackage{amsfonts}
\def\lim{\mathop{\rm lim}\nolimits}
\def\colim{\mathop{\rm colim}\nolimits}
\def\Spec{\mathop{\rm Spec}}
\def\Hom{\mathop{\rm Hom}\nolimits}

\newcommand{\GU}{\mathrm{GU}}

\newcommand{\G}{\mathrm{G}}

\newcommand{\U}{\mathrm{U}}
\newcommand{\Qp}{\mathbb{Q}_p}
\newcommand{\Z}{\mathbb{Z}}
\newcommand{\J}{\mathrm{J}}
\newcommand{\R}{\mathbb{R}}
\newcommand{\C}{\mathbb{C}}

\newcommand{\Res}{\mathrm{Res}}
\newcommand{\Irr}{\mathrm{Irr}}
\newcommand{\tr}{\mathrm{tr}}

\newcommand{\GL}{\mathrm{GL}}

\newcommand{\id}{\mathrm{id}}
\newcommand{\Id}{\mathrm{Id}}

\newcommand{\Gm}{{\mathbb{G}_m}}

\newcommand{\mc}{\mathcal}
\newcommand{\mf}{\mathfrak}
\newcommand{\Q}{\mathbb{Q}}

\newcommand{\Int}{\mathrm{Int}}

\newcommand{\diag}{\mathrm{diag}}
\newcommand{\Lie}{\mathrm{Lie}}

\newcommand{\LL}{{}^L}

\newcommand{\der}{\mathrm{der}}

\newcommand{\D}{\mathbb{D}}
\newcommand{\ov}{\overline}

\newcommand{\Gal}{\mathrm{Gal}}

\newcommand{\Leta}{{}^L\eta}

\newcommand{\bas}{\mathrm{bas}}
\newcommand{\temp}{\mathrm{temp}}
\newcommand{\unit}{\mathrm{unit}}
\newcommand{\iso}{\mathrm{iso}}
\newcommand{\mH}{\mathrm{H}}
\newcommand{\SL}{\mathrm{SL}}

\newcommand{\ab}{\mathrm{ab}}

\newcommand{\M}{\mathrm{M}}

\newcommand{\St}{\mathrm{St}}
\newcommand{\I}{\mathrm{I}}

\newcommand{\Ind}{\mathrm{Ind}}

\newcommand{\bdd}{\mathrm{bdd}}
\newcommand{\hyp}{\mathrm{hyp}}
\newcommand{\Sht}{\mathrm{Sht}}
\newcommand{\Spd}{\mathrm{Spd}}
\newcommand{\alg}{\mathrm{alg}}
\newcommand{\Act}{\mathrm{Act}}
\newcommand{\Bun}{\mathrm{Bun}}
\newcommand{\Rep}{\mathrm{Rep}}
\newcommand{\Dc}{\mathrm{D}}
\newcommand{\Dlis}{\mathrm{D}_{\mathrm{lis}}}
\newcommand{\Perf}{\mathrm{Perf}}
\newcommand{\LLC}{\mathrm{LLC}}
\newcommand{\IndPerf}{\mathrm{IndPerf}}
\newcommand{\ttau}{\tilde{\tau}}
\newcommand{\SuppCusp}{\mathrm{SuppCusp}}
\newcommand{\scusp}{\mathrm{sc}}
\newcommand{\res}{\mathrm{res}}
\newcommand{\odd}{\mathrm{odd}}
\newcommand{\ad}{\mathrm{ad}}
\newcommand{\FS}{\mathrm{FS}}
\newcommand{\mP}{\mathrm{P}}

\def\ol{\overline}
\def\ra{\rightarrow}
\def\Div{\mathrm{Div}}

\newcommand{\semis}{\mathrm{ss}}

\title{Compatibility of the Fargues--Scholze correspondence for Unitary Groups}
\author{Alexander Bertoloni Meli}
\email{abertolo@umich.edu}
\address{University of Michigan, Ann Arbor, Michigan, USA}
\author{Linus Hamann}
\email{lhamann@math.princeton.edu}
\address{Princeton University, Princeton, New Jersey, USA} 
\author{Kieu Hieu Nguyen} 
\email{knguyen@uni-muenster.de}
\address{University of M{\"u}nster, M{\"u}nster, Germany}
\thanks{
  \small{
  Orchid ID: 0000-0001-5015-0718\\
  \hspace*{1.8em}MSC class:	11S37
  }
  }
\usepackage[english]{babel}
\usepackage[utf8]{inputenc}

\usepackage{hyperref}
\begin{document}

\begin{abstract} 
We study unramified unitary and unitary similitude groups in an odd number of variables. Using work of the first and third named authors (\cite{BMN}) on the Kottwitz Conjecture for the similitude groups, we show that the Fargues--Scholze local Langlands correspondence (\cite{FS}) agrees with the semi-simplification of the local Langlands correspondences constructed in \cite{Mok,KMSW,BMN} for the groups we consider. This compatibility result is then combined with the spectral action constructed by \cite[Chapter~X]{FS}, to verify their categorical form of the local Langlands conjecture for supercuspidal $L$-parameters \cite[Conjecture~X.2.2]{FS}. We deduce Fargues' conjecture \cite[Conjecture~4.1]{Fa} and prove the strongest form of Kottwitz's conjecture \cite[Conjecture~I.0.1]{HKW} for the groups we consider, even in the case of non minuscule $\mu$.
\end{abstract}

\maketitle

\section{Introduction}{\label{s: intro}}
For $G/\Q_p$ a connected reductive group, the Langlands correspondence relates the set $\Pi_{\C}(G)$ of isomorphism classes of smooth irreducible $\C$-representations of $G(\Q_p)$ and the set $\Phi(G)$ of conjugacy classes  of L-parameters, 
\[ \phi: W_{\mathbb{Q}_{p}} \times \SL_{2}(\mathbb{C}) \rightarrow \phantom{}^{L}G(\mathbb{C}). \]

For unitary groups, the existence of such a correspondence is known by work of Mok \cite[Theorem~1.5]{Mok} and Kaletha--Minguez--Shin--White (\cite[Theorem 1.6.1]{KMSW}). For odd unitary similitude groups, a correspondence was constructed by the first and third named authors (\cite[Theorem~1.2]{BMN}), by lifting the unitary correspondence. We denote these correspondences by $\LLC_G$, where $G=\GU_n$ or $\U_n$.

On the other hand, Fargues--Scholze \cite{FS} have recently constructed a \emph{semi-simplified} local Langlands correspondence for all connected reductive $G$. Namely, they construct a map
\[ \LLC_{G}^{\mathrm{FS}}: \Pi_{\ov{\Q}_{\ell}}(G) \rightarrow \Phi^{\mathrm{ss}}(G) \quad \quad \pi \mapsto \phi_{\pi}^{\mathrm{FS}},  \]
where $\Phi^{\mathrm{ss}}(G)$ denotes the set of conjugacy classes of continuous semisimple maps
\[ \phi: W_{\mathbb{Q}_{p}} \rightarrow \phantom{}^{L}G(\overline{\mathbb{Q}}_{\ell}), \]
that commute with the projection $\phantom{}^{L}G(\overline{\mathbb{Q}}_{\ell}) \rightarrow W_{\Q_p}$. They do so by constructing excursion operators on the moduli stack of $G$-bundles on the Fargues--Fontaine curve. 

An important problem is to show that the correspondences are compatible. Namely, one wants that the diagram
\begin{equation*}
\begin{tikzcd}[ampersand replacement=\&]
            \Pi_{\C}(G)  \ar[rr, "\LLC_{G}"] \arrow[d, swap, "\iota^{-1}_{\ell}"] \& \&   \Phi(G) \ar[d,"(-)^{ss}"] \\
            \Pi_{\ov{\Q}_{\ell}}(G) \arrow[rr, "\LLC^{\FS}_G"]\& \& \Phi^{\mathrm{ss}}(G)
        \end{tikzcd}
\end{equation*} 
commutes, where we fix an isomorphism $\iota_{\ell}^{-1}: \mathbb{C} \xrightarrow{\sim} \overline{\mathbb{Q}}_{\ell}$, and the semi-simplification map $(-)^{ss}$ precomposes an $L$-parameter $\phi \in \Phi(G)$ with the map  
\[ W_{\mathbb{Q}_{p}} \ra W_{\mathbb{Q}_{p}} \times \SL_{2}(\mathbb{C}) \]
\[g  \mapsto (g,\begin{pmatrix}
|g|^{\frac{1}{2}} & 0 \\
0 & |g|^{-\frac{1}{2}}  
\end{pmatrix}) \]  
and then applies the isomorphism $\iota_{\ell}^{-1}: \mathbb{C} \xrightarrow{\sim} \overline{\mathbb{Q}}_{\ell}$, where $|\cdot|: W_{\mathbb{Q}_{p}} \rightarrow W_{\mathbb{Q}_{p}}^{ab} \simeq \mathbb{Q}_{p}^{\times} \rightarrow \mathbb{C}$ is the norm character. Our first main theorem is then as follows.
\begin{theorem}{\label{thm: introcompatibility}}
Let $G = \U_n, \GU_n$, where $n$ is odd and the unitary groups are defined relative to the unramified extension $E/\Q_p$. Then $\LLC_G$ and $\LLC^{\mathrm{FS}}_G$ are compatible. Namely, we have an equality $\LLC^{\mathrm{FS}}_G(\iota^{-1}_{\ell}\pi) = \LLC_G(\pi)^{ss}$ for all $\pi \in \Pi_{\C}(G)$.
\end{theorem}
Previous results of this form for $G = \GL_{n}$ are due to Fargues--Scholze \cite[Section~IX.7.3]{FS}, Hansen--Kaletha--Weinstein for inner forms of $\GL_{n}$ \cite[Theorem~1.0.3]{HKW}, and the second named author for $\mathrm{GSp}_{4}$ and its inner forms \cite[Theorem~1.1]{Ham}. 

We give a brief indication of how this theorem is proven, following the strategy outlined in \cite{Ham} for $\mathrm{GSp}_{4}$. Compatibility of the correspondences for $\U_n$ is deduced from compatibility for $\GU_n$ and the compatibility of each correspondence with the central isogeny $\U_n \to \GU_n$. Hence, the main difficulty is to prove compatibility for $\GU_n$.

The proof of compatibility for $G = \GU_n$ rests on a detailed study of the cohomology of basic local shtuka spaces $\Sht(G,b,\mu)_{\infty}$, as defined in \cite[Lecture~23]{SW}. Crucially, the cohomology complex  $R\Gamma^{\flat}_c(G,b,\mu)[\rho]$ attached to $\Sht(G,b,\mu)_{\infty}$ and $\rho \in \Pi_{\ov{\Q}_{\ell}}(J_b)$ carries an action of $G(\Q_p) \times W_E$ that is known to be related to both $\LLC_G(\iota_{\ell}\rho)$ and $\LLC^{\mathrm{FS}}_G(\rho)$. In the case where $\LLC_G(\iota_{\ell}\rho)$ is a \emph{supercuspidal} $L$-parameter, a complete description of $R\Gamma^{\flat}_c(G,b,\mu)[\rho]$ in the Grothendieck group of $G(\mathbb{Q}_{p}) \times W_{E}$-representations in terms of $\LLC_G$, known as the \emph{Kottwitz Conjecture}, was proven by the first and third named authors (\cite[Theorem 6.1]{BMN}). On the other hand, just enough is known (by \cite{Ko}, \cite{Ham}) about the relation between $R\Gamma^{\flat}_c(G,b,\mu)[\rho]$ and $\LLC^{\mathrm{FS}}_{G}$ to deduce compatibility in the supercuspidal case. 

The case where the $L$-packet of the parameter $\LLC_G(\pi)$ consists entirely of non-supercuspidal representations then follows by induction from the supercuspidal case and the compatibility of $\LLC_G$ and $\LLC^{\mathrm{FS}}_G$ with parabolic induction. Finally, we have the case when $\LLC_G(\pi)$ has an $L$-packet consisting of both non-supercuspidal and supercuspidal representations. In this case, we can prove compatibility for the non-supercuspidal representations in the packet by induction as before. We then deal with the supercuspidal representations by combining compatibility for the non-supercuspidal elements in the $L$-packet of $\LLC_G(\pi)$ with a description of $R\Gamma_c(G,b,\mu)[\rho]$ as a $G(\Q_p)$-representation (as in \cite{HKW}) to deduce compatibility in the remaining case. The key point is that $R\Gamma_{c}(G,b,\mu)[\rho]$ is isomorphic to a Hecke operator applied to the representation $\rho$, and since Hecke operators and excursion operators commute, it follows that all the representations occurring in this complex will have Fargues--Scholze parameter equal to that of $\rho$. This allows us to propogate compatibility for the non-supercuspidal representations in the $L$-packet to the supercuspidal ones.

We briefly remark on our choice to consider $E/\Q_p$ unramified instead of an arbitrary quadratic extension of $p$-adic fields. These assumptions come from \cite{BMN}, where the unramified condition is needed for the relevant Igusa varieties to be defined. The base field is $\Q_p$ as in that case, the relevant global unitary similitude groups satisfy the Hasse principle, which is needed in several places when working with the trace formula for the cohomology of Shimura varieties using global extended pure inner twists.

The importance of Theorem \ref{thm: introcompatibility} is that it allows one to combine what is known about the classical local Langlands correspondence, such as the endoscopic character identities, with the geometric Langlands techniques introduced in \cite{FS}. This gives one enough information to verify parts of the categorical local Langlands conjecture of Fargues--Scholze \cite[Conjecture~X.I.4]{FS}. 

We briefly recall the statement of this conjecture. It relates sheaves on the stack of Langlands parameters "the spectral side" to sheaves on $\Bun_{G}$ the moduli stack of $G$-bundles on the Fargues--Fontaine curve "the geometric side". In particular, if $G$ is a quasi-split connected reductive group with Whittaker datum $\mf{w} := (B,\psi)$, one can define the Whittaker sheaf $\mathcal{W}_{\psi}$ by $\mathrm{cInd}_{U}^{G}(\psi)$, regarded as a sheaf on $\Bun_{G}^{1}$, the neutral Harder--Narasimhan (abbv. HN)-stratum of $\Bun_{G}$, and then extended by $0$ to all of $\Bun_{G}$. On the spectral side, one considers the stack of $\ol{\mathbb{Q}}_{\ell}$-valued Langlands parameters $X_{\hat{G}}$ of $G$ and writes $\Perf^{\mathrm{qc}}(X_{\hat{G}})$ (resp. $\Dc^{b,\mathrm{qc}}_{\mathrm{coh}}(X_{\hat{G}})$) for the derived category of perfect complexes on $X_{\hat{G}}$ and quasi-compact support (resp. bounded derived category of sheaves with coherent cohomology and quasi-compact support), as in \cite{DH,Zhu1} and \cite[Section~VIII.I]{FS}. On the geometric side, one considers $\Dlis(\Bun_{G},\ol{\mathbb{Q}}_{\ell})$ the category of lisse-\'etale $\ol{\mathbb{Q}}_{\ell}$-sheaves and $\Dlis(\Bun_{G},\ol{\mathbb{Q}}_{\ell})^{\omega}$ the sub-category of compact objects, as defined in \cite[Section~VII.7]{FS}. Fargues and Scholze conjecture that the map 
\[ \Perf^{\mathrm{qc}}(X_{\hat{G}}) \ra \Dlis(\Bun_{G},\ol{\mathbb{Q}}_{\ell}) \quad \quad C \mapsto C \star \mathcal{W}_{\psi}, \]
where $\star$ denotes the spectral action, as constructed in \cite[Chapter~X]{FS} (see \S \ref{sss: spectralaction}) defines a fully faithful embedding that extends to an equivalence of $\ol{\mathbb{Q}}_{\ell}$-linear categories:
\begin{equation}{\label{eqn: introFSconj}}
    \Dc_{\mathrm{coh}}^{b,\mathrm{qc}}(X_{\hat{G}}) \simeq  \Dlis(\Bun_{G},\ol{\mathbb{Q}}_{\ell})^{\omega}.
\end{equation}
Roughly, this conjecture captures the fact that for $\pi \in \Pi_{\ov{\Q}_{\ell}}(G)$, the $\pi$-isotypic part of $R\Gamma_c(G,b,\mu)$ should have cohomology precisely governed by the Fargues--Scholze parameter of $\pi$. 

Compatibility can be used to make progress on this conjecture. In this paper, we show this for the locus defined by supercuspidal $L$-parameters. Indeed, let $\phi$ be a supercuspidal $L$-parameter of $G$ and consider the connected component  $C_{\phi} \hookrightarrow X_{\widehat{G}}$, given by the unramified twists of $\phi$. This is isomorphic to a torus quotiented out by $S_{\phi}$, the centralizer of $\phi$ in $\hat{G}$. Therefore, we have a natural map $C_{\phi} \ra [\Spec(\ol{\mathbb{Q}}_{\ell})/S_{\phi}]$ and vector bundles on the target identify with representations of $S_{\phi}$. Therefore, given $W \in \Rep_{\ol{\mathbb{Q}}_{\ell}}(S_{\phi})$, we can write $\Act_{W}$ for the spectral action of the pullback of the vector bundle defined by $W$ to $C_{\phi}$ on $\Dlis(\Bun_{G},\ol{\mathbb{Q}}_{\ell})$. We explain some part of what \eqref{eqn: introFSconj} becomes in the case that we restrict to the (not full) subcategory coming from perfect complexes on $[\Spec(\ol{\mathbb{Q}}_{\ell})/S_{\phi}]$, by fixing central characters. Define the sub-category $\Dlis^{C_{\phi}}(\Bun_{G},\ol{\mathbb{Q}}_{\ell})^{\omega}$ to be the image of $\Act_{\mathbf{1}}$, where $\mathbf{1}$ denotes the trivial representation, and consider the subcategory $\Dlis^{C_{\phi},\chi}(\Bun_{G},\ol{\mathbb{Q}}_{\ell})^{\omega}$ spanned by representations with fixed central character $\chi$ determined by $\phi$.  This is valued in sheaves whose restrictions to the HN-strata of $\Bun_{G}$ have irreducible constituents with Fargues--Scholze parameter equal to $\phi$. It follows that objects in $\Dlis^{C_{\phi},\chi}(\Bun_{G},\ol{\mathbb{Q}}_{\ell})^{\omega}$ will only be supported on the HN-strata corresponding to the basic elements $B(G)_{\bas}$, where it will be valued in supercuspidal representations with fixed central character. Thus,
\[ \Dlis^{C_{\phi},\chi}(\Bun_{G},\ol{\mathbb{Q}}_{\ell})^{\omega} \simeq \bigoplus_{b \in B(G)_{\bas}} \bigoplus_{\pi_{b}} \pi_{b} \otimes \Perf(\ol{\mathbb{Q}}_{\ell}), \]
where $\Perf(\ol{\mathbb{Q}}_{\ell})$ is the derived category of perfect complexes of $\ol{\mathbb{Q}}_{\ell}$ vector spaces and $\pi_{b}$ ranges over representations of the $\sigma$-centralizer $J_{b}$ of $b$ with Fargues--Scholze parameter equal to $\phi$. The categorical conjecture can be interpreted as an equivalence 
\begin{equation}{\label{eqn: introrefinedFSconj}}
    \Perf([\Spec{\ol{\mathbb{Q}}_{\ell}}/S_{\phi}]) \xrightarrow{\simeq} \bigoplus_{b \in B(G)_{\bas}} \pi_{b} \otimes \Perf(\ol{\mathbb{Q}}_{\ell}) \quad \quad W \mapsto \Act_{W}(\pi_{\mathbf{1}}), 
\end{equation}
which is exact with respect to the standard $t$-structure, as in \cite[Conjecture~X.2.2]{FS}, and $\pi_{\mathbf{1}}$ is the (conjectural) unique generic representation of $G(\mathbb{Q}_{p})$ with L-parameter $\phi$ with respect to the choice of Whittaker datum $\psi$.  Now assume we know compatibility of the Fargues--Scholze  correspondence for $G$ with the refined local Langlands correspondence of Kaletha \cite{Kal}. Then, for a fixed $b \in B(G)_{\bas}$, we can enumerate the set of all $\pi_{b}$ with Fargues--Scholze parameter $\phi$. It is parameterized by the set of all representations $W$ of $S_{\phi}$ with restriction to $Z(\hat{G})^{\Gamma}$ equal to the $\kappa$-invariant of $b$. In particular, if we write $\pi_{W}$ for the representation with supercuspidal parameter $\phi$ that corresponds to an irreducible $W \in \Rep_{\ol{\mathbb{Q}}_{\ell}}(S_{\phi})$ under the refined local Langlands, then  \eqref{eqn: introrefinedFSconj} in this case follows from exhibiting an isomorphism
\[ \Act_{W^{\vee}}(\pi_{\mathbf{1}}) \simeq \pi_{W}, \]
for all irreducible $W$\footnote{This dual is forced on us by the various normalizations. In particular, for $G = \GL_{1}$ the Hecke operator corresponding to the standard representation of $\GL_{1}$ will carry the connected component indexed by an element $1 \in \mathbb{Z} \simeq X^{*}(\GL_{1}) \simeq^{\kappa^{-1}} \pi_{0}(|\Bun_{\GL_{1}}|)$ to the connected component indexed by $0$ (cf. Definition \ref{shiftedbgu}, Lemma \ref{shimhecke}, and \cite[Corollary~5.4]{Vi}), so the relationship is forced on us by Theorem \ref{thm: RGammavsAct}. Here $n \in \mathbb{Z}$ is sent to the character $z \mapsto z^{n}$ of $\GL_{1}$.}. In \S \ref{s: applications}, we verify this expectation for $\GU_{n}$, showing the categorical form of the local Langlands conjecture in this case. 
\begin{theorem}{(Proposition \ref{prop: Allactval})}{\label{thm: introFSconj}}
For $G = \GU_{n}$, we let $\phi$ be a supercuspidal parameter. For any irreducible representation $W \in \Rep_{\ol{\mathbb{Q}}_{\ell}}(S_{\phi})$ with corresponding representation $\pi_{W}$ of $J_{b}$ for $b \in B(G)_{\bas}$ with $\kappa(b)$ equal to the central character of $W$, we have an isomorphism 
\[ \Act_{W^{\vee}}(\pi_{\mathbf{1}}) \simeq \pi_{W} \]
of $J_{b}(\mathbb{Q}_{p})$-representatives, viewed as sheaves on the HN-stratum $\Bun_{G}^{b} \subset \Bun_{G}$, where $\pi_{\mathbf{1}}$ denotes the unique $\mf{w}$-generic representation with parameter $\phi$. 
\end{theorem}
As seen above, to prove Theorems \ref{thm: introcompatibility} and \ref{thm: introFSconj}, we just need the description of the complexes $R\Gamma_{c}(G,b,\mu)[\rho]$ in the Grothendieck group of $G(\mathbb{Q}_{p}) \times W_{E}$-modules for $\mu$ a minuscule cocharacter and $\rho$ a representation with supercuspidal $L$-parameter, as proven in \cite{BMN}. However, using this result, we can actually strengthen it to completely describe the complexes of $G(\mathbb{Q}_{p}) \times W_{E}$-modules $R\Gamma_{c}(G,b,\mu)[\rho]$ without passing to the Grothendieck group and in the case that $\mu$ is not necessarily minuscule (Theorem \ref{nonmindescr}). One way of efficiently describing this is through an eigensheaf. In particular, if we consider $k(\phi)_{\mathrm{reg}}$ the sheaf on $C_{\phi}$ defined by pulling back the regular representation of $S_{\phi}$ viewed as a sheaf on $[\Spec(\ol{\mathbb{Q}}_{\ell})/S_{\phi}]$, we obtain by the previous Theorem an isomorphism
\[ \mathcal{G}_{\phi} := k(\phi)_{\mathrm{reg}} \star \pi_{\mathbf{1}} \simeq \prod_{b \in B(G)_{\bas}} \bigoplus_{\pi_{b} \in \Pi_{\phi}(J_{b})} j_{b!}(\pi_{b}), \]
where $\Pi_{\phi}(J_{b})$ is the $L$-packet over $\phi$ of $J_{b}$. It is easy to check from the construction that this is a Hecke eigensheaf. In particular, given any (not necessarily minuscule) geometric dominant cocharacter $\mu$ with reflex field $E_{\mu}$, and a Hecke operator $T_{\mu}: \Dlis(\Bun_{G},\ol{\mathbb{Q}}_{\ell}) \ra \Dlis(\Bun_{G},\ol{\mathbb{Q}}_{\ell})^{BW_{E_{\mu}}}$ defined by the highest weight representation $V_{\mu}$ of $\hat{G}$, we have an isomorphism 
\[ T_{\mu}(\mathcal{G}_{\phi}) \simeq \mathcal{G}_{\phi} \boxtimes r_{\mu} \circ \phi \]
of sheaves in $\Dlis(\Bun_{G},\ol{\mathbb{Q}}_{\ell})$ with continuous $W_{E_{\mu}}$-action. Here $r_{\mu}: \widehat{G} \rtimes W_{E_{\mu}} \rightarrow \GL(V_{\mu})$ extends the action of $\widehat{G}$ and is as defined in \cite[Lemma (2.1.2)]{KottwitzTOO}. Moreover, if we endow $\mathcal{G}_{\phi}$ with the $S_{\phi}$-action where it acts by $W^{\vee}$ on the representation $\pi_{W}$, this isomorphism is $S_{\phi}$-equivariant in the obvious sense. If we let $b \in B(G,\mu)$ be the unique basic element in the $\mu$-admissible locus and restrict this isomorphism to the connected component of $\Bun_{G}$ containing $b$, then this tells us that the complexes $R\Gamma_{c}(G,b,\mu)[\rho]$ are all concentrated in degree $0$ (which corresponds to middle degree under our conventions) and can be described in terms members of the $L$-packet $\Pi_{\phi}(G)$ pairing with the irreducible summands of $r_{\mu} \circ \phi$, as predicted by the Kottwitz conjecture \cite[Conjecture~1.0.1]{HKW}. In particular, this verifies Fargues' and Kottwitz's conjecture. 
\begin{theorem}{(Theorems \ref{nonmindescr},\ref{thm: GUneigensheaf})}{\label{thm: intro GUeigensheaf}}
For $G = \GU_{n}$, the sheaf $\mathcal{G}_{\phi}$ constructed above satisfies conditions (i)-(iv) of Fargues' Conjecture \cite[Conjecture~4.1]{Fa}. In particular, the Kottwitz conjecture \cite[Conjecture~I.0.1]{HKW} holds for any geometric dominant cocharacter $\mu$. 
\end{theorem}
In \S \ref{ss: unitarygroupapps}, we also treat the case of $G = \U_{n}$. As mentioned above, compatibility of the Fargues--Scholze correspondence for $\U_{n}$ follows easily from the case of $\GU_{n}$ and compatibility of both correspondences with the  central isogeny $\varphi: \U_{n} \ra \GU_{n}$. We can combine this fact with a similar kind of compatibility of the spectral actions of $X_{\hat{\U}_{n}}$ on $\Bun_{\U_{n}}$ and $X_{\hat{\GU}_{n}}$ on $\Bun_{\GU_{n}}$ with respect to the induced map $\Bun_{\U_{n}} \ra \Bun_{\GU_{n}}$ (Proposition \ref{actcentralisog}) to use the values of the $\Act$-functors for $\GU_{n}$ to compute the $\Act$-functors for $\U_{n}$. This allows us to prove the following.
\begin{theorem}{(Proposition \ref{prop: Unallactval}, Corollary \ref{Cor: Uneigsheaf})}{\label{thm: introUneigensheaf}}
The analogues of Theorem 1.2 and 1.4 hold for $G = \U_{n}$ and $p > 2$. 
\end{theorem}
\begin{remark}
The assumption that $p > 2$ is most likely an artifact of the proof. In particular, in order to show that the $\Act$-functors are t-exact, we need to invoke a Theorem of  Hansen \cite[Theorem~1.1]{Han} which guarantees that the local Shtuka spaces uniformizing Shimura varieties have isotypic parts with respect to representations with supercuspidal Fargues--Scholze parameter concentrated in degree $0$. For $\U_{n}$, the relevant local Shtuka spaces will only uniformize Shimura varieties of abelian type, and the relevant uniformization result is only known to hold if $p > 2$, by work of Shen \cite{Shen}.
\end{remark} 
This illustrates the power of these methods. In particular, we started out with the initial input of the statement \cite[Theorem~6.1]{BMN} which described the local shtuka spaces that uniformize certain PEL type Shimura varieties and used it to show compatibility. Then, using purely local techniques, we were able to propagate this to  describe local shtuka spaces that appear in the uniformization of both PEL and abelian type Shimura varieties and local shtuka spaces associated with all dominant cocharacters. We suspect this is one of many possible applications of compatibility to the categorical conjecture and the cohomology of Shimura varieties.  In particular, we point the reader to the paper of \cite{Ko1}, where compatibility of the Fargues--Scholze correspondence with that of Harris--Taylor is used to give a more flexible proof of the torsion vanishing results of Caraiani--Scholze \cite{CS,CS1}, as well as work of the second named author where compatibility is used to construct eigensheaves attached to parameters induced from the maximal torus $T$ \cite{Ham1}.

In \S \ref{ss: od unr unit groups}-\ref{ss: classicalLLC}, we review the classical local Langlands correspondence for $\U_{n}$ and $\GU_{n}$, writing out the endoscopic character identities in this case explicitly. In \S \ref{ss: FSLLC}, we review the Fargues--Scholze local Langlands correspondence and the spectral action, proving a result on compatibility of spectral action with central isogenies. In \S \ref{s: proofofcomp}, we give the proof of Theorem \ref{thm: introcompatibility}, and then in \S \ref{s: applications}, we combine with the spectral action to show Theorems \ref{thm: introFSconj}, \ref{thm: intro GUeigensheaf}, and \ref{thm: introUneigensheaf}. 
\subsection*{Acknowledgements}
We wish to thank David Hansen, who inspired us to embark upon this project. A.B.M.~was partially supported by NSF grant DMS-1840234. K.H.N.~was partially supported by the ERC in form of Consolidator Grant 770936: NewtonStrat and by Deutsche Forschungsgemeinschaft (DFG) through the Collaborative Research Centre TRR 326 "Geometry and Arithmetic of Uniformized Structures", project number 444845124.

\tableofcontents
\section*{Conventions and Notations}
For $F$ a $p$-adic field, we let $\Gamma$ denote the absolute Galois group of $F$. When $K/F$ is a Galois extension, we denote the Galois group by $\Gamma_{K/F}$. We use the geometric normalization of local class field theory whereby uniformizers correspond to lifts of the inverse Frobenius morphism.  We let $E$ denote the degree $2$ unramified extension of $\Q_p$. We fix for once and for all a prime $\ell \neq p$ and isomorphism $\iota_{\ell}: \ov{\Q}_{\ell} \xrightarrow{\sim} \C$. Let $p^{\frac{1}{2}} \in \C$ be the positive square root. Then $\iota_{\ell}^{-1}(p^{\frac{1}{2}})$ is a choice of square root of $p$ in $\ov{\Q}_{\ell}$. We define all half Tate twists when invoking the geometric Satake correspondence of Fargues--Scholze \cite[Chapter~VII]{FS} with respect to this choice. 

Let $\mc{E}_{\iso}$ be the local Kottwitz gerbe  (defined as $\mc{E}$ in \cite[\S3.1]{KalethaRigidvsIsoc}), which fits into an extension 
\begin{equation*}
    \D(\ov{F}) \rightarrow \mc{E}_{\iso} \rightarrow \Gamma,
\end{equation*}
where $\D$ is the pro-torus with character group $\Q$. We let $Z^1_{\alg}( \mc{E}_{\iso}, G(\ov{F}))$ denote the set of $1$-cocycles whose restriction to $\D(\ov{F})$ is the $\ov{F}$-points of an algebraic homomorphism, and let $H^1_{\alg}(\mc{E}_{\iso}, G(\ov{F}))$ be the corresponding cohomology set. We denote an extended pure inner twist of a connected reductive group $G$ over $F$ by $(G', \varrho, z)$ where $G'$ is a connected reductive group over $F$, and $\varrho: G_{\ov{F}} \to G'_{\ov{F}}$ is an isomorphism, and $z \in Z^1_{\alg}(\mc{E}_{\iso}, G(\ov{F}))$ is a cocycle such that $\varrho^{-1} \circ w(\varrho) = \Int(z_w)$. Recall that the Kottwitz set $B(G)$ is in canonical bijection with $H^1_{\alg}(\mc{E}_{\iso}, G(\ov{F}))$ and that these bijections are functorial in $G$ (see \cite[Appendix B]{KottwitzIsocrystals2}). The set $B(G)$ is determined by two maps:
\begin{itemize}
    \item The slope homomorphism
    \[ \nu: B(G) \rightarrow X_*(T_{\overline{\mathbb{Q}}_{p}})^{+,\Gamma}_{\mathbb{Q}} \quad \quad b \mapsto \nu_{b}, \]
    where $\Gamma := \mathrm{Gal}(\overline{\mathbb{Q}}_{p}/\mathbb{Q}_{p})$ and 
    $X_*(T_{\overline{\mathbb{Q}}_{p}})_{\mathbb{Q}}^{+}$ is the set of rational dominant cocharacters of $G$. 
    \item The Kottwitz map
    \[ \kappa_G: B(G) \rightarrow \pi_{1}(G)_{\Gamma} \]
    where $\pi_1(G) = X_*(T_{\bar{\mathbb{Q}}_{p}})/X_*(T_{\bar{\mathbb{Q}}_{p},sc})$ for $T_{\bar{\mathbb{Q}}_{p},sc} \subset G_{\ov{\Q}_{p}, sc}$ the maximal torus containing the pre-image of $T$. Note that $\pi_1(G)_{\Gamma}$ is canonically isomorphic to $X^*(Z(\widehat{G})^{\Gamma})$.
\end{itemize}
When the image of $\nu_b$ is central in $G$, we say that $b$ is \emph{basic} and denote by $B(G)_{\bas} \subset B(G)$ the subset of basic elements.

For $G$ a connected reductive group over a $p$-adic field $F$, we let $\Pi_{\C}(G)$ denote the set of isomorphism classes of irreducible admissible $\C$-representations of $G(F)$. We let $\Pi_{\unit}(G), \Pi_{\temp}(G), \Pi_{2, \temp}(G)$ denote the equivalence classes of unitary, tempered, and tempered essentially square-integrable $\C$-representations respectively. For $P \subset G$ a parabolic subgroup with Levi $M$ and $\sigma \in \Pi_{\C}(M)$, we denote the normalized parabolic induction by $I^G_P(\sigma)$. 

The dual group $\widehat{G}$ of $G$ is a complex Lie group equipped with a fixed splitting $(\hat{B}, \hat{T}, \{X_{\alpha}\})$ and an isomorphism $R(\hat{B}, \hat{T}) \cong R(G)^{\vee}$ between the based root datum of $\widehat{G}$ relative to our fixed splitting and the dual of the canonical based root datum of $G$. For a cocharacter $\mu \in X_{*}(T_{\ol{\mathbb{Q}}_{p}})^{+}$ whose conjugacy class has field of definition denoted $E_{\mu}$, we denote by $r_{\mu}$ the representation of $\widehat{G} \rtimes W_{E_\mu}$ whose restriction to $\widehat{G}$ is irreducible with highest weight $\mu$ (as in \cite[Lemma 2.1.2]{KottwitzTOO}).  We let $\Phi(G)$ denote the set of equivalence classes of Langlands parameters $\mc{L}_F \to \LL G$, where $\mc{L}_F = W_F \times \SL_2(\C)$ and $\LL G \cong \widehat{G} \rtimes W_F$ is the Weil form of the $L$-group of $G$. We let $\Phi_{\bdd}(G) \subset \Phi(G)$ consist of the equivalence classes $[\phi]$ such that $\phi(W_F)$ projects to a relatively compact subset of $\widehat{G}$, and set $\Phi_2(G) \subset \Phi(G)$ to be the collection of  equivalence classes $[\phi]$ such that $\phi$ does not factor through a proper Levi subgroup of $\LL G$ (i.e. $\phi$ is \emph{discrete}). We let $\Phi_{2, \bdd}(G) = \Phi_2(G) \cap \Phi_{\bdd}(G)$. The set $\Phi_{\scusp}(G)$ denotes the set of equivalence classes of supercuspidal $L$-parameters (i.e equivalence classes of discrete $\phi$ such that $\phi(\SL_2(\C))=\{1\}$). For $\phi \in \Phi(G)$, we let $\phi^{ss}$ denote its semi-simplification, as described in \S \ref{s: intro}. For $\phi \in \Phi_{\scusp}(G)$, we will often abuse notation and write $\phi$ for both the parameter and its semi-simplification, as in this case it merely corresponds to forgetting the $\SL_{2}$-factor and composing with the fixed isomorphism $\iota^{-1}_{\ell}: \C \xrightarrow{\sim} \ov{\Q_{\ell}}$. For $\phi \in \Phi(G)$, we let $S_{\phi}$ and $S^{\natural}_{\phi}$ denote the groups $Z_{\widehat{G}}(\phi)$ and $S_{\phi}/[\widehat{G}_{\der} \cap S_{\phi}]^{\circ}$, respectively. 
\section{The Local Langlands Correspondence for Unitary Groups}
\subsection{Odd Unramified Unitary Groups}{\label{ss: od unr unit groups}}
In this paper, we work with extended pure inner twists of two quasi-split groups. Let $n$ be an odd positive integer. These are $\U^*_n$ and $\GU^*_n$, the quasi-split unitary group and unitary similitude group of rank $n$ associated to the unramified quadratic extension $E/\Q_p$. These groups fit into an exact sequence $1 \to \U^*_n \to \GU^*_n \xrightarrow{c} \Gm \to 1$, and explicit presentations are given in \cite{KMSW} and \cite{BMN}. We fix the standard $\Gamma_{\Q_p}$-stable splitting of $\U^*_n$ as in \cite[pg 12]{KMSW} which extends to a $\Gamma_{\Q_p}$-stable splitting of $\GU^*_n$. We note that all extended pure inner twists of these groups are themselves quasi-split, since neither group has a non-trivial inner twist. Indeed, in both cases, the inner twists are classified by $H^1(\Q_p, \U^*_{n, \ad}) \simeq \pi_0(Z(\widehat{\U^*_{n,\ad}})^{\Gamma})^D = 1$ (since $n$ is odd).

We identify $\widehat{\U^*_n}$ and $\widehat{\GU^*_n}$ with $\GL_n(\C)$ and $\GL_n(\C) \times \C^{\times}$, respectively. Under this identification, the map $\widehat{\GU^*_n} \to \widehat{\U^*_n}$ induced by $\U^*_n \to \GU^*_n$ becomes the projection $\mathsf{P}_1 : \GL_n(\C) \times \C^{\times} \to \GL_n(\C)$ onto the first factor. We fix the splitting $(\widehat{T}, \widehat{B}, \{X_{\alpha}\})$ for $\widehat{\U^*_n}$ where $\widehat{T}$ is the diagonal torus, $\widehat{B}$ is the upper-triangular matrices, and $\{X_\alpha \}$ are the standard root vectors. Our splitting for $\widehat{\GU^*_n}$ is $(\widehat{T} \times \C^{\times}, \widehat{B} \times \C^{\times}, \{X_{\alpha}\})$.  Let $\mathrm{J} \in \GL_n(\C)$ be the anti-diagonal matrix given by $\mathrm{J}=(\mathrm{J}_{i,j})$, where $\mathrm{J}_{i,j} = (-1)^{i+1}\delta_{i, n+1 - j}$. Then the action of $W_{\Q_p}$ on $\widehat{\U^*_n}$ (resp. $\widehat{\GU^*_n}$) factors through $\Gamma_{E/\Q_p}$, and the non-trivial element $\sigma$ of this Galois group acts by $g \mapsto \mathrm{J}g^{-t}\mathrm{J}^{-1}$ (resp. $(g,z) \mapsto (\mathrm{J}g^{-t}\mathrm{J}^{-1}, \det(g)z )$).

\subsection{The Correspondence of Mok/Kaletha--Minguez--Shin--White/Bertoloni Meli--Nguyen}{\label{ss: classicalLLC}}

In this section, we recall the classical local Langlands correspondences for extended pure inner twists of $\GU^*_n$ and $\U^*_n$. These correspondences are denoted by $\LLC_{\GU^*_n}$ and $\LLC_{\U^*_n}$.

\subsubsection{The tempered correspondence}
We begin by recalling the statements of the tempered local Langlands correspondence for unitary and unitary similitude groups.

\begin{theorem}[{\cite[Theorem 1.6.1]{KMSW}, \cite[Theorem 2.5.1, Theorem 3.2.1]{Mok}, \cite[Theorem 2.16, \S 3]{BMN}}] \phantomsection \label{itm: local}

Fix an odd natural number $n$ and let $G^*$ be either $\U^*_n$ or $\GU^*_n$. Let $(G, \varrho, z)$ be an extended pure inner twist of $G^*$. Fix a non-trivial character $\psi: \Q_{p} \to \C^{\times}$. Together with our chosen splitting of $G^*$, this gives a Whittaker datum $\mf{w}$ of $G^*$. The following is true.
\begin{enumerate}
    \item For each $\phi \in \Phi_{\bdd}(G^*)$, there exists a finite subset $\Pi_{\phi}(G, \varrho) \subset \Pi_{\temp}(G)$ of tempered representations equipped with a bijection 
\[
\iota_{\mf{w}}: \Pi_{\phi} (G, \varrho) \xrightarrow{\sim} \Irr (S^{\natural}_{\phi}, \kappa_G(z)), \qquad \pi \mapsto \langle \pi, - \rangle,
\]
where $\Irr (S^{\natural}_{\phi}, \kappa_G(z))$ denotes the set of irreducible representations of $S^{\natural}_{\phi}$ restricting on $Z(\widehat{G})^{\Gamma}$ to $\kappa_G(z)$. 

\item We have 
\begin{equation*}
    \Pi_{\temp}(G) = \coprod_{\phi \in \Phi_{\bdd} (G^*)} \Pi_{\phi} (G, \varrho), \quad \quad
    \Pi_{2, \temp}(G) = \coprod_{\phi \in \Phi_{2, \bdd} (G^*)} \Pi_{\phi} (G, \varrho).
\end{equation*}
\item 
For each $\pi \in \Pi_{\temp}(G)$, the central character $\omega_{\pi} :  Z(G) \longrightarrow \C^{\times} $ has a Langlands parameter given by the composition
\[
\mc{L}_{\Q_p} \xrightarrow{\phi} \LL G  \xrightarrow{} \LL Z(G).
\]

\item Let $(\mathrm{H}, s, \Leta)$ be a refined endoscopic datum and let $\phi^{\mH} \in \Phi_{\bdd}(\mH)$ be such that $\Leta \circ \phi^{\mH} = \phi$. If $f^{\mH} \in \mathcal{H}(\mathrm{H})$ and $f \in \mathcal{H}(G)$ are two $\Delta[\mf{w}, \varrho, z]$-matching functions then we have
\begin{equation*}
 \displaystyle \sum_{\pi^{\mH} \in \Pi_{\phi^{\mH}}(\mH)} \langle \pi^{\mH}, 1 \rangle\tr(\pi^{\mH} \mid f^{\mH})   = e(G) \sum_{\pi \in \Pi_{\phi}(\U, \varrho)} \langle \pi, \Leta(s)\rangle \tr(\pi \mid f),
 \end{equation*}
 where $e( \cdot)$ is the Kottwitz sign and $\Delta[\mf{w}, \varrho, z]$ is the transfer factor normalized as in \cite{BMN}.

\end{enumerate}
\end{theorem}
\begin{proof}
For $G^*=\U^*_n$, this is essentially \cite[Theorem 1.6.1]{KMSW}. When $G^*= \GU^*_n$, part $(1)$  of the theorem is \cite[Theorem 2.16]{BMN} and part $(3)$ is implicit in the construction given in loc cit.  Part $(4)$ is proven in \cite[\S 3]{BMN}.

We now check part $(2)$ for $\GU^*_n$. For the first equality, it suffices to check that, for each $\pi \in \Pi_{\temp}(\GU_n)$, there exists a $\phi \in \Phi_{\bdd}(\GU^*_n)$ such that $\pi \in \Pi_{\phi}(\GU_n, \varrho)$. Since $n$ is odd, we have a surjection
\begin{equation*}
    \U_n(\Q_p) \times Z(\GU_n)(\Q_p) \to \GU_n(\Q_p),
\end{equation*}
and a representation $\pi \in \Pi_{\temp}(\GU_n)$ corresponds to a pair $(\pi', \chi)$ such that $\pi'$ is a representation of $\U_n(\Q_p)$ and $\chi$ is a character of $Z(\GU_n)$ such that $\omega_{\pi'} = \chi|_{Z(U_n)}$. Then, $\pi$ is tempered precisely when $\pi'$ is tempered and $\chi$ is unitary. By part (2) of the present theorem for $\U_n$, the parameter $\phi'$ of $\pi'$ lies in $\Phi_{\bdd}(\U^*_n)$. As in \cite[\S 2.3]{BMN}, the $L$-parameter $\phi$ is determined uniquely by the fact that its projection along ${}^L\GU^*_n \to {}^L\U^*_n$ is $\phi'$, and its projection along ${}^L\GU^*_n \to {}^LZ(\GU^*_n)$ is $\chi$. It therefore follows that $\phi \in \Phi_{\bdd}(\GU^*_n)$. The proof of the second equality is analogous.
\end{proof}

\subsubsection{The Langlands classification}

The correspondences described in Theorem \ref{itm: local} are between $\Pi_{\temp}(G)$ and $\Phi_{\bdd}(G^*)$, where $G$ and $G^*$ are as in the referenced theorem. However, the Langlands classification theorem gives a recipe to extend a tempered correspondence for each Levi subgroup of $G$ to a full correspondence between $\Pi_{\C}(G)$ and $\Phi(G^*)$. We briefly review this construction, following the work of Silberger--Zink \cite{SZ}.

For the rest of this subsection, we work with $G$ a connected reductive group over a $p$-adic field $F$. Choose a minimal parabolic subgroup $P_0$ of $G$ and a maximal split torus $A_0 \subset P_0$. Let $M_0$ be the Levi subgroup of $P_0$ that is the centralizer of $A_0$. For each standard parabolic subgroup $P \supset P_0$, we let $M_P$ be the unique Levi subgroup containing $M_0$, let $N_P$ be the unipotent radical, and let $A_P \subset A_0$ be the maximal split torus in the center of $M_P$. Let $\mf{a}^*_P$ denote the set $X^*_F(M_P)_{\R}$, where $X^*_F$ denotes the set of characters defined over $F$. We say that $\nu \in \mf{a}^*_P$ is \emph{regular} if its restriction to $X^*_F(A_P)_{\R}$ satisfies $\langle \nu, \alpha \rangle > 0$ (where $\langle \cdot, \cdot \rangle$ is a fixed Weyl group-invariant pairing on $X^*_F(A_P)_{\R}$, as in \cite[\S1.2]{SZ}) for each simple root $\alpha$ in the adjoint action of $A_P$ on $\Lie(N_P)$. Then we have the Langlands classification theorem:
\begin{theorem}[Langlands]{\label{thm: Langlandsclassification}}
There exists a bijection
\begin{equation*}
    \{(P, \sigma, \nu)\} \Longleftrightarrow \Pi_{\C}(G)
\end{equation*}
where $\sigma$ is a tempered representation of $M_P$ and $\nu \in \mf{a}^*_P$ is a regular character. The bijection takes $(P,\sigma, \nu)$ to the unique irreducible quotient of the normalized parabolic induction $I^G_P(\sigma \otimes \chi_{\nu})$, where $\chi_{
\nu}$ is the positive real unramified character determined by $\nu$ (see \cite[\S1.2]{SZ}).
\end{theorem}

Analogously, there is a Langlands classification for $L$-parameters. Consider the complex torus $Z(\LL M)^{\circ}$ and let $Z(\LL M)^{\circ}_{\hyp}$ denote the subset of \emph{hyperbolic} elements (i.e. those that project to $\R^{\times}_{>0}$ under each complex character). Silberger and Zink define a natural map 
\begin{equation*}
      \mf{a}^*_P \to Z(\LL M)^{\circ}_{\hyp}, \qquad \nu \mapsto z(\nu).
\end{equation*}
Given an element $z \in Z(\LL M)^{\circ}_{\hyp}$ and $\phi \in \Phi(G)$, Silberger and Zink define the twist $\phi_z$ of $\phi$ by $z$ as
\begin{equation*}
    \phi_z(w, A) = \phi(w,A)z^{d(w)},
\end{equation*}
for $(w, A) \in W_F \times \SL_2(\C)$ and where $d(w)$ is such that $|w| = q^{d(w)}$ for $q$ the cardinality of the residue field of $F$. 
\begin{theorem}[Silberger--Zink {\cite[4.6]{SZ}}]{\label{thm: langlandsclassificationparams}}
There exists a bijection
\begin{equation*}
    \{(P, {}^t\phi, \nu)\} \Longleftrightarrow \Phi(G),
\end{equation*}
where ${}^t\phi$ is a tempered $L$-parameter of $M_P$ and $\nu \in \mf{a}^*_P$ is regular. The bijection is obtained by twisting ${}^t\phi$ by $z(\nu)$ to get ${}^t\phi_{z(\nu)}$ and then post-composing with the map $\LL M \to \LL G$.
\end{theorem}
Given a tempered local Langlands correspondence of $G$ as in Theorem \ref{itm: local}, we can define a full correspondence using Theorem \ref{thm: langlandsclassificationparams} and Theorem \ref{thm: Langlandsclassification}. We first need the following lemma.
\begin{lemma}{\label{lem: BGLevilem}}
Let $G^*$ be quasi-split and let $(G, \varrho, z)$ be an extended pure inner twist with class $b \in B(G^*)$. Let $M \subset G$ be a Levi subgroup corresponding to a Levi subgroup $M^*$ of $G^*$. Then, up to equivalence, there is a unique extended pure inner twist $(M, \varrho_M, z_M)$ of $M^*$ whose class under the map $B(M^*) \to B(G^*)$ is $b$.
\end{lemma}
\begin{proof}
Let $P^*$ (resp. $P$) be an $F$-parabolic subgroup of $G^*$ (resp. $G$) with Levi subgroup $M^*$ (resp. $M$). We can assume without loss of generality that $\varrho(P^*) = P$. Choose $K$ a finite Galois extension of $F$ such that $z \in Z^1(\mc{E}_{\iso}(K/F), G(K))$. Then for each $e \in \mc{E}_{\iso}(K/F)$, we have $\varrho^{-1} \circ e(\varrho) = \Int(z_e)$ stabilizes $P^*$ and $M^*$ since $P^*, M^*, M, M$ are all defined over $\Q_p$. In particular, since $N_{G^*}(P^*) = P^*$, it follows that for each $e \in \mc{E}_{\iso}(K/F)$ we have $z_e \in P^*(K)$, and since $N_{P^*}(M^*)=M^*$, we conclude that $z_e \in M^*(K)$. Hence, $\varrho$ restricts to give an extended pure inner twist $(M, \varrho_M, z_M)$ of $M^*$ with class $b_M \in B(M^*)$ satisfying $i_{M^*}(b_M)=b$, where $i_{M^*}: B(M^*) \to B(G^*)$ is the natural map.

We now claim that $b_M$ is the unique element of $B(M^*)$ with this property. First, we note that since the image of $\nu_b$ is central in $G^*$, any element in $B(M^*)$ whose image in $B(G^*)$ equals $b$ must have Newton point equal to $\nu_b$. Now, via twisting as in \cite[\S3.4]{KottwitzIsocrystals2}, it suffices to consider the pre-image of an element $b'$ with trivial Newton point. The subsets of $B(G^*)$ and $B(M^*)$ with trivial Newton point are given by $H^1(F, G^*)$ and $H^1(F, M^*)$ respectively, and so the claim follows from the injectivity of the map $H^1(F, M^*) \hookrightarrow H^1(F, G^*)$ \cite[\S1.3.5]{Kestutistorsors}.
\end{proof}

\begin{theorem}{\label{thm: fullLLC}}
Fix an odd natural number $n$ and let $G^*$ be either $\U^*_n$ or $\GU^*_n$. Let $(G, \varrho, z)$ be an extended pure inner twist of $G^*$. Fix a non-trivial character $\psi: F \to \C^{\times}$. Together with our chosen splitting of $G^*$, this gives a Whittaker datum $\mf{w}$ of $G^*$. Then,
\begin{enumerate}
    \item For each $\phi \in \Phi(G^*)$, there exists a finite subset $\Pi_{\phi}(G, \varrho) \subset \Pi_{\C}(G)$. Our choice of $\mf{w}$ defines a bijection
\[
\iota_{\mf{w}}: \Pi_{\phi} (G, \varrho) \xrightarrow{\sim} \Irr (S^{\natural}_{\phi}, \kappa_{G^*}(z)), \qquad \pi \mapsto \langle \pi, - \rangle.
\]

\item We have
\begin{equation*}
    \Pi_{\C}(G) = \coprod_{\phi \in \Phi(G^*)} \Pi_{\phi} (G, \varrho).
\end{equation*}
\end{enumerate}
\end{theorem}
\begin{proof}
We first claim that the analogue of Theorem \ref{itm: local} is known for each Levi subgroup of $G$. Strictly speaking, that theorem is proven for $\Q_p$, but the argument works over a general $p$-adic field $F$. Then when $G^*= \U^*_n$, the claim follows from the fact that each Levi subgroup is a product of an odd unitary group and a term $\Res_{E/F} G'$ where $G'$ is a product of general linear groups. Each Levi subgroup of $\GU_n$ is an extension by $\Gm$ of a Levi subgroup of $\U_n$. Hence, Theorem \ref{itm: local} is known for Levi subgroups of $\GU_n$ by reduction to the $\U_n$ case following the arguments of \cite[\S2-3]{BMN}.

We now define the set $\Pi_{\phi}(G, \varrho)$. Given $\phi \in \Phi(G)$, we find the corresponding triple $(P, {}^t\phi, \nu)$ as in Theorem \ref{thm: langlandsclassificationparams} and define $\Pi_{\phi}(G, \varrho)$ to consist of the subset of $\Pi_{\C}(G)$ containing those $\pi$ whose triple $(P, \sigma, \nu')$ as in Theorem \ref{thm: Langlandsclassification} satisfies $\sigma \in \Pi_{{}^t\phi}(M, \varrho_M)$ and $\nu=\nu'$, where $\varrho_M$ is as in Lemma \ref{lem: BGLevilem}. Now, the assertion in $(2)$ is clear by construction. Indeed, by Theorem \ref{thm: Langlandsclassification}, each $\pi \in \Pi_{\C}(G)$ is associated to a unique triple $(P, \sigma, \nu)$ and hence $\pi \in \Pi_{\phi}(G, \varrho)$, where $\phi \in \Phi(G)$ corresponds to the unique triple $(P, {}^t\phi, \nu)$ for ${}^t\phi$ the Langlands parameter of $\sigma$.

We now prove $(1)$. We define the map $\iota_{\mf{w}}$ as follows. Given $\pi \in \Pi_{\phi}(G, \varrho)$, we let $(P, \sigma, \nu)$ be the corresponding triple as in Theorem \ref{thm: Langlandsclassification}. Then let $\chi_M \in \Irr(S^{\natural}_{{}^t\phi}, \kappa_{M^*}(z_M))$ be the representation associated to $\sigma$ via the tempered Langlands correspondence for $M$. It is clear by the definition of $\phi$ as a central twist of ${}^t\phi$ that we have an equality $Z_{\widehat{M^*}}(\phi) = Z_{\widehat{M^*}}({}^t\phi)$. Further, it follows from the discussion in \cite[\S 7]{SZ} that $Z_{\widehat{M^*}}(\phi) = Z_{\widehat{G^*}}(\phi)$. Hence, $\chi_M$ pulls back to $\chi \in \Irr(Z_{\widehat{G^*}}(\phi))$.

We claim that $\chi$ factors through the quotient $Z_{\widehat{G^*}}(\phi) \to S^{\natural}_{\phi}$, or equivalently that $\chi$ is trivial on $[(\widehat{G^*})_{\der} \cap S_{\phi}]^{\circ}$. The resulting representation, $\chi_G$, will live in $\Irr(S^{\natural}, \kappa_{G^*}(z))$ and is defined to be $\iota_{\mf{w}}(\pi)$. To prove the claim, we first argue that
\begin{equation}{\label{eqn: nastycocycleeqn}}
    [\widehat{M^*} \cap S_{\phi}]^{\circ} =  [(\widehat{M^*})_{\der}Z(\widehat{M^*})^{\Gamma_F} \cap S_{\phi}]^{\circ}.
\end{equation}
There is clearly an inclusion of the right-hand side into the left. On the other hand, given $m \in \widehat{M^*} \cap S_{\phi}$, we produce an element of $\ker(H^1(\mc{L}_F, Z(\widehat{M^*})_{\der}) \to H^1(\mc{L}_F, Z(\widehat{M^*})))$ as follows (where $Z(\widehat{M^*})_{\der} := Z(\widehat{M^*}) \cap (\widehat{M^*})_{\der}$). We write $m = m_{\der}z$ for $m_{\der} \in (\widehat{M^*})_{\der}$ and  $z \in Z(\widehat{M^*})$ (such a decomposition exists because $Z(\widehat{M^*})$ surjects onto $(\widehat{M^*})_{\ab}$). Then writing $\phi(w) = (\phi'(w), w)$ for $\phi' \in Z^1(\mc{L}_{\Q_p}, \widehat{M^*})$, we get from $m \in S_{\phi}$ that 
\begin{equation*}
    w \mapsto \phi'(w)w(m_{\der})^{-1}\phi'(w)^{-1}m_{\der} = z^{-1}w(z) \in Z(\widehat{M^*})_{\der},
\end{equation*}
is the desired cocycle. It is easy to check that changing our decomposition of $m$ into $m_{\der}z$ changes the cocycle by a coboundary. It is also easy to check this gives a group homomorphism $\widehat{M^*} \cap S_{\phi} \to \ker(H^1(\mc{L}_F, Z(\widehat{M^*})_{\der}) \to H^1(\mc{L}_F, Z(\widehat{M^*})))$. Via the long exact sequence in cohomology associated to
\begin{equation*}
    1 \to Z(\widehat{M^*})_{\der} \to Z(\widehat{M^*}) \to (\widehat{M^*})_{\ab} \to 1,
\end{equation*}
we get a map $\widehat{M^*} \cap S_{\phi} \to (\widehat{M^*})_{\ab}^{\Gamma_F} /Z(\widehat{M^*})^{\Gamma_F}$ which can be taken to be $m \mapsto [z]$, where $[z]$ is the image of $z$ in $(\widehat{M^*})_{\ab}^{\Gamma_F} /Z(\widehat{M^*})^{\Gamma_F}$. The target of this map is finite and hence $[\widehat{M^*} \cap S_{\phi}]^{\circ}$ lies in the kernel. This means that for $m \in [\widehat{M^*} \cap S_{\phi}]^{\circ}$, we have $z \in Z(\widehat{M^*})^{\Gamma_F}$. Thus we have $[\widehat{M^*} \cap S_{\phi}]^{\circ} \subset  [(\widehat{M^*})_{\der}Z(\widehat{M^*})^{\Gamma_F} \cap S_{\phi}]^{\circ}$, using again that $[\widehat{M^*} \cap S_{\phi}]^{\circ}$ is connected.

We now show that $\chi$ is trivial on $[(\widehat{G^*})_{\der} \cap S_{\phi}]^{\circ} = [((\widehat{G^*})_{\der} \cap \widehat{M^*}) \cap S_{\phi}]^{\circ}$. By \eqref{eqn: nastycocycleeqn} and since $Z(\widehat{M^*})^{\Gamma_F} \subset S_{\phi}$ and $\chi$ is already known to vanish on $[(\widehat{M^*})_{\der} \cap S_{\phi}]^{\circ}$, it suffices to show that $\chi$ vanishes on $[Z(\widehat{M^*})^{\Gamma_F} \cap \widehat{G^*}_{\der}]^{\circ}$. As in \cite[1.4.3]{KottwitzGlobal} (see also \cite[Theorem 1.15]{RapoportRichartz}), we have a diagram
\begin{equation}{\label{eqn: BMdiagram}}
    \begin{tikzcd}
            B(M^*)_{\bas} \arrow[r, "\kappa_{M^*}"] \arrow[d, swap, "Newton"] & X^*(Z(\widehat{M^*})^{\Gamma_F}) \arrow[d, "avg"] \arrow[r, "\res"] & X^*([Z(\widehat{M^*})^{\Gamma_F}  \cap (\widehat{G^*})_{\der}]^{\circ}) \arrow[d, hook] \\
            X^*(\widehat{C(M^*)})^{\Gamma_F}_{\Q} \arrow[r, hook] & X^*(Z(\widehat{M^*}))^{\Gamma_F}_{\Q} \arrow[r, "\res"]  & X^*([Z(\widehat{M^*}) \cap (\widehat{G^*})_{\der}]^{\circ})^{\Gamma_F}_{\Q},
    \end{tikzcd}
\end{equation}
where $C(M^*)$ is the maximal torus in the center of $M^*$. The map $avg$ is given as the composition of $X^*(Z(\widehat{M^*})^{\Gamma_F}) \to X^*(Z(\widehat{M^*})^{\Gamma_F})_{\Q}$ and the isomorphism $X^*(Z(\widehat{M^*})^{\Gamma_F})_{\Q} = X^*(Z(\widehat{M^*}))_{\Q, \Gamma_F} \cong X^*(Z(\widehat{M^*}))^{\Gamma_F}_{\Q}$ given by averaging over the $\Gamma_F$ action. Hence, it suffices to show that if $\omega_{\chi} \in X^*(Z(\widehat{M^*})^{\Gamma_F})$ equals the central character of $\chi$ on $Z(\widehat{M^*})^{\Gamma_F}$, then $avg(\omega_{\chi})$ vanishes under the restriction map to $X^*([Z(\widehat{M^*}) \cap \widehat{G^*}_{\der}]^{\circ})^{\Gamma_F}_{\Q}$. But by Theorem \ref{itm: local} for $M$, we have that $\omega_{\chi}$ corresponds to some $b_M \in B(M^*)_{\bas}$ equal to the class of $(M, \varrho_M, z_M)$. In particular, the image of $b_M$ in $B(G^*)$ is basic by our construction of $b_M$ using Lemma \ref{lem: BGLevilem}, and so $\nu_{b_M}$ factors through $C(G^*) \to C(M^*)$. Equivalently, $\nu_{b_M} \in X^*(\widehat{C(M^*)})_{\Q}$  factors through $\widehat{C(M^*)} \to \widehat{C(G^*)}$ and hence is trivial on the kernel, $\ker$, of $\widehat{C(M^*)} \to \widehat{C(G^*)}$. Note that $\widehat{C(M^*)}=\widehat{M^*}_{\ab}, \widehat{C(G^*)}=\widehat{G^*}_{\ab}$, and we have a diagram
\begin{equation*}
\begin{tikzcd}
             \ker \arrow[r, hook] &(\widehat{M^*})_{\ab} \arrow[r, two heads] & (\widehat{G^*})_{\ab} \\
           (\widehat{G^*})_{\der} \cap Z(\widehat{M^*}) \arrow[u, two heads] \arrow[r, hook] & Z(\widehat{M^*}) \arrow[u, two heads], & 
\end{tikzcd}    
\end{equation*}
which implies the image in $X^*(Z(\widehat{M^*}))_{\Q}$ is trivial on $[Z(\widehat{M^*}) \cap (\widehat{G^*})_{\der}]^{\circ}$ as desired. We have now proven that $\chi$ yields a $\chi_G \in \Irr(S^{\natural}_{\phi}, \kappa_G(z))$.

We have constructed a map $\Irr(S^{\natural}_{{}^t\phi} , \kappa_{M^*}(z_M)) \to \Irr(S^{\natural}_{\phi}, \kappa_{G^*}(z))$. It remains to show these sets are in bijection, and to do so we construct an inverse. Given a $\chi_G \in \Irr(S^{\natural}_{\phi}, \kappa_{G^*}(z))$ we can lift to $\chi \in \Irr(S_{\phi})$ and project to $ \chi_M \in \Irr(S^{\natural}_{{}^t\phi})$ since $[\widehat{M}_{\der} \cap S_{{}^t\phi}]^{\circ} \subset [\widehat{G}_{\der} \cap S_{\phi}]^{\circ}$. Again, let $\omega_{\chi} \in X^*(Z(\widehat{M^*})^{\Gamma_F})$ be the central character of $\chi$ on $Z(\widehat{M^*})^{\Gamma_F}$. Then by definition, $\omega_{\chi}|_{[Z(\widehat{M^*})^{\Gamma_F} \cap (\widehat{G^*})_{\der}]^{\circ}}$ is trivial and so by Diagram \eqref{eqn: BMdiagram}, the restriction of $avg(\omega_{\chi})$ to $X^*([Z(\widehat{M^*}) \cap (\widehat{G^*})_{\der}]^{\circ})^{\Gamma_F}_{\Q}$ is trivial. Hence, if $b'_M \in B(M^*)$ is the basic element satisfying $\kappa_{M^*}(b'_M)= \omega_{\chi}$, then the image of $\nu_{b'_M}$ factors through $C(G^*)$. It follows that the image of $b'_M$ in $B(G^*)$ equals $b$, and hence by uniqueness, we have that $b'_M=b_M$ and so $ \chi_M \in \Irr(S^{\natural}_{{}^t\phi}, \kappa_{M^*}(z_M))$. This completes the construction of the inverse map and concludes the proof. 
\end{proof}

\subsubsection{Compatibility with parabolic induction} \phantomsection \label{itm : theta}

In this subsection, we recall the definition of extended cuspidal support of a discrete series representation of a unitary group as well as the relation between the Langlands parameter and the extended cuspidal support. Then we deduce the compatibility of the parameters with parabolic induction. All the materials here can be found in \cite{Moe}.

We return to the notation of \S \ref{ss: od unr unit groups}.  We denote by $\theta$ the outer automorphism of $\GL_n(E)$ that sends $g$ to $\J^t \overline{g}^{-1} \J^{-1}$ where $\J = (\J_{i,j})$ is the antidiagonal matrix with $\J_{i, n+1-i} = (-1)^{i+1}$. We consider the semi-direct product of $\GL_n(E)$ with $\{ 1,\theta \}$ and let $\Tilde{\G}_n$ be the component of $\theta$ in this group. 

We first recall the notion of extended supercuspidal support. Suppose $\tau \in \Pi_{\C}(\U^*_n)$ is supercuspidal. Then for each natural number $m$ and $\theta$-invariant supercuspidal $\rho \in \Pi_{\C}(\GL_m(E))$, there is at most one $x \in \R_{>0}$ so that the normalized parabolic induction $I^{\U^*}_{\mP} (\rho||^x \otimes \tau)$ is reducible (here $\U^* = \U^*_{n + 2m}$ and $\mP$ is the standard parabolic whose Levi factor is $\M = \U^*_{n} \times \Res_{E/\Q_p} \GL_{m}$). Any such $x$ will be a half-integer, i.e there is a positive integer $a_{\rho, \tau}$ such that $ x = (a_{\rho, \tau}+1)/2 $. 

In fact, there are only finitely many $\rho$ such that the normalized parabolic induction $I^{\U^*}_{\mP} (\rho||^x \otimes \tau)$ is reducible and we can define a representation $\pi_{\tau}$ of a certain general linear group $G_{S_{\tau}}$ by
\[
\pi_{\tau} := I^{\G_{S_{\tau}}(E)}_{\mP_{S_{\tau}}(E)} \bigotimes_{(\rho, b) \in S_{\tau} } \St(\rho, b),
\]
where $S_{\tau}$ is the set of pairs $(\rho, b)$ such that the associated number $a_{\rho, \tau}$ is strictly positive and $b$ is an integer satisfying $ 0 \leq b \leq a_{\rho, \tau} $ and $ b \equiv a_{\rho, \tau} $ modulo $2$ and $\St(\rho, b)$ is the Steinberg representation. The group $\mP_{S_{\tau}} \subset \G_{S_{\tau}}$ is a parabolic subgroup.

\begin{lemma}[{\cite[\S5.3]{Moe}}]
The representation $\pi_{\tau}$ is a representation of $\GL_n(E)$. In particular $G_{S_{\tau}} = \GL_n$.
\end{lemma}
\begin{definition}
The \emph{extended cuspidal support} $\SuppCusp (\tau)$ of $\tau$ is the cuspidal support of $\pi_{\tau}$.
\end{definition}

We now suppose that $\tau$ is an irreducible discrete series representation of $\U^*_n$. Then there exists a supercuspidal representation $\tau_0$ of $\U^*_{n_0}$ with $n_0 \leq n$ and a set $\mathcal{J}$ of triples $(\rho, a, b)$ such that $\rho$ is a $\theta$-invariant supercuspidal representation of a general linear group and $a > b$ are half-integers such that $a - b \in \Z$ (by \cite[sections. 4.1, 5.1]{Moe}) and $\tau$ is a subquotient of the parabolic induction
\[
I_{\mP(\mathbb{Q}_{p})}^{\U^*_n(\mathbb{Q}_{p})}\big( \bigotimes_{(\rho, a, b) \in \mathcal{J}} \langle \rho ||^{a}_E, \rho ||^{b}_E \rangle \otimes \tau_{0} \big),
\]
where $\langle \rho ||^{a}_E, \rho ||^{b}_E \rangle$ is the Steinberg representation $\St(\rho, a-b+1) ||^{(a+b)/2} $.

\begin{definition}
The extended cuspidal support $\SuppCusp (\tau)$ of $\tau$ is the following

\[
\SuppCusp (\tau) := \cup_{(\rho, a, b) \in \mathcal{J}} \cup_{y \in \{b, b+1, \dotsc, a\}} \{ \rho ||^{y}_E, \rho ||^{-y}_E \} \cup \SuppCusp (\tau_0).
\]
\end{definition}
This definition depends only on the cuspidal support of $\tau$ and is independent of the choice of $\mc{J}$. Indeed, if $\tau$ is a sub-quotient of a parabolic induction from some supercuspidal representation of the form $ \otimes_{\lambda \in \mathcal{S}} \lambda \otimes \tau_0 $ where $ \mathcal{S}$ is a set of supercuspidal representation of general linear groups and $\tau_0$ is a supercuspidal representation of unitary group then $ \SuppCusp (\tau) = \{ \lambda, \theta(\lambda) \ | \ \lambda \in \mathcal{S} \} \cup \SuppCusp (\tau_0)$. We summarize this as a Proposition for future use.
\begin{proposition}
    There is a unique tempered representation $ \pi_{\tau} $ of $\GL_n(E)$ such that the cuspidal support of $ \pi_{\tau} $ is exactly $\SuppCusp (\tau)$. Moreover, $\pi_{\tau}$ is $\theta$-discrete. Namely, its isomorphism class is invariant under the action of $\theta$ and it is not a subquotient of the parabolic induction of a $\theta$-stable representation of a $\theta$-stable parabolic subgroup.
\end{proposition}

Let $\phi : \mc{L}_E \ra \GL_n(\C)$ be an L-parameter. It decomposes into a sum of irreducible representations $ \phi = \bigoplus (\rho \otimes \sigma_a)^{\oplus m_{\rho , a}} $, where $\sigma_a$ is the unique representation of dimension $a$ of $\SL_2(\C)$. Then $\phi$ is $\theta$-discrete if $m_{\rho , a}$ is at most $1$ and if $m_{\rho , a} = 1$ then $\rho$ is invariant under the composition of $ g \longmapsto {}^{t}g^{-1} $ and the conjugation in $W_E$ by the action of the non-trivial element in $\Gamma_{E/\Q_p}$.  

Recall that Moeglin constructed a stable packet consisting of discrete series representations of $\U^*_n$ from a conjugacy class of $\theta$-discrete and $\theta$-stable morphisms of $W_E \times \SL_2(\C)$ in $\GL_n(\C)$. More precisely, let $\phi$ be a $\theta$-discrete and $\theta$-stable morphism. Then the local Langlands correspondence for $\GL_n$ associates with $ \phi $ a $\theta$-discrete representation $\pi(\phi)$ of $\GL_n(E)$. The corresponding packet $\Pi^{\mathrm{Mo}}_{\phi}$ is the set of discrete series representations $\tau$ such that the extended cuspidal support of $\tau$ is the same as the cuspidal support of $\pi(\phi)$.

There is another characterization of the discrete series representations $\tau$ belonging to $\Pi^{\mathrm{Mo}}_{\phi}$. Since $\phi$ is conjugate-orthogonal, by \cite[Theorem 8.1 (ii)]{GGP}, it is the restriction to $W_E$ of an $L$-parameter $\phi_{\U^*_n}$ of $\U^*_n$. The following result is well known for the experts but for completeness, we also give a proof.

\begin{proposition}{\label{Moeglinpacket}}
The packet $\Pi^{\mathrm{Mo}}_{\phi}$ coincides with the packet corresponding to $\phi_{U^*_n}$ constructed in \cite{Mok}.
\end{proposition}
\begin{proof}
Denote $ \I_{cusp} (\G) $ the space of orbital integrals of cuspidal functions of a reductive group $\G$ over $\Qp$. Recall that a function $f$ is cuspidal if for each proper Levi subgroup $\mathrm{L} \subset \G$ with parabolic $\mathrm{Q}$, we have that the constant term $f_{\mathrm{Q}}$ satisfies $\tr(\pi_{\mathrm{L}} \mid f_{\mathrm{Q}})=0$ for each $\pi_{\mathrm{L}} \in \Pi_{\temp}(\mathrm{L})$. Following \cite[\S6]{Art93} and \cite[\S1]{Art96}, one can define a convenient inner product by
\[
(a_G, b_G) := \int_{\Gamma_{\mathrm{ell}}(G)} a_G(\gamma) \overline{b_G(\gamma)} d\gamma,
\]
for $a_G, b_G \in \I_{cusp}(\G)$ and $\Gamma_{\mathrm{ell}}(G)$ is the set of elliptic, strongly regular, semi-simple conjugacy class in $G(\Q_p)$. The set $T_{\mathrm{ell}}(\G)$ parametrizes an orthonormal basis of $\I_{cusp}(\G)$. More precisely, we consider the triples $ \omega = (M, \sigma, r)$ where $M$ is a Levi component of some parabolic subgroup $P$ of $G$, $\sigma$ is a square integrable representation of $M(\Q_p)$, and $r$ is an element in the extended $R$-group $\widetilde{R}_{\sigma}$ of $\sigma$. We can associate a distribution to each triple by the formula
\[
\Phi(\omega, f) = \tr(\widetilde{R}_P(r, \sigma)\mathcal{I}_P(\sigma, f)) \quad f \in \mathcal{H}(G(\Q_p)),  
\]
where $\mathcal{I}_P$ is the parabolic induction. We consider only essential triples. Roughly speaking they are the ones whose corresponding distribution is non-zero. These triples have an action of some group $W_0^G$ and we denote $T(\G)$ the set of orbits of essential triples. Then $T_{\mathrm{ell}}(\G)$ is the set of $\omega$ in $T(\G)$ such that the distribution $\Phi(\omega)$ does not vanish on $G(\Q_p)_{\mathrm{ell}}$, the set of elliptic elements of $G(\Q_{p})$. The virtual characters corresponding to these distributions form an orthogonal basis of $\I_{cusp}(\G)$. In particular, each square integrable representation $\pi$ of $G(\Q_p)$ belongs to $T_{\mathrm{ell}}(\G)$ (here we take $M = G$ and $\pi = \sigma$) and the corresponding virtual character is in fact the Harish-Chandra character $\theta_{\pi}$. 

We also need a twisted version of some of the above notations. We can defined a space $ \I_{cusp}(\Tilde{\G}_n) $ as in \cite[\S 1.1]{Moe}. For completeness, we give a brief details of the definition. In fact, one can define the subsets $\Tilde{\G}_{\text{reg}}$, $\Tilde{\G}_{\text{n,reg}}$, $\Tilde{\G}_{\text{ell}}$ respectively $\Tilde{\G}_{\text{n,ell}}$ of regular, strongly regular, elliptic and respectively strongly elliptic elements of $\Tilde{\G}$. Let $f$ be a locally constant and compact support function on $\Tilde{\G}$, one can define the orbital integral $J^{\GL_n(E)}(g, f)$ for $g \in \Tilde{\G}_{\text{reg}}$. We say that $f$ is a cuspidal function if the integral $J^{\GL_n(E)}(g, f)$ vanishes for all $ g \in \Tilde{\G}_{\text{n,reg}} \setminus \Tilde{\G}_{\text{n,ell}} $. We denote by $C_{cusp}(\Tilde{\G})$ the set of cuspidal functions and by $\I_{cusp}(\Tilde{\G})$ the vector space of maps
\begin{align*}
  I(f) : \Tilde{\G}_{\text{n,ell}} &\longrightarrow \quad \C \\
                             g \ \     &\longmapsto J^{\GL_n(E)}(g, f),  
\end{align*}
for $f \in C_{cusp}(\Tilde{\G})$.

Let $ \pi^{+}(\phi) $ be an extension of $\pi(\phi)$ to $ \GL_n(E) \rtimes \{ 1, \theta \} $ and $\Tilde{\pi}(\phi)$ be its restriction to $\Tilde{\G}_n$. By \cite[Theorem 1.2]{Moe}, there exists a pseudo-coefficient $f$ of $\Tilde{\pi}(\phi)$ and we denote by $ f_{\phi} $ its projection in $ \I_{cusp}(\Tilde{\G}_n) $. We have a decomposition $ \I_{cusp}(\Tilde{\G}_n) = \bigoplus_{D(n)} \I^{st}_{cusp} (\U^*_{n_1} \times \U^*_{n_2} ) $ where $D(n)$ is the set of ordered pairs $(n_1, n_2)$ whose sum is $n$. In fact, by \cite[\S 1.4]{Moe} and \cite[Proposition 3.5]{Art96}, the map that sends $f$ to the sum of its endoscopic transfers gives an isometry between them. 

Since $\U^*_n$ can be considered as a twisted endoscopic group of $ \Tilde{\G}_n $, there exists $ f^{\U^*_n}_{\phi} \in \I^{st}_{cusp} (\U^*_n) $ a stable transfer of $f_{\phi}$. We need to determine which component in the decomposition $ \I_{cusp}(\Tilde{\G}_n) = \bigoplus_{D(n)} \I^{st}_{cusp} (\U^*_{n_1} \times \U^*_{n_2} ) $ does $f^{\U^*_n}_{\phi}$ belong to. By \cite[Theorem 6.1]{Moe}, $f^{\U^*_n}_{\phi}$ belongs exactly to the component $(n, 0)$ or to the component $(0, n)$ since both pairs $(n, 0)$ and $(0, n)$ give us the group $ \U^*_n $. 
However, $f^{\U^*_n}_{\phi}$ is a stable transfer, we deduce that it is in the component corresponding to the ordered pair $(n, 0)$ by the assertion in the last paragraph of page $3$ in \cite{Moe}. Thus, $f^{\U^*_n}_{\phi}$ corresponds to a stable distribution of $\U^*_n(\Q_p)$ and since the map is isometric, we have an equality
\[
\tr(\Tilde{\pi}(\phi)(h)) = \int_{\Gamma_{\mathrm{ell}}(\Tilde{\G}_n)} f_{\phi} (\gamma) h (\gamma) d\gamma = \int_{\Gamma_{\mathrm{ell}}(\U^*_n)} f^{\U^*_n}_{\phi} (\gamma) h^{\U^*_n} (\gamma) d\gamma, \qquad h \in \mathcal{H}(\Tilde{\G}_n).
\]
Moreover, by \cite[\S 5.5]{Moe}, the packet $ \Pi^{\mathrm{Mo}}_{\phi} $ consists of $\tau$ which are in the decomposition of this stable distribution with respect to the above basis. By \cite[Theorem~3.2.1]{Mok}, applied to the pair  $(\Tilde{G}_n, \pi(\phi))$ we have
\[
\tr(\Tilde{\pi}(\phi)(h)) = \sum_{\pi \in \Pi_{\phi}(\U^*_n)} \tr (\pi(h^{\U^*_n})), \qquad h \in \mathcal{H}(\Tilde{\G}_n),
\]
where the sum is over the packet $\Pi_{\phi}(\U^*_n)$ constructed in loc.cit. Indeed, part (a) of Mok's theorem tells us that the left hand side equals $\Tilde{h}^{(\U_n^*)}(\phi_{\U_n^*})$ (here we follow the notation in \cite{Mok}) and part (b) tells us that this quantity equals the right hand side.

The distributions $\tr (\pi(h)), h \in \mathcal{H}(\U^*_n)$ (for $\pi \in \Pi_{\phi}(\U^*_n)$) are independent and they are part of the basic of $\I_{\text{cusp}}(\U_n^*)$ indexed by $T_{\text{ell}}(\U_n^*)$ explained above, thus the packet $\Pi^{\mathrm{Mo}}_{\phi}$ coincides with the packet corresponding to $ \phi $ constructed in \cite{Mok}. 
\end{proof}

\begin{proposition}{\label{parindcomp}}
Let $n$ be an odd integer and $G$ be the quasi-split unitary group $\U^*_n$ or similitude unitary group $\GU^*_n$. Then $\LLC_{G}$ is compatible with parabolic induction. Namely, if $P \subset G$ is a proper parabolic subgroup with Levi factor $M$ and $\pi_{M} \in \Pi_{\C}(M)$ is an irreducible admissible representation of $M$ with (semi-simplified) $L$-parameter $\phi_{M}^{ss}: W_{\mathbb{Q}_{p}} \rightarrow \phantom{}^{L}M(\ov{\Q}_{\ell})$ and $\pi$ is an irreducible constituent of the (normalized) parabolic induction $I_{P(\mathbb{Q}_{p})}^{G(\mathbb{Q}_{p})}(\pi_{M})$ then the (semi-simplified) $L$-parameter of $\pi$ is given by the embedding
\[ \phi_G^{ss} : W_{\mathbb{Q}_{p}} \xrightarrow{\phi_M^{ss}} \phantom{}^{L}M(\ov{\Q}_{\ell}) \rightarrow \phantom{}^{L}G(\ov{\Q}_{\ell}). \]
\end{proposition}

\begin{remark}
We note that the statement is well-defined since any Levi subgroup of $M \subset G$ is a product of an odd unitary similitude group and general linear groups. 
\end{remark}

\begin{proof}
It is enough to prove the result for $G=\U^*_n$ since an irreducible representation of $\GU^*_n$ (resp. its $L$-parameter) is completely determined by its central character and  restriction to the corresponding unitary group (resp. the $L$-parameter corresponding to the central character and projection to the corresponding unitary group). 

Let $\pi \in \Pi_{\C}(G)$ and $(M, \pi_M)$ be its cuspidal support. We claim that it is enough to prove compatibility with parabolic induction of semi-simplified $L$-parameters for all such triples $(\pi, \pi_M, M)$. Indeed, suppose we have shown this. We first note that this implies the same statement for any Levi subgroup of $G$, since compatibility with parabolic induction is also known for general linear groups. Then, let $(\pi, \pi_{M'}, M')$ be any triple such that $\pi$ is a subquotient of $I_{P'(\mathbb{Q}_{p})}^{G(\mathbb{Q}_{p})}  \pi_{M'}$. The cuspidal support $(M_1, \pi_{M_1})$ of $\pi_{M'}$ is also the cuspidal support of $\pi$ by transitivity of parabolic induction. By assumption, we have 
\begin{itemize}
    \item an equality between $\phi^{ss}_{\pi_{M'}}$ and the composition of $\phi^{ss}_{\pi_{M_1}}$ and ${}^LM_1(\ov{\Q}_{\ell}) \to {}^LM'(\ov{\Q}_{\ell})$,
    \item an equality between $\phi^{ss}_{\pi}$ and the composition of $\phi^{ss}_{\pi_{M_1}}$ and ${}^LM_1(\ov{\Q}_{\ell}) \to {}^LG(\ov{\Q}_{\ell})$.
\end{itemize}
This clearly implies an equality between $\phi^{ss}_{\pi}$ and the composition of $\phi^{ss}_{\pi_{M'}}$ and ${}^LM'(\ov{\Q}_{\ell}) \to {}^LG(\ov{\Q}_{\ell})$, as desired.

We now proceed by proving the claim when $\pi \in \Pi_2(G)$ and then deduce the tempered case and finally the general case. If $\pi \in \Pi_2(G)$, there exists a supercuspidal representation $\pi_0$ of $\U^*_{n_0}$ with $n_0 \leq n$ and a set $\mathcal{J}$ of triples $(\rho, a, b)$ such that $\rho$ is a $\theta$-invariant supercuspidal representation of $\GL_{n_{\rho}}(E)$ and $a > b$ are demi-integers and $\pi$ is a subquotient of the parabolic induction $I_{P(\mathbb{Q}_{p})}^{G(\mathbb{Q}_{p})}\pi_M$,
where $ \pi_M := \otimes_{(\rho, a, b) \in \mathcal{J}} \otimes_{y \in [a, b]} \rho ||^y \otimes \pi_{0}$ and $M = \times_{(\rho, a, b) \in \mathcal{J}} \Res^E_{\mathbb{Q}_{p}} \GL^{\otimes(a-b+1)}_{n_{\rho}} \times \U^*_{n_0} $. Moreover, $(M, \pi_M)$ is the cuspidal support of $\pi$. 

Denote by $\tau_G$ the irreducible representation of $\GL_n(E)$ whose cuspidal support is $\SuppCusp (\pi)$. Then by Proposition \ref{Moeglinpacket}, the base change $ \phi_{G, E}^{ss} $ of the semi-simplified $L$-parameter $ \phi_G^{ss} $ of $\pi$ corresponds to $\tau_G$ via the local Langlands correspondence for the group $\GL_n(E)$. Remark that $\phi_{\G, E}^{ss} = \phi_{\G | W_E }^{ss} $ and we can identify $\G(\Q_p)$ as the subset of $\GL_n(E)$ preserving some Hermitian form. Thus by \cite[Theorem~8.1 (ii)]{GGP}, $\phi_{\G, E}^{ss}$ uniquely determines $\phi_{\G}^{ss}$ by the formula $ \phi_{\G}^{ss} = \Ind_{W_E}^{W_{\Q_p}} \phi_{\G, E}^{ss} $. 

Denote by $\tau_M$ the irreducible representation of $M(E)$ whose cuspidal support is $\SuppCusp (\pi)$. Then by Proposition \ref{Moeglinpacket}, the base change $\phi_{M, E}^{ss}$ of the semi-simplified $L$-parameter $ \phi_M^{ss} $ of $ \pi_M $ corresponds to $\tau_M$ via the local Langlands correspondence for $M(E)$. Similarly, we have the identity $\phi_{\M, E}^{ss} = \phi_{\M | W_E }^{ss} $ and $ \phi_{\M}^{ss} = \Ind_{W_E}^{W_{\Q_p}} \phi_{\M, E}^{ss} $. Therefore it is enough to prove compatibility with parabolic induction for the parameters $\phi_{G, E}^{ss}$ and $\phi_{M, E}^{ss}$.
 
 By definition $\tau_G$ and $\tau_M$ have the same cuspidal support. By the compatibility of semi-simplified $L$-parameters with parabolic induction for $\GL_n$, we deduce that $ \phi_{G, E}^{ss} $ and $ \phi_{M, E}^{ss} $ are compatible with parabolic induction. This is what we want to prove in this case. 

Consider now a tempered representation $\pi$ which is not a discrete series. Denote $\phi_G$ its $L$-parameter. Then by  \cite[Proposition~3.4.4]{Mok} and its proof, there is a proper Levi subgroup $M$ of $G$ and a discrete $L$-parameter $\phi_M$ of $M$ such that we have the following identification : 
\begin{equation} \label{*}
 \phi_G : W_{\Q_p} \xrightarrow{ \phi_M } \phantom{}^{L}M(\ov{\Q}_{\ell}) \rightarrow \phantom{}^{L}\G(\ov{\Q}_{\ell}).   
\end{equation}

Moreover, we have that $\Pi_G(\phi_G)$ equals the set of irreducible subquotients of $I^G_P(\pi_M)$ ranging over all $\pi_M \in \Pi_M(\phi_M)$. We can suppose that $\pi$ is a subquotient of the parabolic induction of some discrete representation $\pi_M$ in the packet $\Pi_M(\phi_M)$. Therefore the cuspidal support $ ( M_1, \pi_{M_1}) $ of $\pi_M$ is also the cuspidal support of $\pi$.

Since the Levi subgroup $M$ is a product of a unitary group $\U^*_{n_0}$ for some $n_0$ smaller than $n$ and some groups of the form $\Res_{E/\Q_p} \GL_m$, our claim is also true for the discrete representations $\pi_M$ and parabolic subgroups of $M$ by what we have proved for discrete representations of unitary groups. Thus we can identify $\phi_M^{ss}$ with the composition of $ \phi_{M_1}^{ss} $ and the inclusion $\phantom{}^{L}M_1(\ov{\Q}_{\ell}) \rightarrow \phantom{}^{L}M(\ov{\Q}_{\ell})$. 

Combining with (\ref{*}), we see that $\phi_G^{ss}$ is the composition of $\phi_{M_1}^{ss}$ and the inclusion $\phantom{}^{L}M_1(\ov{\Q}_{\ell}) \rightarrow \phantom{}^{L}\G(\ov{\Q}_{\ell})$ as desired.

Finally, we consider the case where $\pi$ is a general irreducible admissible representation whose corresponding $L$-parameter is $\phi_G$. By Theorem \ref{thm: Langlandsclassification}, there exists a triple $(P, \sigma, \nu)$ such that $\pi$ is the unique irreducible quotient of the normalized parabolic induction $I^G_P(\sigma \otimes \chi_{\nu})$, where $P$ is a parabolic subgroup with Levi factor $M$ and $\sigma$ is a tempered representation of $M$ as well as $\chi_{\nu}$ is the positive real unramified character determined by $\nu$. Denote by ${}^t\phi$ the $L$-parameter corresponding to $\sigma$. Then by Theorem \ref{thm: fullLLC}, the $L$-parameter of $\pi$ is the one corresponding to $(P, {}^t\phi , \nu)$ given by Theorem \ref{thm: langlandsclassificationparams}. Thus we have $ \phi_G : W_{\Q_p} \xrightarrow{ {}^t\phi_{z(\nu)} } \phantom{}^{L}M(\ov{\Q}_{\ell}) \rightarrow \phantom{}^{L}\G(\ov{\Q}_{\ell})$. Therefore compatibility with parabolic induction is true for $(\pi, M, \sigma \otimes \nu )$. By using the argument with cuspidal support as above, we see that the claim holds true when $\pi$ is an irreducible admissible representation. This completes the proof of the proposition.
\end{proof}
\subsubsection{Galois representation computations}{\label{sss: galoisrepcomp}}
In this subsection, we make a number of computations that are used in the rest of the paper. We begin by introducing the following notations.

Let $\phi \in \Phi(\GU^*_n)$ be an $L$-parameter. We consider the representation $r_{\mu} \circ \phi|_{\mc{L}_E}$ for the cocharacter $\mu$ of $\GU^*_n$ with weights $((1,0, \dots, 0),1)$. Explicitly, this equals the restriction of $\phi$ to $\mc{L}_{E}$ composed with the map $(\GL_{n}(\mathbb{C}) \times \mathbb{C}^{\times}) \times W_{E}  \ra \GL_{n}(\mathbb{C})$ given by $(g,\lambda, w) \mapsto g\lambda$.
We also define \[ \overline{\phi}: \mc{L}_{\Q_p} \xrightarrow{\phi}  \phantom{}^{L}\GU^*_{n}(\mathbb{C}) \rightarrow \phantom{}^{L}\U^*_{n}(\mathbb{C}) = \GL_{n}(\mathbb{C}) \rtimes W_{\mathbb{Q}_{p}}, \]
via the natural projection, and $\overline{\phi}_{E}$ to be the representation obtained by restricting $\overline{\phi}$ to $\mc{L}_{E}$ and projecting to $\GL_{n}(\mathbb{C})$. We write $\chi_{\phi}$ for the character given by restricting $\phi$ to $\mc{L}_E$ and projecting to the $\C^{\times}$ corresponding to the similitude factor. Hence we have $r_{\mu} \circ \phi|_{\mc{L}_E} = \ov{\phi}_E \otimes \chi_{\phi}$. 

We set $S_{\phi} := Z_{\widehat{\GU_n}}(\phi)$ and $S_{\overline{\phi}} = Z_{\widehat{\U_n}}(\ov{\phi})$. We will write 
\[ \overline{\phi}_{E} = \bigoplus_{i = 1}^{r} \ov{\phi}_{i}^{\oplus n_{i}}, \quad \quad  r_{\mu} \circ \phi|_{\mathcal{L}_{E}} = \bigoplus_{i = 1}^{r} \phi_{i}^{\oplus n_{i}} \]
as a direct sum of pairwise distinct irreducible $\mathcal{L}_{E}$-representations $\ov{\phi}_{i}, \phi_i $ occurring with positive multiplicity $n_{i}$, and where $\phi_i = \ov{\phi}_i \otimes \chi_{\phi}$.

We now assume further that $\phi \in \Phi_2(\GU_n)$. It follows by \cite[Pages~62,63]{KMSW} and \cite[Lemma~10.3.1]{KottwitzCuspidal}, that we have each $n_i=1$ and
\[ S^{\natural}_{\ov{\phi}} = S_{\overline{\phi}} = \bigoplus_{i = 1}^{r} \mathrm{O}(1,\mathbb{C}) \simeq (\mathbb{Z}/2\mathbb{Z})^{r}. \]
The center $Z(\hat{\U_{n}})^{\Gamma} = \{\pm I_{n}\} \simeq \mathbb{Z}/2\mathbb{Z}$ embeds diagonally into $S^{\natural}_{\ov{\phi}}$.
We have

\[ S^{\natural}_{\phi} = S_{\phi} = \{g \in \GL_n(\C) | \det(g) = 1, g \in S_{\ov{\phi}}\} \times \C^{\times} \simeq (\mathbb{Z}/2\mathbb{Z})^{r - 1} \times \mathbb{C}^{\times} \]
and the center $Z(\widehat{\GU_{n}})^{\Gamma} = I_{n} \times \mathbb{C}^{\times}$ embeds in the obvious way.

We now give a combinatorial description of the character groups $X^*(S_{\phi})$ and $X^*(S_{\ov{\phi}})$ that will be used in the explicit computations in \S \ref{s: applications}. We let $\tau_i \in X^*(S_{\ov{\phi}})$ be the sign character on the $i$th $\Z/2\Z$ factor and trivial on the other factors. Then $X^*(S_{\ov{\phi}})$ is generated as a $\Z$-module by $\{\tau_i\}$ up to the obvious relation that $\tau^{\otimes 2}_i$ is trivial. There is a natural map $S_{\phi} \to S_{\ov{\phi}}$ which induces a map $X^*(S_{\ov{\phi}}) \to X^*(S_{\phi})$. We denote the image of $\tau_i$ under this map by $\ttau_i$ and we let $\widehat{c}$ denote the character of $S_{\phi}$ that is trivial on the $(\Z/2 \Z)^{r-1}$ part and is the identity on $\C^{\times}$. Then $X^*(S_{\phi})$ is generated by $\widehat{c}$ and the $\ttau_i$ with one additional relation that the determinant character of $\widehat{\U_n}$ is trivial on $S_{\phi}$. More concretely, if $\phi_i$ has dimension $d_i$ as an irreducible representation, then $\det := \prod\limits_i \tau^{d_i}_i = \prod\limits_{i: ~ d_i \text{ odd }} \tau_i $ is the determinant character of $\widehat{\U_n}$ restricted to $S_{\ov{\phi}}$ and hence we have $\tilde{\det} := \prod\limits_{i:  ~ d_i \text{ odd }} \ttau_i$ is trivial in $X^*(S_{\phi})$.

The set $X^*(S_{\ov{\phi}})$ is naturally identified with subsets $I \subset \{1, \dots, r\}$ by $I \mapsto \prod\limits_{i \in I} \tau_i =: \tau_I$. 

We say that two subsets $I_1, I_2$ of $\{1, ..., r\}$ are equivalent if either $I_1 = I_2$ or we have
\begin{itemize}
    \item $I_1 \cap \{i: d_i \text{ even}\} = I_2 \cap \{i: d_i \text{ even}\}$
    \item $I_1 \cap \{i: d_i \text{ odd}\} \coprod I_2 \cap \{i: d_i \text{ odd}\} = \{i: d_i \text{ odd}\}.$
\end{itemize}
Then the set $X^*(S_{\phi})$ is in bijection with pairs $([I], m)$ such that $[I]$ is an equivalence class under the above relation and $m \in \Z$ by the map $([I], m) \mapsto \widehat{c}^m \otimes \prod\limits_{i \in I} \ttau_i =: \ttau_{[I],m}$. For ease of notation, we denote $\ttau_{[\{i\}], m}$ by $\ttau_{i,m}$. The symmetric difference operation on subsets of $\{1, \dots, r\}$ (denoted by $\oplus$) gives a group structure and it is easy to check that this descends to a group structure on equivalence classes. Hence, the above map induces a group isomorphism between $(\mc{P}(\{1, \dots, r\})/ \sim) \times \Z$ and $X^*(S_{\phi})$.
    
Fix an extended pure inner twist $(\GU_n, \varrho, z)$ of $\GU^*_n$. Fix also a dominant cocharacter $\mu$ of $\GU_n$ and let $b \in B(\GU_n, \mu)$ be the unique basic element. Then $J_b$ has the structure of an extended pure inner twist $(J_b, \varrho_b, z_b)$ of $\GU_n$ and also $(J_b, \varrho_b \circ \varrho, z + z_b)$ of $\GU^*_n$. Suppose that $\kappa(z) = m_0, \kappa(b) = m_1 \in \Z=X^*(\widehat{\GU^*_n})^{\Gamma})$. For $I \subset \{1, \dots, r\}$, we let $\pi_{[I],m_0 + m_1}$ denote the unique element of $\Pi_{\phi}(J_b, \varrho_b \circ \varrho)$ such that $\iota_{\mf{w}}(\pi_{[I],m_0 + m_1}) = \ttau_{[I],m_0 + m_1}$.

We now make the following definition.
\begin{definition}
 For $\phi$ a discrete parameter of $\GU_{n}$ and $\pi \in \Pi_{\phi}(\GU_{n}, \varrho)$ and $\rho \in \Pi_{\phi}(J_{b}, \varrho_b \circ \varrho)$, we define the representation $\delta_{\pi,\rho} := \iota_{\mf{w}}(\pi)^{\vee} \otimes \iota_{\mf{w}}(\rho)$ of $S_{\phi}$. 
\end{definition}
\begin{remark}
The representation $\delta_{\pi, \rho}$ does not depend on $\mf{w}$ since modifying $\mf{w}$ scales both $\iota_{\mf{w}}(\pi)$ and $\iota_{\mf{w}}(\rho)$ and hence keeps $\delta_{\pi, \rho}$ constant (\cite{KalethaContragredients}).
\end{remark}

In preparation for applications to the Kottwitz conjecture, we further elucidate the representation $\delta_{\pi, \rho}$ in the particular case that $\mu$ is given by $z \mapsto (\diag(z,1, \dots, 1), z)$, with reflex field $E$. In particular, we describe the $W_{E}$-representation $\Hom_{S_{\phi}}(\delta_{\pi, \rho}, r_{\mu} \circ \phi|_{W_E})$ where $r_{\mu}$ is the representation of $\widehat{\GU_n} \rtimes W_E \subset \LL \GU_n$ as in \cite[Lemma (2.1.2)]{KottwitzTOO}. 

The representation $r_{\mu}$ restricted to $\widehat{\GU_n} \cong \GL_n(\C) \times \C^{\times}$ is the irreducible representation with highest weight $\mu$ and (since $\mu$ is minuscule) has weights given by the Weyl orbit of the character $\widehat{\mu}$ of $\widehat{T} \times \C^{\times} \subset \widehat{\GU_n}$ corresponding to $\mu$. As a representation of $S_{\phi}$, the $\C^{\times}$-factor acts by $\widehat{c}$ and the $(\Z/2\Z)^{r-1}$-factor acts by $\bigoplus\limits^r_{i=1} \ttau^{\oplus d_i}_i$.  Hence, as an $S_{\phi}$-representation, we have $r_{\mu} \circ \phi|_{\mc{L}_E} \cong \bigoplus\limits^r_{i=1} \ttau^{\oplus d_i}_{i,1}$.

The packet $\Pi_{\phi}(\GU_n, \varrho)$ (resp. $\Pi_{\phi}(J_b, \varrho_b \circ \varrho)$) corresponds via $\iota_{\mf{w}}$ to the $2^{r-1}$ characters of the form $\ttau_{[I], m_0}$ (resp. $\ttau_{[I], m_0+1}$).  Hence, for each $\rho=\pi_{[I], m_0+1} \in \Pi_{\phi}(J_b, \varrho_b \circ \varrho)$, there is precisely one representation $\pi^{\rho}_i=\pi_{[I\oplus \{i\}],m_0}  \in \Pi_{\phi}(\GU_n, \varrho)$ such that $\delta_{\pi^{\rho}_i, \rho}  \cong \ttau_{i,1}$. Thus, we have $\bigoplus_{\pi \in \Pi_{\phi}(\GU_n, \varrho)} \mathrm{Hom}_{S_{\phi}}(\delta_{\pi,\rho},r_{\mu} \circ \phi|_{W_E}) \simeq  r_{\mu} \circ \phi|_{W_E}$ as $W_E$-representations (the same is true for $\mc{L}_E$ representations, but in the paper we will only be considering the $W_E$-action).

We now repeat our analysis for $\U_n$ equipped with the extended pure inner form $(\U_n, \varrho, z)$ and $\mu$ given by $z \mapsto \diag(z,1, \ldots, 1)$. Let $\ov{\phi} \in \Phi_2(\U^*_n)$. In this case, the unique basic $b \in B(\U_n, \mu)$ satisfies $\kappa(b) \in X^*(Z(\widehat{\U_n})^{\Gamma}) \cong \Z / 2 \Z$ is nontrivial. The packet $\Pi_{\ov{\phi}}(\U_n, \varrho)$ (resp. $\Pi_{\ov{\phi}}(J_b, \varrho_b \circ \varrho)$) is of size $2^{r-1}$ and corresponds via $\iota_{\mf{w}}$ with those characters of $S_{\ov{\phi}}$ whose restriction to $Z(\widehat{\U_n})^{\Gamma}$ equals $\kappa(z)$ (resp. $\kappa(z)+\kappa(b)$). As before, we have that $r_{\mu} \circ \ov{\phi}|_{W_E} = \ov{\phi}_E \cong \bigoplus\limits^r_{i=1} \tau^{\oplus d_i}_i$. We note that each $\tau_i$ restricts non-trivially to $Z(\widehat{\U_n})^{\Gamma}$. It follows that $\bigoplus\limits_{\pi \in \Pi_{\ov{\phi}}(\U_n, \varrho)}\Hom_{S_{\ov{\phi}}}(\delta_{\pi, \rho}, r_{\mu} \circ \ov{\phi}|_{W_E}) = \ov{\phi}_E|_{W_E}$.

We record the results of these computations in the following corollary.

\begin{corollary}{\label{stdmucor}}
Suppose that $\phi$ is a discrete parameter of $\GU_{n}$ and consider the dominant cocharacter $\mu$ given by $z \mapsto (\diag(z,1, \ldots, 1), z)$ of $\GU_{n}$ and $b \in B(\GU_n,\mu)$ the unique basic element, then we have an isomorphism
\[ \bigoplus_{\pi \in \Pi_{\phi}(\GU_n, \varrho)} \mathrm{Hom}_{S_{\phi}}(\delta_{\pi,\rho},r_{\mu} \circ \phi|_{\mc{W}_{E}}) \simeq  r_{\mu} \circ \phi|_{W_E} \]
for all $\rho \in \Pi_{\phi}(J_b, \varrho_{b} \circ \varrho)$. 
Similarly, if $\ov{\phi}$ is a discrete parameter of $\U_{n}$ and $\mu$ is given by $z \mapsto \diag(z,1, \ldots, 1)$ and $b \in B(\U_n,\mu)$ is the unique basic element, then we have an isomorphism
\[ \bigoplus_{\pi \in \Pi_{\ov{\phi}}(\U_n, \varrho)} \mathrm{Hom}_{S_{\ov{\phi}}}(\delta_{\pi,\rho},r_{\mu} \circ \ov{\phi}|_{W_{E}}) \simeq  \ov{\phi}_E|_{W_{E}} \]
for all $\rho \in \Pi_{\ov{\phi}}(J_b, \varrho_{b}\circ \varrho)$.

More precisely, in the $\GU_n$ case, if $\rho \in \Pi_{\phi}(J_{b}, \varrho_{b} \circ \varrho)$, and we write $r_{\mu} \circ \phi|_{\mathcal{L}_{E}} = \bigoplus_{i = 1}^{r} \phi_{i}$, where $\phi_{i}$ are distinct irreducible representations of $\mathcal{L}_{E}$ for $i = 1,\ldots,r$. Then, if $\rho = \pi_{[I], m_0+1}$ is the representation corresponding to $[I] \in (\mc{P}(\{1, \dots, r\})/ \sim)$, we have that 
\[ \mathrm{Hom}_{S_{\phi}}(\delta_{\pi_{[I \oplus \{i\},m_0]},\rho},r_{\mu} \circ \phi|_{W_{E}}) \simeq \phi_{i}|_{W_{E}}, \]
for all $i = 1,\ldots,r$, and is $0$ for any irreducible representation $\pi$ other than the $\pi_{[I \oplus \{i\},m_0]}$ (since these are the only $\pi$ such that $\delta_{\pi, \rho}$ appears in the decomposition of $r_{\mu} \circ \phi|_{W_{E}})$. Similarly, in the $\U_n$ case, if we write $\ov{\phi}|_{\mathcal{L}_{E}} = \bigoplus_{i = 1}^{r} \ov{\phi}_{i}$ as a direct sum of irreducible $\mathcal{L}_{E}$-representations then, if $\rho = \pi_{[I]}$ is the representation corresponding to $I \subset \{1,\ldots,r\}$ with odd cardinality we have that
\[ \mathrm{Hom}_{S_{\ov{\phi}}}(\delta_{\pi_{[I \oplus \{i\}]},\rho},r_{\mu} \circ \ov{\phi}|_{W_{E}}) \simeq \ol{\phi}_{i}|_{W_{E}} \]
for all $i = 1,\ldots,r$, and is $0$ otherwise. 
\end{corollary} 
We will now upgrade this. Consider the case where $\mu_{d} = ((1^{d},0^{n - d}),1)$ and $b \in B(\GU_{n},\mu_{d})$ is the unique basic element. We note that, in this case, we have identifications
\[ r_{\mu_{d}} \circ \phi|_{\mc{L}_E} = \Lambda^{d}(\ov{\phi}_E)\otimes\chi_{\phi} \simeq \bigoplus  \otimes_{i = 1}^{r} \Lambda^{c_{i}}(\ov{\phi}_{i}) \otimes \chi_{\phi}, \]
where the direct sum is over partitions $\sum_{i = 1}^{r} c_{i} = d$ for $c_{i} \in \mathbb{N}_{\geq 0}$. It follows that, as an $S_{\phi}$-representation, the summand associated to $(c_1, \dots, c_r)$ is isomorphic (up to multiplicity) to $\ttau_{[J],1}$, where $J$ contains each $i \in \{1, \dots, r\}$ such that $c_i$ is odd. For simplicity, we will just be interested in the case where $d \leq r$ and $|J|=d$. We now have the following corollary of the above discussion.
\begin{corollary}{\label{extmucor}}
Suppose that $\phi$ is a discrete parameter of $\GU^*_{n}$. Write $r_{\mu_{1}} \circ \phi|_{\mathcal{L}_{E}} = \bigoplus_{i = 1}^{r} \phi_{i}$ for distinct irreducible representations of $\mathcal{L}_{E}$, and consider the dominant cocharacter given by $\mu_{d} = ((1^{d},0^{n - d}),1)$ for $1 \leq d \leq r$, and $b \in B(\GU_{n},\mu_{d})$. Then, if $\rho = \pi_{[I],m_0+1} \in \Pi_{\phi}(J_b, \varrho_{b} \circ \varrho)$ is the representation corresponding to $[I] \in (\mc{P}(\{1, \dots, r\})/ \sim)$ as above, we have an isomorphism 
\[ \mathrm{Hom}_{S_{\phi}}(\delta_{\pi,\rho},r_{\mu_{d}} \circ \phi|_{W_{E}}) \simeq \bigotimes_{j \in J} \phi_{j}|_{W_E} \]
for varying $[J] \in (\mc{P}(\{1, \dots, r\})/ \sim)$ represented by $J \in \mc{P}(\{1, \dots, r\})$ with cardinality $d$, and $\pi = \pi_{[I \oplus J],m_0} \in \Pi_{\phi}(\GU_{n},\varrho)$ the representation corresponding to the symmetric difference of $I$ and $J$. 
\end{corollary}
We now record some similar consequences for $\U_{n}$. We abuse notation and similarly write $\mu_{d}$ for the geometric dominant cocharacter of $\U_{n}$ with weights given by $(1^{d},0^{n-d})$. Let $b_d \in B(G,\mu_{d})$ denote the unique basic element. We note that the $\kappa$-invariant induces an isomorphism: $B(\U_{n})_{\bas} \simeq X_*(Z(\hat{\U}_{n})^{\Gamma}) \simeq \mathbb{Z}/2\mathbb{Z}$. We let $b$ denote the element of $B(\U^*_n)_{\bas}$ associated to the fixed extended pure inner twist $(\U_n, \varrho, z)$. If $\ov{\phi}$ is a discrete parameter of $\U_{n}$, we denote $r_{\mu_{1}} \circ \phi|_{\mc{L}_{E}}$ by $\ov{\phi}_{E}$ and observe that, for all $d$, we have an identification $r_{\mu_{d}} \circ \ov{\phi}|_{\mc{L}_{E}} \simeq \Lambda^{d}(\ov{\phi}_{E})$. Writing $\ov{\phi}_{E} = \bigoplus_{i = 1}^{r} \ov{\phi}_i$ as above, we have that $\Lambda^{d}(\ov{\phi}_{E})$ breaks up as a direct sum of terms
\[ \bigotimes_{i = 1}^{r} \Lambda^{c_{i}}(\ov{\phi}_{i}) \] 
for tuples $(c_{i})_{i = 1}^{r}$ such that $\sum_{i = 1}^{r} c_{i} = d$. This summand is (up to multiplicity) isomorphic to $\otimes_{i = 1}^{r} \tau_{i}^{\otimes c_{i}}$
as an $S_{\ov{\phi}}$-representation. Now, restricting to the case that $1 \leq d \leq r$, since the $\tau_{i}$ are of order $2$, this will be isomorphic to representations of the form $\tau_{J}$ for $J \subset \{1,\ldots,r\}$ satisfying $|J| \leq d$, and all the $J$ satisfying $|J| = d$ occur. Moreover, by considering the central character, it follows that all $J$ appearing must satisfy that $|J| = d \mod 2$. This allows us to deduce the following. 
\begin{corollary}{\label{Unextmucor}}
Suppose that $\ov{\phi}$ is a discrete parameter of $\U_{n}$. We write $r_{\mu_{1}} \circ \phi|_{\mathcal{L}_{E}} = \bigoplus_{i = 1}^{r} \ov{\phi}_{i}$ for $r$ distinct irreducible representations of $\mathcal{L}_{E}$, and consider the dominant cocharacter given by $\mu_{d} = (1^{d},0^{n - d})$ for $1 \leq d \leq r$, and $b_d \in B(\U_n,\mu_{d})$. Then, if $\rho = \pi_{I} \in \Pi_{\ov{\phi}}(\U_{n},\varrho_{b_d} \circ \varrho)$ is a representation corresponding to $I \in \mc{P}(\{1,\ldots,r\})$ with $|I|=d + \kappa(b) \mod 2$, we have an isomorphism 
\[ \mathrm{Hom}_{S_{\ov{\phi}}}(\delta_{\pi,\rho},r_{\mu_{d}} \circ \ov{\phi}|_{W_{E}}) \simeq \bigotimes_{j \in J} \ov{\phi}_{j}|_{W_{E}}  \]
for varying $J \in \mc{P}(\{1, \dots, r\})$ such that $|J| = d \mod 2$, and $\pi = \pi_{I \oplus J} \in \Pi_{\ov{\phi}}(\U_{n},\varrho)$ the representation corresponding to the symmetric difference of $I$ and $J$. Moreover, $\mathrm{Hom}_{S_{\ov{\phi}}}(\delta_{\pi,\rho},r_{\mu_{d}} \circ \ov{\phi}|_{W_{E}})$ will be $0$ unless $\pi = \pi_{J}$ for $|J| = d \mod 2$ and $|J| \leq d$. 
\end{corollary}
\subsection{The Correspondence of Fargues--Scholze}    
In this section, we review the Fargues--Scholze local Langlands correspondence as well as its connection to the cohomology of local shtuka spaces. This correspondence is denoted by $\LLC^{\mathrm{FS}}_G$ for $G$ a connected reductive group.

\subsubsection{The Fargues--Scholze correspondence}{\label{ss: FSLLC}}
Let $G/\mathbb{Q}_{p}$ be a connected reductive group. We let $\Perf$ denote the category of perfectoid spaces over $\ol{\mathbb{F}}_{p}$. We write $\ast := \Spd(\ol{\mathbb{F}}_{p})$ for the natural base. The key object of study is the moduli stack $\Bun_{G}$ sending $S \in \Perf$ to the groupoid of $G$-bundles on the relative Fargues--Fontaine curve $X_{S}$. We recall that, for any Artin $v$-stack $X$, Fargues--Scholze define a triangulated category $\Dc_{\blacksquare}(X,\overline{\mathbb{Q}}_{\ell})$ of solid $\overline{\mathbb{Q}}_{\ell}$-sheaves \cite[Section~VII.1]{FS} and isolate a nice full subcategory $\Dlis(X,\overline{\mathbb{Q}}_{\ell}) \subset \Dc_{\blacksquare}(X,\overline{\mathbb{Q}}_{\ell})$ of lisse-\'etale $\overline{\mathbb{Q}}_{\ell}$-sheaves \cite[Section~VII.6.]{FS}. We will be interested in the derived category $\Dlis(\Bun_{G},\ol{\mathbb{Q}}_{\ell})$. The key point is that objects in this category are manifestly related to smooth admissible representations of $G(\mathbb{Q}_{p})$. In particular, bundles on the Fargues--Fontaine curve are parametrized by elements $b \in B(G)$, and each element gives rise to a locally-closed Harder--Narasimhan stratum $\Bun_{G}^{b} \hookrightarrow \Bun_{G}$. These strata admit a natural map $\Bun_{G}^{b} \ra [\ast/J_{b}(\mathbb{Q}_{p})]$ to the classifying stack defined by the $\mathbb{Q}_{p}$-points of the $\sigma$-centralizer $J_{b}$, and, by \cite[Proposition~VII.7.1]{FS}, pullback along this map induces an equivalence
\[ \Dc(J_{b}(\mathbb{Q}_{p}),\ol{\mathbb{Q}}_{\ell}) \simeq \Dlis([\ast/J_{b}(\mathbb{Q}_{p})],\ol{\mathbb{Q}}_{\ell}) \xrightarrow{\simeq} \Dlis(\Bun_{G}^{b},\ol{\mathbb{Q}}_{\ell}), \]
where $\Dc(J_{b}(\mathbb{Q}_{p}),\ol{\mathbb{Q}}_{\ell})$ is the unbounded derived category of smooth $\ol{\mathbb{Q}}_{\ell}$-representations of 
$J_{b}(\mathbb{Q}_{p})$. We will repeatedly use this identification in what follows. In particular, for a smooth irreducible representation $\pi \in \Pi_{\ov{\Q}_{\ell}}(J_{b})$, we get an object $\rho \in \Dlis(\Bun_{G}^{b},\ol{\mathbb{Q}}_{\ell}) \subset \Dc(\Bun_{G},\ol{\mathbb{Q}}_{\ell})$ by extension by zero along the locally closed embedding $j_{b}: \Bun_{G}^{b} \hookrightarrow \Bun_{G}$\footnote{We note that in general a lower $!$-functor does not exist in the world of lisse sheaves; however, in this particular case one can use the definition given in the discussion after \cite[Lemma~3.1]{ImaConv}, where it is also verified that this has correct formal properties.}, and the Fargues--Scholze parameter comes from acting on this representation by endofunctors of $\Dlis(\Bun_{G},\ol{\mathbb{Q}}_{\ell})$ called Hecke operators. To introduce this, for a finite index set $I$, we let $\Rep_{\ol{\mathbb{Q}}_{\ell}}(\phantom{}^{L}G^{I})$ denote the category of algebraic representations of $I$-copies of the Langlands dual group, and we let $\Div^{I}$ be the product of $I$-copies of the diamond $\Div^{1} = \Spd(\Breve{\mathbb{Q}}_{p})/\mathrm{Frob}^{\mathbb{Z}}$. The diamond $\Div^{1}$ parametrizes, for $S \in \Perf$, characteristic $0$ untilts of $S$, which in particular give rise to Cartier divisors in $X_{S}$. We then have the Hecke stack
\[ \begin{tikzcd}
& & \arrow[dl, swap, "h^{\leftarrow}"] \mathrm{Hck} \arrow[dr,"h^{\rightarrow} \times supp"] & & \\
& \Bun_{G} & & \Bun_{G} \times \Div^{I}  & 
\end{tikzcd} \]
defined as the functor parametrizing, for $S \in \Perf$ together with a map $S \rightarrow \Div^{I}$ corresponding to characteristic $0$ untilts $S_{i}^{\sharp}$ defining Cartier divisors in $X_{S}$ for $i \in I$, a pair of $G$-torsors $\mathcal{E}_{1}$, $\mathcal{E}_{2}$ together with an isomorphism
\[ \beta:\mathcal{E}_{1}|_{X_{S} \setminus \bigcup_{i \in I} S_{i}^{\sharp}} \xrightarrow{\simeq} \mathcal{E}_{2}|_{X_{S} \setminus \bigcup_{i \in I} S_{i}^{\sharp}},\]
where $h^{\leftarrow}((\mathcal{E}_{1},\mathcal{E}_{2},i,(S_{i}^{\sharp})_{i \in I})) = \mathcal{E}_{1}$ and $h^{\rightarrow} \times supp((\mathcal{E}_{1},\mathcal{E}_{2},\beta,(S_{i}^{\sharp})_{i \in I})) = (\mathcal{E}_{2},(S_{i}^{\sharp})_{i \in I})$. For each element $W \in \Rep_{\overline{\mathbb{Q}}_{\ell}}(^{L}G^{I})$, the geometric Satake correspondence of Fargues--Scholze \cite[Chapter~VI]{FS} furnishes a solid $\overline{\mathbb{Q}}_{\ell}$-sheaf $\mathcal{S}'_{W}$ on $\mathrm{Hck}$. This allows us to define Hecke operators.
\begin{definition}
For each $W \in \Rep_{\overline{\mathbb{Q}}_{\ell}}(\phantom{}^{L}G^{I})$, we define the Hecke operator
\[ T_{W}: \Dlis(\Bun_{G},\overline{\mathbb{Q}}_{\ell}) \rightarrow \Dc_{\blacksquare}(\Bun_{G} \times X^{I},\ol{\mathbb{Q}}_{\ell}) \]
\[ A \mapsto R(h^{\rightarrow} \times supp)_{\natural}(h^{\leftarrow *}(A) \otimes^{\mathbb{L}} \mathcal{S}'_{W}),\]
where $\mathcal{S}'_{W}$ is the solid $\overline{\mathbb{Q}}_{\ell}$-sheaf defined in \cite[Section~IX.2]{FS}, and the functor $R(h^{\rightarrow} \times supp)_{\natural}$ is the natural push-forward (i.e the left adjoint to the restriction functor in the category of solid $\overline{\mathbb{Q}}_{\ell}$-sheaves \cite[Proposition~VII.3.1]{FS}). 
\end{definition}
It follows by 
\cite[Theorem~I.7.2, Proposition~IX.2.1, Corollary~IX.2.3]{FS} that this induces a functor 
\[ \Dlis(\Bun_{G},\overline{\mathbb{Q}}_{\ell}) \rightarrow \Dlis(\Bun_{G},\overline{\mathbb{Q}}_{\ell})^{BW_{\mathbb{Q}_{p}}^{I}} \]
which we will also denote by $T_{W}$. If $I = \{\ast\}$ is a singleton and $E_{W}$ denotes the reflex field of $W$ then $T_{W}$ factors through a functor taking values in $\Dlis(\Bun_{G},\overline{\mathbb{Q}}_{\ell})^{BW_{E_{W}}}$, composed with the map $\Dlis(\Bun_{G},\overline{\mathbb{Q}}_{\ell})^{BW_{E_{W}}} \ra \Dlis(\Bun_{G},\overline{\mathbb{Q}}_{\ell})^{BW_{\mathbb{Q}_{p}}}$ given by inducing the $W_{E_{W}}$-action to $W_{\mathbb{Q}_{p}}$.  We will abuse notation by writing $T_{W}$ for the functor taking values in $\Dlis(\Bun_{G},\overline{\mathbb{Q}}_{\ell})^{BW_{E_{W}}}$.

The Hecke operators are natural in $I$ and $W$ and compatible with exterior tensor products. For a finite set $I$, a representation $W \in \Rep_{\overline{\mathbb{Q}}_{\ell}}(\phantom{}^LG^{I})$, maps $\alpha: \overline{\mathbb{Q}}_{\ell} \rightarrow \Delta^{*}W$ and $\beta: \Delta^{*}W \rightarrow \overline{\mathbb{Q}}_{\ell}$, and elements $(\gamma_{i})_{i \in I} \in W_{\mathbb{Q}_{p}}^{I}$ for $i \in  I$, one defines the excursion operator on $\Dlis(\Bun_{G},\overline{\mathbb{Q}}_{\ell})$ to be the natural transformation of the identity functor given by the composition:
\[ id = T_{\overline{\mathbb{Q}}_{\ell}} \xrightarrow{\alpha} T_{\Delta^{*}W} = T_{W} \xrightarrow{(\gamma_{i})_{i \in I}} T_{W} = T_{\Delta^{*}W}  \xrightarrow{\beta} T_{\overline{\mathbb{Q}}_{\ell}} = id. \]
In particular, for all such data, we get an endomorphism of a smooth irreducible $\pi \in \Dc(G(\mathbb{Q}_{p}),\ol{\mathbb{Q}}_{\ell}) \simeq \Dlis(\Bun_{G}^{1},\ol{\mathbb{Q}}_{\ell}) \subset \Dlis(\Bun_{G},\ol{\mathbb{Q}}_{\ell})$ which, by Schur's lemma will give us a scalar in $\ol{\mathbb{Q}}_{\ell}$. In other words, to the datum $(I,W,(\gamma_{i})_{i \in I},\alpha,\beta)$ we assign a scalar. The natural compatibilities between Hecke operators will give rise to natural relationships between these scalars. These scalars and the relations they satisfy can be used, via Lafforgue's reconstruction theorem \cite[Proposition~11.7]{VL}, to construct a unique continuous semisimple map
\[ \phi^{\mathrm{FS}}_{\pi}: W_{\mathbb{Q}_{p}} \rightarrow \phantom{}^{L}G(\overline{\mathbb{Q}}_{\ell}), \]
which is the Fargues--Scholze parameter of $\pi$. It is characterized by the property that, for all $I,W,\alpha,\beta$ and $(\gamma_{i})_{i \in I} \in W_{\mathbb{Q}_{p}}^{I}$, the corresponding endomorphism of $\pi$ defined above is given by multiplication by the scalar that results from the composite
\[ \overline{\mathbb{Q}}_{\ell} \xrightarrow{\alpha} \Delta^{*}W = W \xrightarrow{(\phi^{\mathrm{FS}}_{\pi}(\gamma_{i}))_{i \in I}} W = \Delta^{*}W \xrightarrow{\beta} \overline{\mathbb{Q}}_{\ell}. \]
Fargues and Scholze show that their correspondence has various good properties which we will invoke throughout. 
\begin{theorem}{\cite[Theorem~I.9.6]{FS}}{\label{FSproperties}}
The mapping defined above 
\[ \pi \mapsto \phi^{\mathrm{FS}}_{\pi} \]
enjoys the following properties:
\begin{enumerate}
    \item (Compatibility with Local Class Field Theory) If $G = T$ is a torus, then $\pi \mapsto \phi_{\pi}^{\mathrm{FS}}$ is the usual local Langlands correspondence constructed from class field theory.
    \item The correspondence is compatible with character twists, passage to contragredients, and central characters.
    \item (Compatibility with products) Given two irreducible representations $\pi_{1}$ and $\pi_{2}$ of two connected reductive groups $G_{1}$ and $G_{2}$ over $\mathbb{Q}_{p}$, respectively, we have
    \[ \pi_{1} \boxtimes \pi_{2} \mapsto \phi^{\mathrm{FS}}_{\pi_{1}} \times \phi^{\mathrm{FS}}_{\pi_{2}}\]
    under the Fargues--Scholze local Langlands correspondence for $G_{1} \times G_{2}$. 
    \item (Compatibility with parabolic induction) Given a parabolic subgroup $P \subset G$ with Levi factor $M$ and a representation $\pi_{M}$ of $M$, then the Weil parameter corresponding to any sub-quotient of $I_{P}^{G}(\pi_{M})$ the (normalized) parabolic induction is the composition
    \[ W_{\mathbb{Q}_{p}}\xrightarrow{\phi^{\mathrm{FS}}_{\pi_{M}}} \\  ^{L}M(\overline{\mathbb{Q}}_{\ell}) \rightarrow ^{L}G(\overline{\mathbb{Q}}_{\ell}) \]
    where the map $\phantom{}^{L}M(\overline{\mathbb{Q}}_{\ell}) \rightarrow \phantom{}^{L}G(\overline{\mathbb{Q}}_{\ell})$ is the natural embedding. 
    \item (Compatibility with Harris--Taylor/Henniart LLC)
    For $G = GL_{n}$ or an inner form of $G$ (\cite[Theorem 6.6.1]{HKW}) the Weil parameter associated to $\pi$ is the (semi-simplified) parameter $\phi^{\mathrm{FS}}_{\pi}$ associated to $\pi$ by Harris--Taylor/Henniart.
    \item (Compatibility with Restriction of Scalars) Via an analogous construction to the one above, we can construct Fargues--Scholze parameters for any $G'$ a connected reductive group over any finite extension $E'/\mathbb{Q}_{p}$, where one then gets a Weil parameter valued on $W_{E'}$. If $G = Res_{E'/\mathbb{Q}_{p}}G'$ is the Weil restriction of some $G'/E'$ then the Fargues--Scholze correspondence for $G/\mathbb{Q}_{p}$ agree with the Fargues--Scholze correspondence for $G'/E'$ in the usual sense.
    \item (Compatibility with Isogenies) If $G' \rightarrow G$ is a map of reductive groups inducing an isomorphism of adjoint groups, $\pi$ is an irreducible smooth representation of $G(E)$ and $\pi'$ is an irreducible constituent of $\pi|_{G'(E)}$ then $\phi^{\mathrm{FS}}_{\pi'}$ is the image of $\phi^{\mathrm{FS}}_{\pi}$ under the induced map $\LL G \rightarrow \LL G'$. 
\end{enumerate}
\end{theorem}
\begin{remark}
Assume that $G$ is quasi-split for simplicity. Given $b \in B(G)$ with $\sigma$-centralizer $J_{b}$, we note that there are two natural ways of defining the Fargues--Scholze parameter of a smooth irreducible representation $\rho \in \Pi_{\ov{\Q}_{\ell}}(J_{b})$. One could consider the action of the excursion algebra of $\Bun_{J_{b}}$ on $\rho \in \Dlis(\Bun_{J_{b}}^{1},\ol{\mathbb{Q}}_{\ell}) \subset \Dlis(\Bun_{J_{b}},\ol{\mathbb{Q}}_{\ell})$, where $1 \in B(J_{b})$ is the trivial element or one could consider the action of the excursion algebra of $\Bun_{G}$ on $\rho \in \Dlis(\Bun_{G}^{b},\ol{\mathbb{Q}}_{\ell}) \subset \Dlis(\Bun_{G},\ol{\mathbb{Q}}_{\ell})$. The first construction gives a parameter $\phi_{\rho}^{1}$ valued in $\phantom{}^{L}J_{b}$ and the second construction gives a parameter $\phi_{\rho}^{b}$ valued in $\phantom{}^{L}G$. Since $J_{b}$ is an extended pure inner form of a Levi subgroup of $G$, it follows that we have an induced embedding $\phantom{}^{L}J_{b} \ra \phantom{}^{L}G$. Now, by \cite[Theorem~IX.7.2]{FS}, this embedding twisted appropriately takes $\phi_{\rho}^{1}$ to $\phi_{\rho}^{b}$. In particular, if $b$ is basic there are no twists and these parameters are the same under the identification $\phantom{}^{L}J_{b} \simeq \phantom{}^{L}G$ provided by the inner twisting.  
\end{remark}
\subsubsection{Local shtuka spaces}
A major goal of the present work is to relate this correspondence to classical instances of the local Langlands correspondence. The key insight is that the Hecke operators will be computed in terms of the cohomology of certain shtuka spaces, which for the highest weight representations of $\phantom{}^{L}G$ defined by minuscule cocharacters of $G$, will uniformize global Shimura varieties and in turn be related to the trace formula and classical instances of the Langlands correspondence. We now turn to defining the relevant shtuka spaces.

We now introduce the $\mu$-admissible locus. Given a geometric dominant cocharacter $\mu$ of $G$ with reflex field $E_{\mu}$, we can define the element:
\[ \tilde{\mu} := \frac{1}{[E_{\mu}:\mathbb{Q}_{p}]} \sum_{\gamma \in \Gamma/ \Gamma_{E_{\mu}}} \gamma(\mu) \in X_*(T_{\overline{\mathbb{Q}}_{p}})^{+,\Gamma}_{\mathbb{Q}} \]
We let $\mu^{\flat}$ be the image of $\mu$ in $\pi_1(G)_{\Gamma}$.
\begin{definition}{\label{shiftedbgu}}
Let $b'$ be an element in the Kottwitz set $B(G)$ of $G$ and $\mu$ a geometric dominant cocharacter, we define the set $B(G,\mu,b')$ to be set of $b \in B(G)$ for which $\nu_{b} - \nu_{b'} \leq \tilde{\mu}$  with respect to the Bruhat ordering and $\kappa(b) - \kappa(b') = \mu^{\flat}$. 
\end{definition}
These sets will carve out the necessary conditions for our shtuka spaces to be non-empty (See \cite[Corollary~5.4]{Vi}). In particular, if we fix an element $b' \in B(G)$ and let $b \in B(G,\mu,b')$. We call the quadruple $(G,b,b',\mu)$ a local shtuka datum. For simplicity, we will work with just a basic local Shtuka datum (i.e where both $b$ and $b'$ are basic elements), as this will be the only relevant case for us and there are various subtleties involved in making the proceeding discussion work in general. Attached to it, we define the shtuka space
\[ \Sht(G,b,b',\mu)_{\infty} \rightarrow \Spd(\Breve{E}_{\mu}), \]
as in \cite{SW}, to be the space parametrizing modifications 
\[  \mathcal{E}_{b} \dashrightarrow \mathcal{E}_{b'} \]
of $G$-bundles on the Fargues--Fontaine curve $X$ with meromorphy bounded by $\mu$.
\begin{remark}
We note that our definition of $\Sht(G,b,b',\mu)_{\infty}$ coincides with $\Sht(G,b,b',-\mu)_{\infty}$ in the notation of \cite{SW}, where $-\mu$ is the dominant inverse of $\mu$. This convention limits the appearance of duals when studying the cohomology of these spaces.  
\end{remark}
This has commuting actions of $J_{b}(\mathbb{Q}_{p})$ and $J_{b'}(\mathbb{Q}_{p})$ coming from automorphisms of $\mathcal{E}_{b}$ and $\mathcal{E}_{b'}$, respectively. Let $E_{\mu}$ be the reflex field of $\mu$. We define the tower 
\[ \Sht(G,b,b',\mu)_{K} := \Sht(G,b,b',\mu)/\underline{K} \ra \Spd(\Breve{E}_{\mu}) \] of locally spatial diamonds \cite[Theorem~23.1.4]{SW}
for varying open compact subgroups $K \subset J_{b'}(\mathbb{Q}_{p})$, where $\Breve{E}_{\mu} := \Breve{\mathbb{Q}}_{p}E_{\mu}$. When $b'$ is trivial, we denote $\Sht(G,b,b',\mu)$ by $\Sht(G,b,\mu)$. 
There is a natural map 
\[ \mathsf{p}: \Sht(G,b,b',\mu)_{\infty} \ra \mathrm{Hck} \] 
to the Hecke stack, which maps to the locus of modifications with meromorphy bounded by $\mu$. Attached to the geometric cocharacter $\mu$, consider the highest weight representation $V_{\mu} \in \Rep_{\ol{\mathbb{Q}}_{\ell}}(\phantom{}^{L}G)$. The associated $\ol{\mathbb{Z}}_{\ell}$-sheaf  $\mathcal{S}'_{\mu}$ on $\mathrm{Hck}$ considered above will be supported on this locus, and we abusively denote $\mathcal{S}'_{\mu}$ for the pullback of this sheaf along $\mathsf{p}$. Since $\mathsf{p}$ factors through the quotient of $\Sht(G,b,b',\mu)_{\infty}$ by the simultaneous group action of $J_{b}(\mathbb{Q}_{p}) \times J_{b'}(\mathbb{Q}_{p})$, this sheaf will be equivariant with respect to these actions. This allows us to define the complex
\[ R\Gamma_{c}(G,b,b',\mu) := \colim_{K \rightarrow \{1\}} R\Gamma_{c}(\Sht(G,b,b',\mu)_{K,\mathbb{C}_{p}},\mathcal{S}'_{\mu}) \otimes \overline{\mathbb{Q}}_{\ell} \]
which will be a complex of smooth admissible $J_{b}(\mathbb{Q}_{p}) \times J_{b'}(\mathbb{Q}_{p}) \times W_{E_{\mu}}$-modules, where $\Sht(G,b,b',\mu)_{K,\mathbb{C}_{p}}$ is the base change of $\Sht(G,b,b',\mu)_{K}$ to $\mathbb{C}_{p}$. For $\pi_{b} \in \Pi_{\ov{\Q}_{\ell}}(J_{b})$, this allows us to define the following complexes 
\begin{equation}{\label{eq: RGammaflat}}
    R\Gamma^{\flat}_{c}(G,b,b',\mu)[\pi_{b}] := R\mathcal{H}om_{J_{b}(\mathbb{Q}_{p})}(R\Gamma_{c}(G,b,b',\mu),\pi_{b})
\end{equation}
and
\begin{equation}{\label{eq: RGamma}}
    R\Gamma_{c}(G,b,b',\mu)[\pi_{b}] := R\Gamma_{c}(G,b,b',\mu)) \otimes^{\mathbb{L}}_{\mathcal{H}(J_{b})} \pi_{b}
\end{equation}
where $\mathcal{H}(J_{b})$ is the smooth Hecke algebra. Analogously, for $\pi_{b'} \in \Pi_{\ov{\Q}_{\ell}}(J_{b'})$, we can define $R\Gamma_{c}(G,b,b',\mu)[\pi_{b'}]$ and $R\Gamma^{\flat}_{c}(G,b,b',\mu)[\pi_{b'}]$. It follows by \cite[Corollary~I.7.3]{FS} and \cite[Page~317]{FS} that these will be valued in smooth admissible representations of finite length. We now relate these complexes to Hecke operators on $\Bun_{G}$. In particular, we have the following result.
\begin{lemma}{\label{shimhecke}}{\cite[Section~IX.3]{FS}}
Given a basic local shtuka datum $(G,b,b',\mu)$ as above and $\pi_{b'}$ (resp. $\pi_{b}$) a smooth irreducible representation of $J_{b'}(\mathbb{Q}_{p})$ (resp. $J_{b}(\mathbb{Q}_{p})$), we can consider the associated sheaves $\pi_{b} \in \Dc(J_{b}(\mathbb{Q}_{p}),\Lambda) \simeq \Dlis(\Bun_{G}^{b})$ and $\pi_{b'} \in \Dlis(\Bun_{G}^{b'}) \simeq \Dc(J_{b'}(\mathbb{Q}_{p}),\Lambda)$ on the HN-stratum $j_{b}: \Bun_{G}^{b} \hookrightarrow \Bun_{G}$ and $j_{b'}: \Bun_{G}^{b'} \hookrightarrow \Bun_{G}$. There then exists isomorphisms
\[ R\Gamma_{c}(G,b,b',\mu)[\pi_{b}] \simeq  j_{b'}^{*}T_{\mu}j_{b!}(\pi_{b}), \qquad  R\Gamma^{\flat}_{c}(G,b,b',\mu)[\pi_{b}] \simeq  j_{b'}^{*}T_{\mu}Rj_{b*}(\pi_{b}),\]
of complexes of $J_{b'}(\mathbb{Q}_{p}) \times W_{E_{\mu}}$-modules and isomorphisms
\[ R\Gamma_{c}(G,b,b',\mu)[\pi_{b'}] \simeq j_{b}^{*}T_{-\mu}j_{b'!}(\pi_{b}), \qquad  R\Gamma^{\flat}_{c}(G,b,b',\mu)[\pi_{b'}] \simeq j_{b}^{*}T_{-\mu}Rj_{b'*}(\pi_{b}), \]
of complexes of $J_{b}(\mathbb{Q}_{p}) \times W_{E_{\mu}}$-modules, where $-\mu$ is a dominant cocharacter conjugate to the inverse of $\mu$. 
\end{lemma}
\begin{remark}
In the case that $b'$ is the trivial element, we drop it from our notation, simply writing $R\Gamma_{c}(G,b,\mu)$ and $B(G,\mu)$. 
\end{remark}
The distinction between these the complexes $R\Gamma_{c}(G,b,b',\mu)[\pi_{b}]$ and $R\Gamma^{\flat}_{c}(G,b,b',\mu)[\pi_{b}]$ is a bit cumbersome. The complex $R\Gamma_{c}(G,b,b',\mu)[\pi_{b}]$ is much more natural from the point of view of geometric arguments on $\Bun_{G}$ as it involves the much simpler extension by zero functor, while the complex $R\Gamma^{\flat}_{c}(G,b,\mu)[\pi_{b}]$ is what is classically studied in the literature. Fortunately, since we will ultimately only be interested in describing these complexes for representations with supercuspidal Fargues--Scholze parameter, there is in fact no difference between the two. 
\begin{proposition}{\label{flatnatural}}
Let $(G,b,b',\mu)$ be a basic local shtuka datum, $\pi_{b} \in \Pi_{\ov{\Q}_{\ell}}(J_{b})$ a representation with supercuspidal Fargues--Scholze parameter. Then we have an isomorphism
\[ R\Gamma_{c}(G,b,b',\mu)[\pi_{b}] \simeq R\Gamma^{\flat}_{c}(G,b,b',\mu)[\pi_{b}], \]
of $J_{b'}(\mathbb{Q}_{p}) \times W_{E_{\mu}}$-modules. Similarly, we have an isomorphism  
\[ R\Gamma_{c}(G,b,b',\mu)[\pi_{b'}] \simeq R\Gamma^{\flat}_{c}(G,b,b',\mu)[\pi_{b'}], \]
of $J_{b}(\mathbb{Q}_{p}) \times W_{E_{\mu}}$-modules. 
\end{proposition}
\begin{proof}
We just do it for $\pi_{b}$ with the proof for $\pi_{b'}$ being the same. Using Lemma \ref{shimhecke}, there exist isomorphisms
\[ R\Gamma^{\flat}_{c}(G,b,b',\mu)[\pi_{b}] \simeq j_{b'}^{*}T_{\mu}j_{b!}(\pi_{b})\]
and
\[ R\Gamma_{c}(G,b,b',\mu)[\pi_{b}] \simeq j_{b'}^{*}T_{\mu}Rj_{b*}(\pi_{b}) \]
of complexes of $J_{b'}(\mathbb{Q}_{p}) \times W_{E_{\mu}}$-modules. Therefore, the claim follows from the following Lemma.
\end{proof}
\begin{lemma} Let $\pi_{b} \in \Pi_{\ov{\Q}_{\ell}}(J_{b})$ be a representation with supercuspidal Fargues--Scholze parameter $\phi$. Then the natural map 
\[ j_{b!}(\pi_{b}) \rightarrow Rj_{b*}(\pi_{b}) \]
of sheaves in $\Dlis(\Bun_{G},\overline{\mathbb{Q}}_{\ell})$ is an isomorphism. 
\end{lemma}
\begin{proof}
It suffices to show that the sheaf $Rj_{b*}(\pi_{b})$ has trivial restriction to $\Bun_{G}^{b'}$ for any non-basic $b' \in B(G)$. However, the restriction of $Rj_{b*}(\pi_{b})$ must have cohomology valued in representations with Fargues--Scholze parameter equal to $\phi$ under the appropriate twisted embedding $\phantom{}^LJ_{b} \rightarrow \phantom{}^{L}G$, by \cite[Section~IX.7.1]{FS}. However, for any non-basic $b \in B(G)$, the image of $\phantom{}^LJ_{b} \to \LL G$ is a proper Levi subgroup of $\phantom{}^{L}G$. Therefore, this restriction must vanish since $\phi$ was assumed to be supercuspidal and therefore does not factor through a Levi subgroup. 
\end{proof}

To study the cohomology of the complexes $R\Gamma_{c}(G,b,b',\mu)$, we will invoke a tool coming from the geometric Langlands correspondence known as the spectral action.
\subsubsection{The spectral action}{\label{sss: spectralaction}}
We now fix a supercuspidal parameter $\phi$ and $\pi_{b} \in \Pi_{\ov{\Q}_{\ell}}(J_b), \pi_{b'} \in \Pi_{\ov{\Q}_{\ell}}(J_{b'})$ such that the Fargues--Scholze $L$-parameter attached to $\pi_{b}$ and $\pi_{b'}$ is $\phi$. We then work with the complexes $R\Gamma_{c}(G,b,b',\mu)[\pi_{b}]$ and $R\Gamma_{c}(G,b,b',\mu)[\pi_{b'}]$. Our goal is to outline some categorical structure on these complexes coming from the spectral action of the stack of Langlands parameters. This can be found in  \cite[Section~X.2]{FS} when $G$ has trivial split center, and in \cite[Section~3.2]{Ham} in general. We consider the moduli stack $X_{\hat{G}}$ of Langlands parameters over $\ol{\mathbb{Q}}_{\ell}$, as defined in \cite{DH,Zhu1} and \cite[Section~VIII.I]{FS}, and the derived category $\Perf(X_{\hat{G}})$ of perfect complexes on $X_{\hat{G}}$. We write $\Perf(X_{\hat{G}})^{BW_{\mathbb{Q}_{p}}^{I}}$ for the derived category of objects with a continuous $W_{\mathbb{Q}_{p}}^{I}$ action for a finite index set $I$, and $\Dlis(\Bun_{G},\ol{\mathbb{Q}}_{\ell})^{\omega}$ for the triangulated sub-category of compact objects in $\Dlis(\Bun_{G},\ol{\mathbb{Q}}_{\ell})$. By \cite[Corollary~X.I.3]{FS}, there exists a $\ol{\mathbb{Q}}_{\ell}$-linear action 
\[ \Perf(X_{\hat{G}})^{BW_{\mathbb{Q}_{p}}^{I}} \ra \mathrm{End}(\Dlis(\Bun_{G},\ol{\mathbb{Q}}_{\ell})^{\omega})^{BW_{\mathbb{Q}_{p}}^{I}} \]
\[ C \mapsto \{A \mapsto C \star A\}\]
which, extending by colimits, gives rise to an action 
\[ \IndPerf(X_{\hat{G}})^{BW_{\mathbb{Q}_{p}}^{I}} \ra \mathrm{End}(\Dlis(\Bun_{G},\ol{\mathbb{Q}}_{\ell}))^{BW_{\mathbb{Q}_{p}}^{I}} \]
where $\IndPerf(X_{\hat{G}})$ is the triangulated category of Ind-Perfect complexes, and this action is uniquely characterized by some complicated properties. For our purposes, we will need the following:
\begin{enumerate}
    \item For $V = \boxtimes_{i \in I} V_{i}\in \Rep_{\ol{\mathbb{Q}}_{\ell}}(\phantom{}^{L}G^{I})$, there is an attached vector bundle $C_{V} \in \Perf(X_{\hat{G}})^{BW_{\mathbb{Q}_{p}}^{I}}$ whose evaluation at a $\ol{\mathbb{Q}}_{\ell}$-point of $X_{\hat{G}}$ corresponding to a (not necessarily semi-simple) $L$-parameter $\tilde{\phi}: W_{\mathbb{Q}_{p}} \ra \phantom{}^{L}G(\ol{\mathbb{Q}}_{\ell})$ is the vector space $V$ with $W_{\mathbb{Q}_{p}}^{I}$-action given by $\boxtimes_{i \in I} r_{V_{i}} \circ \tilde{\phi}$. The endomorphism
    \[ C_{V} \star (-): \Dlis(\Bun_{G},\ol{\mathbb{Q}}_{\ell}) \ra \Dlis(\Bun_{G},\ol{\mathbb{Q}}_{\ell})^{BW_{\mathbb{Q}_{p}}^{I}} \]
    is the Hecke operator $T_{V}$ defined above.  
    \item The action is symmetric monoidal in the sense that given $C_{1},C_{2} \in \IndPerf(X_{\hat{G}})$, we have a natural equivalence of endofunctors: 
    \[ (C_{1} \otimes^{\mathbb{L}} C_{2}) \star (-) \simeq C_{1} \star (C_{2} \star (-)). \]
\end{enumerate}
We will only be interested in the action of a small piece of $X_{\hat{G}}$. Namely, let $S_{\phi}$ be the centralizer of our fixed supercuspidal $\phi$, and we set $\Rep_{\overline{\mathbb{Q}}_{\ell}}(S_{\phi})$ be the category of finite-dimensional algebraic $\overline{\mathbb{Q}}_{\ell}$-representations of the group $S_{\phi}$. The unramified twists of the parameter $\phi$ define a connected component $C_{\phi} \subset X_{\hat{G}}$, which is isomorphic to $[S/S_{\phi}]$ for some torus $S/\ol{\mathbb{Q}}_{\ell}$. This defines a direct summand 
\[ \Perf(C_{\phi}) \subset \Perf(X_{\hat{G}}), \]
and the spectral action gives rise to a direct summand
\[ \Dlis^{C_{\phi}}(\Bun_{G},\ol{\mathbb{Q}}_{\ell})^{\omega} \subset \Dlis(\Bun_{G},\ol{\mathbb{Q}}_{\ell})^{\omega}, \]
where Schur irreducible objects inside $\Dlis^{C_{\phi}}(\Bun_{G},\ol{\mathbb{Q}}_{\ell})^{\omega}$ will have Fargues-Scholze parameter given by an unramified twist of $\phi$. Thus, we can assume that the non-basic restrictions of $A \in \Dlis^{C_{\phi}}(\Bun_{G},\ol{\mathbb{Q}}_{\ell})^{\omega}$ have to be trivial, and the representations occurring as irreducible constituents of $A$ on the basic locus have to be supercuspidal, by compatibility of Fargues--Scholze with parabolic induction, \ref{FSproperties} (4) and Remark 2.20. Therefore we obtain a decomposition  
\[ \Dlis^{C_{\phi}}(\Bun_{G},\ol{\mathbb{Q}}_{\ell})^{\omega} \simeq \bigoplus_{b \in B(G)_{\mathrm{basic}}} \mathrm{D}^{C_{\phi}}(J_{b}(\mathbb{Q}_{p}),\ol{\mathbb{Q}}_{\ell})^{\omega}, \]
where $\mathrm{D}^{C_{\phi}}(J_{b}(\mathbb{Q}_{p}),\ol{\mathbb{Q}}_{\ell})^{\omega} \subset \mathrm{D}(J_{b}(\mathbb{Q}_{p}),\ol{\mathbb{Q}}_{\ell})^{\omega}$ is a full subcategory. We set $\chi: Z(G)(\mathbb{Q}_{p}) \ra \ol{\mathbb{Q}}_{\ell}^{*}$, to be the central character determined by our parameter $\phi$, and define $\Dlis^{C_{\phi},\chi}(\Bun_{G},\ol{\mathbb{Q}}_{\ell})^{\omega} \simeq \bigoplus_{b \in B(G)_{\mathrm{basic}}} \mathrm{D}^{C_{\phi},\chi}(J_{b}(\mathbb{Q}_{p}),\ol{\mathbb{Q}}_{\ell})^{\omega}$. Here $\mathrm{D}^{C_{\phi},\chi}(J_{b}(\mathbb{Q}_{p})),\ol{\mathbb{Q}}_{\ell})^{\omega}$ is the (not full) subcategory of $\mathrm{D}^{C_{\phi}}(J_{b}(\mathbb{Q}_{p}),\ol{\mathbb{Q}}_{\ell})^{\omega}$ spanned by compact objects with fixed central character $\chi: Z(J_{b}(\mathbb{Q}_{p})) 
 \simeq Z(G)(\mathbb{Q}_{p}) \ra \ol{\mathbb{Q}}_{\ell}^{*}$, where the isomorphism $Z(G) \simeq Z(J_{b})$ is guaranteed by the fact that $b$ is basic. Now, since supercuspidal representations are injective/projective in the category of smooth representations with fixed central character, this implies that $\Dlis^{C_{\phi},\chi}(\Bun_{G},\ol{\mathbb{Q}}_{\ell})^{\omega}$ can be identified with 
\[ \bigoplus_{b \in B(G)_{\bas}} \bigoplus_{\pi_{b}} \Perf(\ol{\mathbb{Q}}_{\ell}) \otimes \pi_{b}, \]
where $\Perf(\overline{\mathbb{Q}}_{\ell})$ is the category of perfect complexes of $\overline{\mathbb{Q}}_{\ell}$-modules, and $\pi_{b}$ runs over representations of $J_{b}(\mathbb{Q}_{p})$ with Fargues--Scholze parameter equal to $\phi$. 

As in \cite[Section~3.2]{Ham}, we can show that Hecke operators preserve this subcategory and it follows that the same is true for the spectral action, using \cite[Theorem~VIII.5.1]{FS}. We note that any $W \in \Rep_{\ol{\mathbb{Q}}_{\ell}}(S_{\phi})$ defines a vector bundle on the classifying stack $[\Spec{(\ol{\mathbb{Q}}_{\ell})}/S_{\phi}]$, which we can pullback to get a vector bundle on the classifying stack $[S/S_{\phi}] \simeq C_{\phi}$. The spectral action of this vector bundle induces a functor
\[ \bigoplus_{b \in B(G)_{\bas}} \bigoplus_{\pi_{b}} \Perf(\overline{\mathbb{Q}}_{\ell}) \otimes \pi_{b} \rightarrow  \bigoplus_{b \in B(G)_{\bas}} \bigoplus_{\pi_{b}} \Perf(\overline{\mathbb{Q}}_{\ell}) \otimes \pi_{b}, \]
which we denote by $\Act_{W}$. We will also abuse notation and write $\Act_{W}$ for the endofunctor of $\Dlis(\Bun_{G},\ol{\mathbb{Q}}_{\ell})^{\omega}$.

We note that the functor $\Act_{\mathbf{1}}$ identifies with the spectral action of the structure sheaf on $C_{\phi}$, and therefore identifies with the idempotent projection
\[ \Dlis(\Bun_{G},\ol{\mathbb{Q}}_{\ell})^{\omega} \ra \Dlis^{C_{\phi}}(\Bun_{G},\ol{\mathbb{Q}}_{\ell})^{\omega} \subset \Dlis(\Bun_{G},\ol{\mathbb{Q}}_{\ell})^{\omega} \]
to the direct summand. In particular, we note that the corresponding endofunctor on $\Dlis^{C_{\phi},\chi}(\Bun_{G},\ol{\mathbb{Q}}_{\ell})^{\omega}$ is the identity map. 

We now consider a geometric dominant minuscule cocharacter $\mu$ with attached highest weight representation $V_{\mu} \in \Rep_{\ol{\mathbb{Q}}_{\ell}}(\phantom{}^{L}G)$ and compute the spectral action of $C_{V_{\mu}}$ on a $\pi_{b}$. This computes the Hecke operator $T_{\mu}(\pi_{b})$ which, by Lemma \ref{shimhecke}, will be given by the cohomology of a local Shtuka space. However, by the above discussion, we can rewrite this in terms of the spectral action of the vector bundle given by the restriction of $C_{V_{\mu}}$ to the connected component $C_{\phi}$, which will be precisely a direct sum of the $\Act$-functors given by the representations of $S_{\phi}$ occurring in $r_{\mu} \circ \phi$. Moreover, if we apply the $\Act$-functors to a representation $\pi_{b}$ with Fargues-Scholze parameter $\phi$ the spectral action will factor over the base-change to the localization of $C_{\phi}$ at the closed point defined by $x_{\phi}$, since, if $\mf{m}_{\phi} \subset \mathcal{O}_{X_{\hat{G}}}$ denotes the maximal ideal in the ring of global functions defined by $\mf{m}_{\phi}$ all functions $f \in \mathcal{O}_{X_{\hat{G}}} \setminus \mf{m}_{\phi}$ will act via isomorphisms on $\mathcal{O}_{X_{\hat{G}}} \star \pi_{b} \simeq \pi_{b}$, by construction of the Fargues-Scholze parameter. By combining these observations, we obtain the following result. 
\begin{theorem}{\cite[Corollary~3.11]{Ham},\cite[Section~X.2]{FS}}{\label{thm: RGammavsAct}}
Let $(G,b,b',\mu)$ be a local shtuka datum with reflex field $E_{\mu}$. We write $\pi_{b}$ for a representation of $J_{b}(\mathbb{Q}_{p})$ with supercuspidal $L$-parameter $\phi$ and $r_{\mu}$ for the representation of $\hat{G} \rtimes W_{E_{\mu}}$ defined by $\mu$. We can view $r_{\mu} \circ \phi$ as a representation of $S_{\phi} \times W_{E_{\mu}}$, then we obtain a decomposition $\bigoplus_{i = 1}^{k} W_{i} \boxtimes \sigma_{i}$, where $W_{i} \in \Rep_{\overline{\mathbb{Q}}_{\ell}}(S_{\phi})$ is an irreducible representation of $S_{\phi}$ and $\sigma_{i}$ is a finite-dimensional $\overline{\mathbb{Q}}_{\ell}$-representation of $W_{E_{\mu}}$. Then we have an isomorphism
\[ R\Gamma_{c}(G,b,b',\mu)[\pi_{b}] \simeq \bigoplus_{i = 1}^{k} \Act_{W_{i}}(\pi_{b}) \boxtimes \sigma_{i} \]
 of $J_{b'}(\mathbb{Q}_{p}) \times W_{E_{\mu}}$-modules. 
\end{theorem}
The significance of this result is that these $\Act$-functors are monoidal since the spectral action is monoidal, and pullback along the map $C_{\phi} \ra [\Spec{\ol{\mathbb{Q}}_{\ell}}/S_{\phi}]$ commutes with tensor products. In particular, for $W,W' \in \Rep_{\ol{\mathbb{Q}}_{\ell}}(S_{\phi})$, we have an isomorphism:
\[ \Act_{W} \circ \Act_{W'}(-) \simeq \Act_{W \otimes W'}(-) \]
This means that if one knows the values of the $\Act$-functors on a set of irreducible representations $W \in \Rep_{\ol{\mathbb{Q}}_{\ell}}(S_{\phi})$ which generate all irreducible representations of $S_{\phi}$ under the operation of tensor products, then we can determine the value of $\Act$-functors for all representations and in turn how representations with supercuspidal Fargues--Scholze parameter contribute to the cohomology of local Shtuka spaces (even non-minuscule ones). In particular, for $\GU_{n}$, $\mu = ((1,0,\ldots,0),1)$ and $b'$ the trivial element, we can use the results of \cite{BMN} together with Theorem 1.1 to tell us just enough about the values of the $\Act$-functors to determine the rest by this monoidal property. To pass between the description of the complexes $R\Gamma_{c}(G,b,b',\mu)[\pi_{b}]$ and the values of the $\Act$-functors, we will invoke the following easy Lemma. 
\begin{lemma}{\label{actirred}}
For a basic element $b \in B(G)$, we let $\pi_{b} \in \Pi_{\ov{\Q}_{\ell}}(J_{b})$ be a smooth irreducible representation with supercuspidal Fargues--Scholze parameter $\phi$. Suppose that $\chi \in \Rep_{\ol{\mathbb{Q}}_{\ell}}(S_{\phi})$ is an irreducible character of $S_{\phi}$.  Then $\Act_{\chi}(\pi_{b})$ is an irreducible smooth representation supported on a single basic HN-stratum. 
\end{lemma}
\begin{proof}
Let $\chi^{-1}$ be the inverse of $\chi$. We note that, since $\pi_{b}$ has supercuspidal Fargues--Scholze parameter $\phi$, $\Act_{\mathbf{1}}(\pi_{b}) = \pi_{b}$, where $\mathbf{1}$ is the trivial representation. Since the $\Act$-functors are monoidal, we have that $\Act_{\chi^{-1}} \circ \Act_{\chi}(\pi_{b}) \simeq \Act_{\mathbf{1}}(\pi_{b}) \simeq \pi_{b}$. This in particular allows us to see that $\Act_{\chi}$ is a conservative functor on representations with Fargues--Scholze parameter $\phi$ (i.e sends non-zero objects to non-zero objects). Moreover, if the complex $\Act_{\chi}(\pi_{b})$ were not irreducible then we could write it as a direct sum of copies of supercuspidal representations sitting in possibly different degrees of cohomology. Acting on this decomposition by $\Act_{\chi^{-1}}(-)$, would induce an analogous direct sum decomposition of $\pi_{b}$. This would contradict irreducibility unless the functor failed to be conservative, which we just showed not to be the case. 
\end{proof}
We conclude this section by discussing the compatibility of the spectral action with central isogenies. Namely, suppose we have a map $G' \ra G$ of reductive groups which induces an isomorphism of adjoint groups. We consider the following diagram
\begin{equation}{\label{eqn: adjisodiagram}}
\begin{tikzcd}
\IndPerf(X_{\hat{G'}}) \arrow[dr] \arrow[d] & \\
\IndPerf(X_{\hat{G}}) \arrow[r] &  \mathrm{End}(\Dlis(\Bun_{G},\ol{\mathbb{Q}}_{\ell}))^{BW_{\mathbb{Q}_{p}}^{I}}
\end{tikzcd}    
\end{equation}
where:
\begin{enumerate}
    \item The horizontal arrow is given by the endomorphism induced by tensoring by $\varphi_{\natural}(\ol{\mathbb{Q}}_{\ell})$ followed by the spectral action for $G$.
    \item The vertical arrow is given by pullback along the natural map $\varphi: X_{\hat{G}} \ra X_{\hat{G}'}$ given by the map of dual groups $\hat{G} \ra \hat{G'}$ induced by the central isogeny. 
    \item The diagonal arrow is given by the map
    \[ C \mapsto \varphi_{\natural}(C \star \varphi^{*}(A)) \]
    where
    \[ \varphi: \Bun_{G'} \ra \Bun_{G} \]
    is the natural map induced by $G' \ra G$. 
\end{enumerate}
We would like to show that this diagram is well-defined and commutes. However, If we consider a basic element $b \in B(G)_{\bas}$ which does not occur in the image of $B(G') \ra B(G)$ and a complex $A$ supported on $\Bun_{G}^{b}$ then we have that $\varphi^{*}(A) = 0$. In particular, $\mathcal{O}_{X_{\hat{G}'}}$ will act by $0$, but the pullback $\mathcal{O}_{X_{\hat{G}}}$ will act via the identity on $A$, so the diagram will not commute for all $A$. To remedy this, we let $C \subset \pi_{1}(G)_{\Gamma}$ denote the set of elements which occur in the image of the induced map $\pi_{1}(G')_{\Gamma} \ra \pi_{1}(G)_{\Gamma}$, and consider the full sub-category $\Dlis(\Bun_{G}^{C},\ol{\mathbb{Q}}_{\ell}) \subset \Dlis(\Bun_{G},\ol{\mathbb{Q}}_{\ell})$, where $\Bun_{G}^{C} := \bigsqcup_{\alpha \in C} \Bun_{G}^{\alpha} \subset \bigsqcup_{\alpha \in \pi_{1}(G)_{\Gamma}} \Bun_{G}^{\alpha}$ are the connected components corresponding to $\alpha \in C \subset \pi_{1}(G)_{\Gamma}$, using the description of connected components given in \cite[Corollary~IV.1.23]{FS}. Now we claim the following.
\begin{proposition}{\label{prop: actisogeny}}
For $G' \ra G$ a map of reductive groups inducing an isomorphism of adjoint groups, the diagram 
\begin{equation*}
\begin{tikzcd}
\IndPerf(X_{\hat{G'}}) \arrow[dr] \arrow[d] & \\
\IndPerf(X_{\hat{G}}) \arrow[r] &  \mathrm{End}(\Dlis(\Bun_{G}^{C},\ol{\mathbb{Q}}_{\ell}))^{BW_{\mathbb{Q}_{p}}^{I}}
\end{tikzcd}
\end{equation*}
is well-defined and commutes, with the arrows and notation as defined above. 
\end{proposition}
\begin{proof}
We invoke \cite[Theorem~VIII.5.1]{FS}, which tells us that $\IndPerf(X_{\hat{G'}})$ is identified with the $\infty$-category of modules over $\mathcal{O}_{X_{\hat{G'}}}$ in $\IndPerf(B\hat{G'})$. Concretely, $\Perf(B\hat{G'})$ is generated by the vector bundles $C_{V}$ for $V \in \Rep_{\ol{\mathbb{Q}}_{\ell}}(\phantom{}^{L}G'^{I})$ which act via the spectral action by the Hecke operators $T_{V}$. The $\mathcal{O}_{X_{\hat{G}'}}$-module structure is given by excursion operators, where we note, by \cite[Theorem~VIII.3.6]{FS}, that there is an identification between the ring of global sections and excursion operators. It therefore suffices to exhibit a natural isomorphism for all $V$ that respects the excursion algebra. To do this, let $\widetilde{V} \in \Rep_{\ol{\mathbb{Q}}_{\ell}}(\phantom{}^{L}G^{I})$ be the representation determined by precomposing $V$ with the map $\hat{G}^{I} \ra \hat{G}'^{I}$. Then the commutativity of the above diagram reduces us to showing that we have a natural transformation 
\[ \varphi_{\natural}T_{V}(\varphi^{*}(-)) \simeq T_{\widetilde{V}}(- \otimes \varphi_{\natural}(\ol{\mathbb{Q}}_{\ell})),   \]
for all $A \in \Dlis(\Bun_{G}^{C},\ol{\mathbb{Q}}_{\ell})$, but this follows from the analogous compatibility of geometric Satake with central isogenies\footnote{In the current draft of \cite{FS}, there are two errors in the proof of \cite[Theorem~IX.6.1]{FS}. In particular, the last diagram in the proof is not Cartesian as claimed without making a restriction to the connected components indexed by $C$. Moreover, in the chain of isomorphisms given in the proof, one is missing a tensor product by a factor of $\phi_{\natural}(\ol{\mathbb{Q}}_{\ell})$ coming from the projection formula, as one can see by applying the claimed relationship in the case where $V$ is the trivial representation.}. Moreover, this map is natural in $V$ and $I$, and it follows that this identification is compatible with excursion operators \cite[Theorem~IX.6.1]{FS}. The claim follows.
\end{proof}
We now record some key consequences of Proposition \ref{prop: actisogeny}. Consider a supercuspidal parameter $\phi: W_{\mathbb{Q}_{p}} \ra \phantom{}^{L}G(\ol{\mathbb{Q}}_{\ell})$ and let $\ol{\phi}:  W_{\mathbb{Q}_{p}} \ra \phantom{}^{L}G(\ol{\mathbb{Q}}_{\ell}) \ra \phantom{}^{L}G'(\ol{\mathbb{Q}}_{\ell})$ be the composite parameter. The map $X_{\hat{G}} \ra X_{\hat{G'}}$ then induces a natural map of stacks $[\Spec{\ol{\mathbb{Q}}_{\ell}}/S_{\phi}] \ra [\Spec{\ol{\mathbb{Q}}_{\ell}}/S_{\ol{\phi}}]$ induced by the natural map $S_{\phi} \ra S_{\ol{\phi}}$ of centralizers. The induced map $C_{\phi} \ra C_{\ol{\phi}}$ of connected components lies over this natural map of classifying stacks. Given a representation $W \in \Rep_{\ol{\mathbb{Q}}_{\ell}}(S_{\ol{\phi}})$, we write $\widetilde{W}$ for the precomposition of $W$ with the map of centralizers $S_{\phi} \ra S_{\ol{\phi}}$. The pullback of the vector bundle corresponding to $W$ on $[\Spec{\ol{\mathbb{Q}}_{\ell}}/S_{\ol{\phi}}]$ is isomorphic to the bundle corresponding $\widetilde{W}$. In particular, the diagram \eqref{eqn: adjisodiagram} induces a commutative diagram 
\begin{equation*}
\begin{tikzcd}
\Rep_{\ol{\mathbb{Q}}_{\ell}}(S_{\ol{\phi}}) \arrow[dr] \arrow[d] & \\
\Rep_{\ol{\mathbb{Q}}_{\ell}}(S_{\phi}) \arrow[r] &  \mathrm{End}(\Dlis(\Bun_{G}^{C},\ol{\mathbb{Q}}_{\ell}))^{BW_{\mathbb{Q}_{p}}^{I}},
\end{tikzcd}    
\end{equation*}
where the horizontal and diagonal maps are as above. In particular, we can deduce the following claim as a consequence.
\begin{corollary}{\label{actcentralisog}}
Let $G' \ra G$ be a morphism of algebraic groups inducing an isomorphism of adjoint groups with induced map $\varphi: \Bun_{G'} \ra \Bun_{G}$. Given a supercuspidal parameter $\phi$ of $G$ with $\ol{\phi}$ the parameter of $G'$ induced by the central isogeny, and $W \in \Rep_{\ol{\mathbb{Q}}_{\ell}}(S_{\ol{\phi}})$ with $\widetilde{W} \in \Rep_{\ol{\mathbb{Q}}_{\ell}}(S_{\phi})$ the representation induced via precomposition with the induced map $S_{\phi} \ra S_{\ol{\phi}}$, we have, for all $A \in \Dlis(\Bun_{G}^{C},\ol{\mathbb{Q}}_{\ell})$, an isomorphism 
\[ \varphi_{\natural}\Act_{W}(\varphi^{*}(A)) \simeq \Act_{\tilde{W}}(A \otimes \phi_{\natural}(\ol{\mathbb{Q}}_{\ell})) \]
of sheaves in $\Dlis(\Bun_{G}^{C},\ol{\mathbb{Q}}_{\ell})$. 
\end{corollary}
We will use this in the last section to deduce the full strength of the Kottwitz conjecture for $\U_{n}$ from $\GU_{n}$.
\section{Proof of Theorem \ref{thm: introcompatibility}}{\label{s: proofofcomp}}

In this section, we prove Theorem \ref{thm: introcompatibility} when $G = \GU_{n}$. The proof proceeds by induction over odd $n$ for the groups $\GU_n$. When $n=1$, we have $\GU_n \cong \Res_{E/\Q_p}\Gm$ and hence compatibility is known by Theorem \ref{FSproperties} (1).

We suppose that $n$ is odd and at least $3$. Let $(\GU_n, \varrho, z)$ be an extended pure inner twist of $\GU^*_n$ corresponding to some $b_0 \in B(\GU^*_n)$. We assume that compatibility of the local Langlands correspondences is known for each extended pure inner form of $\GU^*_k$ for every odd $k<n$ 

Let $\phi \in \Phi(\GU^{*}_n)$ be an $L$-parameter. By Theorem \ref{thm: fullLLC} (2), we need to show that, for all $\phi$ and $\pi \in \Pi_{\phi}(\GU_n, \varrho)$, we have
\begin{equation}{\label{equalityofparams}}
    \phi^{ss} = \phi^{\mathrm{FS}}_{\pi},
\end{equation}
where $\phi^{\mathrm{FS}}_{\pi}$ is the semisimple parameter attached to $\pi$ by $\LLC^{\mathrm{FS}}_{\GU_n}$. Note that the $L$-packet $\Pi_{\phi}(\GU_n, \varrho)$ under the $\LLC_{\GU_n}$ correspondence is finite and tautologically satisfies one of the following:
\begin{enumerate}
    \item The packet $\Pi_{\phi}(\GU_n, \varrho)$ consists entirely of supercuspidal representations,\\
    \item The packet $\Pi_{\phi}(\GU_n, \varrho)$ contains no supercuspidal representations,\\
    \item The packet $\Pi_{\phi}(\GU_n, \varrho)$ contains both supercuspidal and non-supercuspidal representations.
\end{enumerate}

Before proving the theorem, let us recall a result by Moeglin that characterizes the packets consisting entirely of supercuspidal representations. Let $\phi$ be a discrete $L$-parameter of $\U_n$ and $\phi_E$ the corresponding $L$-parameter of $\GL_n(E)$ obtained by the base change map. In fact, we have $ \phi_E = \phi_{| \mathcal{L}_E} $ and an identification $ S^{\natural}_{\phi} = \mathrm{Cent}_{\GL_n(\C)}(\phi_E)^{\theta} $ where $\mathrm{Cent}_{\GL_n(\C)}$ denotes the centralizer in $\GL_n(\C)$ and $\theta$ is defined in the beginning of \S \ref{itm : theta}. We will need the notion of a "without gaps" parameter  ("sans trou" in French) and the construction of a subgroup $A(\phi)$ of $S^{\natural}_{\phi}$. 

We denote by $\phi_E[a]$ the isotypical component of the representation $\phi_E$, seen as a representation of $\SL_2(\mathbb{C})$, for the unique irreducible representation $\sigma_a$ of $\SL_2(\mathbb{C})$ of dimension $a$. For all a, $\phi_E[a]$ is a representation of $W_E$. We say that $\phi$ is without gaps if for all $a > 2$, $\phi_E[a]$ is a sub-representation of $\phi_E[a-2]$ as a representation of $W_E$.

We now explain the construction of $A(\phi)$.  Decompose $\phi_E$ into direct sum of irreducible representations and let  $(\rho, a)$ be a pair so that $ \rho \otimes \sigma_a $ is a sub-representation of $\phi_E$. We assume that there exists an integer $ 0 \leq b < a $ such that if $ b > 0 $ then $ \rho \otimes \sigma_b $ is also a sub-representation of $\phi_E$ and if $ b = 0 $ then $a$ is even. In particular, if $a = 1$ then there is no $b$ satisfying the assumption. Denote by $a_{-}$ the largest b verifying this property and denote by $ z_{(\rho, a)} $ the element in the centralizer of $\phi_E$ in $\GL_n(\C)$ whose $-1$-eigenspace is the sum $ \rho \otimes \sigma_a \oplus \rho \otimes \sigma_{a_{-}} $. We denote by $A(\phi)$ the subgroup of $S^{\natural}_{\phi}$ generated by these elements $ z_{(\rho, a)} $. 

\begin{lemma} \phantomsection \label{itm : supercuspidal parameter} (Moeglin)
    Let $\phi$ be a discrete $L$-parameter of $\U_n$. Then its corresponding $L$-packet contains only supercuspidal representations if and only if $\phi$ is \emph{supercuspidal} (i.e the $\SL_{2}(\mathbb{C})$-factor in the domain of $\phi$ acts trivially and $\phi$ does not factor through $\phantom{}^{L}M$ for $M$ any proper Levi subgroup of $\LL \GU_n$).
\end{lemma}
\begin{proof}
    Suppose first that $\phi$ is supercuspidal. Then $\phi$ is without gaps. Moreover, $\phi_E$ is a direct sum of pairwise non-isomorphic representations of the form $\rho \otimes \sigma_1$. Since the pairs $(\rho, 1)$ do not satisfy the assumption in the construction of the group $A(\phi)$, we see that the group $A(\phi)$ is trivial and by \cite[\S 8.4.4]{Moe}, the packet $\Pi_{\phi}(\U_n^*)$ of $\phi$ contains only supercuspidal representations.  
    
    Suppose now that the packet $\Pi_{\phi}(\U_n^*)$ corresponding to $\phi$ contains only supercuspidal representations. Then by \cite[\S 8.4.4]{Moe}, $\phi$ is without gaps and $A(\phi)$ is a sub-group of $\{ \Id, -\Id \}$. Indeed, if $A(\phi)$ is not a sub-group of $\{ \Id, -\Id \}$ then there exists a character $\epsilon$ of $S^{\natural}_{\phi}$ whose restriction to $\{ \Id, - \Id \}$ is trivial and whose restriction to $A(\phi)$ is not equal to $\epsilon_{alt}$. Hence the representation $\tau_{\epsilon}$ corresponding to $\epsilon$ in $\Pi_{\phi}(\U_n^*) $ is not supercuspidal (by \cite[Theorem 8.4.4]{Moe}), a contradiction. Let $\rho \otimes \sigma_a$ be a sub-representation of $\phi_E$ and it is enough to show that $a = 1$. 

    If $a > 2$ then $ \rho \otimes \sigma_{a-2} $ is a sub-representation of $\phi_E$ since $\phi$ is without gaps. Thus $a_{-} = a-1$ or $a_{-} = a-2$ and in the first case it is obvious that $ \rho \otimes \sigma_a \oplus \rho \otimes \sigma_{a_{-}} $ is not equal to $\phi_E$ and in the second case $ \rho \otimes \sigma_a \oplus \rho \otimes \sigma_{a_{-}} $ is not equal to $\phi_E$ by the parity of dimension (recall that $n$ is odd). Thus $ z_{(\rho, a)} \notin \{ \Id, -\Id \} $, a contradiction.

    If $a = 2$ then $a_{-} = 0$ or $a_{-} = 1$. If $a_{-} = 0$ then $ \rho \otimes \sigma_a \oplus \rho \otimes \sigma_{a_{-}} $ is not equal to $\phi_E$ by the parity of dimension and $ z_{(\rho, a)} \notin \{ \Id, -\Id \} $, a contradiction. If $a_{-} = 1$ then the condition $ z_{(\rho, a)} \in \{ \Id, -\Id \} $ implies that $ \phi_E = \rho \otimes \sigma_2 \oplus \rho \otimes \sigma_1 $ and $ A(\phi) = \{ \Id, -\Id \} $. However if $\tau$ is a supercuspidal representation in the packet $\Pi_{\phi}(\U_n^*)$ and $\epsilon_{\tau}$ is the corresponding character of $S^{\natural}_{\phi}$ then the restriction of $\epsilon_{\tau}$ to $\{ \Id, -\Id \}$ is trivial. It is once again a contradiction since by \cite[\S 8.4.4]{Moe}, we have $ \epsilon_{\tau}|_{A(\phi)} = \epsilon_{alt} $, where $\epsilon_{alt}$ is the alternating character of $A(\phi)$. 

\end{proof}

We now prove that, in each case, Equation \eqref{equalityofparams} is satisfied for all $\pi \in \Pi_{\phi}(\GU_n, \varrho)$.
\subsection{Case (1)}{\label{Case (1)}}
 This corresponds to the case in which the $L$-parameter is \emph{supercuspidal} (i.e the $\SL_{2}(\mathbb{C})$-factor in the domain of $\phi$ acts trivially and $\phi$ does not factor through $\phantom{}^{L}M$ for $M$ any proper Levi subgroup of $\LL \GU_n$). Indeed, Lemma \ref{itm : supercuspidal parameter} implies this for $L$-parameters of $\U_n$, and the $\GU_n$ case follows easily from the description of L-packets for $\GU_n$ in terms of those of $\U_n$, as given in \cite{BMN}. 

Compatibility in this case ultimately follows from the the results of \cite{BMN} on the Kottwitz conjecture for $\GU_{n}$. To introduce this, we let $\mu_{d}$ be the geometric dominant cocharacter of $\GU_n$ with weights $((1^{d},0^{n - d}),1)$ for $1 \leq d < n$, and let $b \in B(\GU_n,\mu_{d})$ be the unique basic element (which is independent of $d$). Consider the local shtuka datum $(\GU_n,b,\mu_{d})$. The element $b$ gives $J_b$ the structure of extended pure inner twist $(J_b, \varrho_b, z_b)$ of $\GU_n$ and hence the structure $(J_b, \varrho_b \circ \varrho, z+z_b)$ of an extended pure inner twist of $\GU^*_n$. Let $m_0 = \kappa(b_0)$. We recall that, for $\rho \in \Pi_{\phi}(J_{b}, \varrho_b \circ \varrho)$, we have that $\rho = \pi_{[I],m_0+1}$ for some $[I] \in \mc{P}(\{1, \dots, r\})/ \sim$, as in Section \ref{sss: galoisrepcomp}. We then have the following key result:
\begin{theorem}[{Bertoloni Meli -- Nguyen \cite[Theorem 6.1]{BMN}}]{\label{hardinput}}
For $(\GU_n,b,\mu_{d})$ the local Shtuka datum considered above and $\rho \in \Pi_{\phi}(J_{b}, \varrho_b \circ \varrho)$, we have an equality
\[ [R\Gamma^{\flat}_{c}(\GU_n,b,\mu_{d})[\rho]] = \sum_{\pi \in \Pi_{\phi}(\GU_n, \varrho)} [\pi \boxtimes \mathrm{Hom}_{S_{\phi}}(\delta_{\pi,\rho},r_{\mu_{d}}\circ \phi|_{W_{E}})]\]
in $K_{0}(\GU_n(\mathbb{Q}_{p}) \times W_{E})$, the Grothendieck group of $\GU_n(\mathbb{Q}_{p}) \times W_{E}$-representations. 
\end{theorem}
Now, we deduce the following consequence by Corollaries \ref{stdmucor} and \ref{extmucor}. 
\begin{corollary}{\label{explicitstdmuKottwitzConj}}
For $\pi_{[I],m_0+1} \in \Pi_{\phi}(J_{b}, \varrho_b \circ \varrho)$ as above, we have an equality 
\[ [R\Gamma^{\flat}_{c}(\GU_n,b,\mu_{1})[\pi_{[I], m_0+1}]] = \sum_{i = 1}^{r} [\pi_{[I \oplus \{i\}],m_0} \boxtimes {\phi_{i}}]  \]
in $K_{0}(\GU_n(\mathbb{Q}_{p}) \times W_{E})$ the Grothendieck group of finite length admissible $\GU_n(\mathbb{Q}_{p})$-representations admitting a smooth action of $W_{E}$-representations, where
\begin{itemize}
    \item 
     $r_{\mu} \circ \phi|_{W_E} := \bigoplus_{i = 1}^{r} \phi_{i}$
is a decomposition into irreducible representations $\phi_{i}$ of $W_{E}$,
    \item $\pi_{[I \oplus \{i\}],m_0}$ is the unique representation in $\Pi_{\phi}(\GU_n, \varrho)$ corresponding to the symmetric difference $[I \oplus \{i\}] \in \mathcal{P}(\{1,\ldots,r\})/\sim$ (see \S\ref{sss: galoisrepcomp}).
    \end{itemize}
More generally, for $1 \leq d \leq r$, we write $\Pi^{d}_{\phi}(\GU_{n},\varrho) \subset \Pi_{\phi}(\GU_{n},\varrho)$ for the subset of representations of the form $\pi_{[I \oplus J],m_0}$, for $[J] \in \mc{P}(\{1,\ldots,r\})/\sim$ represented by a subset $J \in \mc{P}(\{1,\ldots,r\})$ of cardinality $|d|$. We then have an equality 
\[ [R\Gamma^{\flat}_{c}(\GU_n,b,\mu_{1})[\pi_{[I],m_0 + 1}]] = \sum_{\pi_{[I \oplus J],m_0} \in \Pi^{d}_{\phi}(\GU_{n},\varrho)} [\pi_{[I \oplus J],m_0} \boxtimes \phi_{J}]  +  \sum_{\pi \notin \Pi^{d}_{\phi}(\GU_{n},\varrho)} [\pi \boxtimes \mathrm{Hom}_{S_{\phi}}(\delta_{\pi,\rho},r_{\mu_{d}}\circ \phi|_{W_{E}})] \]
in $K_{0}(\GU_n(\mathbb{Q}_{p}) \times W_{E})$, where $\phi_{J} := \bigotimes_{j \in J} \phi_{j}$. 
\end{corollary}
We leverage this result to deduce compatibility.
\begin{proposition}{\label{prop: case 1 compatibility}}
Let $\phi$ be a supercuspidal parameter of $\GU_n$ as above. Then, for any $\rho \in \Pi_{\phi}(J_{b},  \varrho_{b} \circ \varrho)$, the Fargues--Scholze correspondence is compatible with $\LLC_{\GU_n}$.
\end{proposition}
\begin{proof}
Fix $\rho \in \Pi_{\phi}(J_{b}, \varrho_{b} \circ \varrho)$ and let $\mu = \mu_{1}$. Corollary \ref{explicitstdmuKottwitzConj} tells us that $\phi_{i}$ for $i = 1,\ldots,r$ occurs as a subquotient of the cohomology of $R\Gamma^{\flat}_{c}(\GU_n,b,\mu_{1})[\rho]$. It follows, by \cite[Theorem~1.3]{Ko} (c.f. \cite[Lemma~3.10]{Ham}), that $\phi_{i}$ is an irreducible constituent of $r_{\mu} \circ {\phi_{\rho}^{\mathrm{FS}}}|_{W_E}$ for all $i = 1,\ldots,r$. This implies that $r_{\mu} \circ \phi|_{W_E}$ and $r_{\mu} \circ {\phi_{\rho}^{\mathrm{FS}}}|_{W_E}$ are equal in $K_{0}(W_{E})$. Since $r_{\mu} \circ \phi|_{W_E}$ and $r_{\mu} \circ \phi^{\mathrm{FS}}_{\rho}|_{W_E}$ are semisimple, we have an equality, $r_{\mu} \circ \phi|_{W_E} = r_{\mu} \circ {\phi_{\rho}^{\mathrm{FS}}}|_{W_E}$ ,
as conjugacy classes of $L$-parameters. Now we want to conclude that $\phi = \phi_{\rho}^{\mathrm{FS}}$. We have a cartesian diagram
\begin{equation*}
\begin{tikzcd}
\LL \GU_n \arrow[r] \arrow[d] &  \LL Z(\GU_n) \arrow[d]\\
 \LL \U_n \arrow[r] &  \LL Z(\U_n),
\end{tikzcd}    
\end{equation*}
where the maps are the obvious ones induced by duality. Hence, to show that $\phi = \phi^{\mathrm{FS}}_{\rho}$, we need only show that the central characters $\omega, \omega^{\mathrm{FS}}$ and induced unitary parameters $\ov{\phi}, \ov{\phi}^{\mathrm{FS}}_{\rho}$ coincide. Since both correspondences are compatible with the Langlands correspondence for their central characters ( \cite[Theorem~I.9.6 (iii)]{FS}, Theorem \ref{itm: local}), we have that $\omega = \omega^{\mathrm{FS}}$. From the above, we have that $r_{\mu} \circ \phi|_{W_E} = r_{\mu} \circ {\phi_{\rho}^{\mathrm{FS}}}|_{W_E}$. The representation $r_{\mu} \circ \phi|_{W_E}$ differs from $\ov{\phi}|_{W_E}$ by the character $p_2 \circ \omega|_{W_E}$, where $p_2$ is the projection of $\LL Z(\GU_n)$ onto the copy of $\C^{\times}$ corresponding to the similitude factor, and the analogous statement is also true for ${\phi_{\rho}^{\mathrm{FS}}}|_{W_E}$. Hence, we deduce that  $\overline{\phi}_{E} = \overline{\phi_{\rho}}_{E}^{\mathrm{FS}}.$
It now follows from \cite[Theorem~8.1 (ii)]{GGP} that this implies that $\overline{\phi} = \overline{\phi}_{\rho}^{\mathrm{FS}}$, as desired. 
\end{proof}

Finally, to deduce the compatibility for $\pi \in \Pi_{\phi_{\pi}}(\GU_n, \varrho)$, we apply Proposition \ref{prop: case 1 compatibility} to the extended inner twist of $\GU^*_n$ induced by $z-z_b$, so that $(J_b, \varrho_b\circ \varrho)$ becomes $(\GU_n, \varrho)$.
\subsection{Case (2)}
The main goal of this subsection is to prove the following, from which Case (2) of Theorem 1.1 follows.

\begin{proposition}{\label{prop: case 2 compatibility}}
Let $\pi\in \Pi(\GU_n)$ be a subquotient of a parabolic induction. Then the Fargues--Scholze correspondence for $\pi$ is compatible with $\LLC_{\GU_n}$.
\end{proposition}
\begin{proof}
Suppose $\pi$ is a sub-quotient of $I^{\GU_n}_P(\sigma)$ for some proper parabolic subgroup $P$ of $\GU_n$ with Levi subgroup $M$ and $\sigma \in \Pi(M)$.  By compatibility of the Fargues--Scholze correspondence and classical correspondence with parabolic induction (\cite[Theorem~I.9.6 (viii)]{FS}, Proposition \ref{parindcomp}), compatibility of $\pi$ will follow from compatibility of $\sigma$. The Levi subgroups of $\GU_n$ are of the form $\GU_k \times \Res_{E/\Q_p} G$ where $k<n$ is odd and $G$ is a product of general linear groups. Hence, compatibility follows from inductive assumption and compatibility of the Fargues--Scholze construction with products and in the $\GL_n$ case ( \cite[Theorem~I.9.6 (vi), (ix)]{FS} ).
\end{proof}
\subsection{Case (3)} In this case, $\Pi_{\phi}(\GU_n, \varrho)$ contains both supercuspidal and non-supercuspidal representations. We call such a parameter a \emph{mixed supercuspidal parameter}. Since $L$-packets are disjoint and supercuspidal representations are essentially tempered, it follows from Theorem \ref{itm: local} that $\phi$ is bounded modulo center. Hence, by \cite[\S8.4.4]{Moe}, it follows that $\phi$ is a discrete parameter with non-trivial $\SL_2$-factor. 
We will deduce compatibility by combining the results of the previous sections with a weaker description of the cohomology groups $R\Gamma^{\flat}_{c}(\GU_n,b,\mu_{d})[\rho]$ for $\rho$ satisfying  $\LLC_{\GU}(\rho) \in \Phi_2(\GU_n)$, as described in \cite{HKW}. In particular, we describe the projection of $R\Gamma^{\flat}_{c}(\GU_n,b,\mu_{d})[\rho]$ to  $K_{0}(\GU_n(\mathbb{Q}_{p}))_{\mathrm{ell}}$ of elliptic admissible finite length $\GU_n(\mathbb{Q}_{p})$-representations (defined to be the quotient of $K_0(\GU_n(\Q_p))$ by the span of the non-elliptic representations as in \cite[Appendix C]{HKW}). In particular, \cite[Theorem~1.0.2]{HKW} and Corollaries \ref{stdmucor} and \ref{extmucor} imply the following.
\begin{theorem}{\label{appliedHKW}}
For $\phi$ a mixed supercuspidal $L$-parameter, $\pi_{[I], m_0+1} \in \Pi_{\phi}(J_{b}, \varrho_b \circ \varrho)$ corresponding to $[I] \in \mc{P}(\{1, \dots, r\})/ \sim$, we have an equality 
\[ [R\Gamma^{\flat}_{c}(\GU_n,b,\mu_{1})[\pi_{[I], m_0+1}]] = \sum^r_{i=1} d_{i}[\pi_{[I \oplus \{i\}],m_0}]  \]
in $K_{0}(\GU_n(\mathbb{Q}_{p}))_{\mathrm{ell}}$, where if we write $\bigoplus_{i = 1}^{r} \phi_{i} = \phi|_{\mathcal{L}_{E}}$ as a direct sum of distinct irreducible representation of $\mathcal{L}_{E}$ then $d_{i} = \dim(\phi_{i})$. Moreover, if $1 \leq d \leq r$, we have an equality 
\[ [R\Gamma^{\flat}_{c}(\GU_n,b,\mu_{1})[\pi_{[I],m_0+1}]] = \sum_{\pi_{[I \oplus J],m_0} \in \Pi^{d}_{\phi}(\GU_{n},\varrho)} d_{J}\pi_{[I \oplus J],m_0}  +  \sum_{\pi \notin \Pi^{d}_{\phi}(\GU_{n},\varrho)} \pi \boxtimes \mathrm{Hom}_{S_{\phi}}(\delta_{\pi,\rho},r_{\mu_{d}}\circ \phi|_{W_{E}}), \]
in $K_{0}(\GU_n(\mathbb{Q}_{p}))_{\mathrm{ell}}$, where $\Pi^{d}_{\phi}(\GU_{n},\varrho)$ is defined as in Theorem \ref{explicitstdmuKottwitzConj} and $d_{J} := \prod_{j \in J} d_{i}$. 
\end{theorem}

With this in hand, we can finally prove the following which will conclude the proof of Theorem 1.1.
\begin{proposition}
Let $\phi$ be a mixed supercuspidal parameter then, for any $\pi \in \Pi_{\phi}(\GU_n, \varrho)$, we have that $\LLC_{\GU_n}$ is compatible with $\LLC^{\mathrm{FS}}_{\GU_n}$. 
\end{proposition}
\begin{proof}
By assumption (since $J_b$ will be quasi-split), there exists some $\rho_{nsc} \in \Pi_{\phi}(J_{b}, \varrho_b \circ \varrho)$ which is non-supercuspidal. We write $\rho_{nsc} = \pi_{[I],m_0+1}$ for some $[I] \in \mc{P}(\{1, \dots, r\})/ \sim$. By Proposition \ref{prop: case 2 compatibility}, we already know compatibility for any non-supercuspidal representation in $\Pi_{\phi}(\GU_{n},\varrho)$, and we need to show it for the supercuspidal representations. Given any such supercuspidal, we can write it as $\pi_{[I \oplus J],m_0}$ for some $J \in \mc{P}(\{1,\ldots,r\})$. Let $d = |J|$. We now apply Theorem \ref{appliedHKW} to $R\Gamma^{\flat}_{c}(\GU_{n},b,\mu_{d})[\rho_{nsc}]$, and denote the supercuspidal part by $R\Gamma^{\flat}_{c}(\GU_{n},b,\mu_{d})[\rho_{nsc}]_{sc}$. We view this as an element of $K_0(\GU_n(\Q_p))$ via the splitting $K_0(\GU_n(\Q_p))_{\mathrm{ell}} \to K_0(\GU_n(\Q_p))$ induced by the standard basis of $K_0(\GU_n(\Q_p))$ consisting of the irreducible admissible representations. Then we have in $K_{0}(\GU_{n}(\mathbb{Q}_{p}))$ that the supercuspidal $\pi_{[I \oplus J],m_0} \in \Pi(\GU_n, \varrho)$ occurs in the class of $R\Gamma^{\flat}_{c}(\GU_{n},b,\mu_{d})[\rho_{nsc}]_{sc}$ in $K_{0}(\GU_{n}(\mathbb{Q}_{p}))$. Now it follows, by \cite[Corollary~3.15]{Ham}, that any smooth irreducible representation of $\GU_n(\mathbb{Q}_{p})$ occurring in the cohomology of $R\Gamma^{\flat}_{c}(\GU_n,b,\mu_{d})[\rho_{nsc}]$ has Fargues--Scholze parameter equal to $\phi_{\rho_{nsc}}^{\mathrm{FS}}$ under the inner twisting $\phantom{}^{L}J_{b} \simeq \phantom{}^{L}G$. This implies a chain of equalities 
\[ \phi^{\mathrm{FS}}_{\pi_{[I \oplus J],m_0}} =  \phi_{\rho_{nsc}}^{\mathrm{FS}} = \phi^{\semis} \]
where the last equality follows by Proposition \ref{prop: case 2 compatibility}.
\end{proof}
\section{Applications}{\label{s: applications}}
In this section, we will showcase the power of Theorem \ref{thm: introcompatibility}. In particular, it allows us to combine what we know about the local Langlands correspondence for unitary groups together with the techniques from geometric Langlands introduced in Section \ref{sss: spectralaction}. This will result in the construction of eigensheaves attached to supercuspidal $L$-parameters of $\GU_{n}$ or $\U_{n}$, verifying Fargues' conjecture \cite[Conjecture~4.3]{Fa} for these groups. In particular, this is essentially equivalent to the strongest form of the Kottwitz conjecture for these groups and all (not necessarily minuscule) geometric dominant cocharacters $\mu$. 
\subsection{Unitary Similitude Groups}
In the previous section, for a basic local Shtuka datum $(\GU_n,b,\mu)$ we considered the $\rho$-isotypic part $R\Gamma^{\flat}_{c}(\GU_{n},b,\mu)[\rho]$ for $\rho \in \Pi_{\phi}(J_{b},\varrho_{b} \circ \varrho)$; however, for the geometric Langlands style arguments in this section it will be more convenient to work with the complex $R\Gamma_{c}(\GU_{n},b,\mu)[\rho]$. We note that $R\Gamma^{\flat}_{c}(\GU_{n},b,\mu)[\rho] \simeq R\Gamma_{c}(\GU_{n},b,\mu)[\rho]$ for any $\rho$ with supercuspidal Fargues--Scholze parameter, by Proposition \ref{flatnatural}. Throughout this section, we will implicitly combine this fact with Theorem \ref{thm: introcompatibility}, since we will only consider $\rho$ with supercuspidal $L$-parameter (= Fargues--Scholze parameter), and just speak about the complexes $R\Gamma_{c}(G,b,\mu)[\rho]$.
\subsubsection{The minuscule case}
To study the cohomology of the local Shimura varieties defined by unitary groups more carefully, we will invoke the following general result of Hansen.
\begin{theorem}{\cite[Theorem~1.1.]{Han}}{\label{concmiddledegree}}
Let $(G,b,\mu)$ be a basic local Shtuka datum with $E_{\mu}$ the reflex field of a minuscule dominant geometric cocharacter $\mu$, and let $\rho$ be a smooth irreducible representation of $J_{b}(\mathbb{Q}_{p})$. Suppose the following conditions hold: 
\begin{enumerate}
    \item The spaces $(\Sht(G,b,\mu)_{K})_{K \subset G(\mathbb{Q}_{p})}$ occur in the basic uniformization at $p$ of a global Shimura variety in the sense of \cite[Definition~3.2]{Han}.
    \item The Fargues--Scholze parameter $\phi_{\rho}^{\mathrm{FS}}: W_{\mathbb{Q}_{p}} \rightarrow ^{L}G(\overline{\mathbb{Q}}_{\ell})$ is supercuspidal. 
\end{enumerate}
Then the complex $R\Gamma_{c}(G,b,\mu)[\rho]$ is concentrated in degree $0$ (middle degree under our normalizations).
\end{theorem}
Now, if we consider the case where $G = \GU_{n}$, $\mu = \mu_{d} = ((1^{d},0^{n - d}),1)$, and $b \in B(G, \mu)$ is the unique basic element. We let $\rho \in \Pi(J_{b})$ be a representation with supercuspidal $L$-parameter $\phi$. We know, by Theorem \ref{thm: introcompatibility}, that $\rho$ has supercuspidal Fargues--Scholze parameter. Therefore, if we can show that the local Shimura variety $\Sht(G,b,\mu)_{\infty}$ occurs in the basic uniformization of a global Shimura variety, we can conclude the complex $R\Gamma_{c}(G,b,\mu)[\rho]$ is concentrated in degree $0$ by Theorem \ref{concmiddledegree}. Uniformization is defined with respect to our fixed isomorphism $\iota_p: \overline{\mathbb{Q}}_{p} \xrightarrow{\sim} \mathbb{C}$. Following \cite[Section~2.1]{Mo}, we can consider the global unitary similitude group $\mathbf{G} := \mathbf{G}\mathbf{U}(d,n - d)/\mathbb{Q}$ defined by the natural Hermitian form of signature $(d,n - d)$ over $\R$ such that there exists a global Shimura datum  $(\mathbf{G},X)$ satisfying the following conditions:
\begin{enumerate}
    \item $(\mathbf{G},X)$ is of PEL type. 
    \item We have an isomorphism: $\mathbf{G}_{\mathbb{Q}_{p}} \simeq G$.
    \item The composition
    \[ \mathbb{G}_{m.\mathbb{C}} \hookrightarrow \prod_{\Gal(\mathbb{C}/\mathbb{R})} \mathbb{G}_{m,\mathbb{C}} \xrightarrow{X_{\mathbb{C}}} \mathbf{G}_{\mathbb{C}} \]
    is in the same conjugacy class as the geometric cocharacter $\mu_{d}$ after applying the isomorphism $\iota^{-1}_p$.
\end{enumerate}
Since the unitary group $G$ was defined with respect to $E/\mathbb{Q}_{p}$ an unramified extension it follows that this satisfies the conditions necessary to apply the results of \cite[Section~6]{RZ} (See also \cite[Theorem~3.4]{Han}) which in particular implies the desired uniformization condition. Combining this with Theorem \ref{hardinput}, we can deduce the following. 
\begin{theorem}
Let $\phi$ be a supercuspidal $L$-parameter and $\rho \in \Pi_{\phi}(J_b, \varrho_{b} \circ \varrho)$. We let $\mu_{d} = ((1^{d},0^{n - d}),1)$ as in Section 3.1 and $b \in B(\GU_{n},\mu_{d})$ be the unique basic element. We have an isomorphism 
\[ R\Gamma_{c}(\GU_n,b,\mu_{d})[\rho] = \bigoplus_{\pi \in \Pi_{\phi}(\GU_n, \varrho)} \pi \boxtimes \mathrm{Hom}_{S_{\phi}}(\delta_{\pi,\rho},r_{\mu_{d}}\circ \phi|_{W_{E}}) \]
of complexes of $G(\mathbb{Q}_{p}) \times W_{E}$-modules.
\end{theorem} 
\begin{proof}
By Theorem \ref{hardinput}, we know that the above formula holds in the Grothendieck group of $\GU_{n}(\mathbb{Q}_{p}) \times W_{E}$-modules. By the above discussion, Theorem \ref{thm: introcompatibility}, and Theorem \ref{concmiddledegree}, it follows that $R\Gamma_{c}(\GU_n,b,\mu_{d})[\rho]$ is concentrated in degree $0$. Thus, $R\Gamma_{c}(\GU_n,b,\mu_{d})[\rho]$ is an iterated extension of $G(\mathbb{Q}_{p}) \times W_{E}$-representations of the desired form. If $\omega_{\rho}$ denotes the central character of $\rho$ then $R\Gamma_{c}(\GU_{n},b,\mu_{d})[\rho]$ is valued in a complex of smooth representations with central character $\omega_{\rho}$. However, since all the representations occurring are supercuspidal since they have supercuspidal $L$-parameter, it follows that this extension must split because supercuspidals are injective/projective in the category of smooth representations with fixed central character. The claim follows.
\end{proof}
In particular, we will be interested in the following consequence, as in Corollary \ref{explicitstdmuKottwitzConj}. 
\begin{theorem}{\label{thm: kottwitzconjreps}}
For $\pi_{[I],m_0+1} \in \Pi_{\phi}(J_{b}, \varrho_b \circ \varrho)$ corresponding to $I = \mc{P}(\{1,\ldots,r\})/\sim$, we have an isomorphism
\[ R\Gamma_{c}(\GU_n,b,\mu_{1})[\pi_{[I],m_0+1}] = \bigoplus_{i = 1}^{r} \pi_{[I \oplus \{i\}],m_0} \boxtimes {\phi_{i}}  \]
of complexes of $\GU_n(\mathbb{Q}_{p}) \times W_{E}$-modules, where
\begin{itemize}
    \item 
     $r_{\mu} \circ \phi|_{W_E} := \bigoplus_{i = 1}^{r} \phi_{i}$
is a decomposition into irreducible representations $\phi_{i}$ of $W_{E}$,
    \item $\pi_{[I \oplus \{i\}],0}$ is the unique representation in $\Pi_{\phi}(\GU_n, \varrho)$ corresponding to the symmetric difference $[I \oplus \{i\}] \in \mathcal{P}(\{1,\ldots,r\})/\sim$ (see \S\ref{sss: galoisrepcomp}).
    \end{itemize}
More generally, for $1 \leq d \leq r$, we write $\Pi^{d}_{\phi}(\GU_{n},\varrho) \subset \Pi_{\phi}(\GU_{n},\varrho)$ for the subset of representations of the form $\pi_{[I \oplus J], m_0}$, for $[J] \in \mc{P}(\{1,\ldots,r\})/\sim$ represented by a subset $J \in \mc{P}(\{1,\ldots,r\})$ of cardinality $d$. We then have an isomorphism
\[ R\Gamma_{c}(\GU_n,b,\mu_{d})[\pi_{[I],m_0+1}] \simeq \bigoplus_{\pi_{[I \oplus J],m_0} \in \Pi^{d}_{\phi}(\GU_{n},\varrho)} \pi_{[I \oplus J],m_0} \boxtimes \phi_{J}  \oplus  \bigoplus_{\pi \notin \Pi^{d}_{\phi}(\GU_{n},\varrho)} \pi \boxtimes \mathrm{Hom}_{S_{\phi}}(\delta_{\pi,\rho},r_{\mu_{d}}\circ \phi|_{W_{E}}) \]
where $\phi_{J} := \bigotimes_{j \in J} \phi_{j}$. 
\end{theorem}
Now we will combine these strong forms of the Kottwitz conjecture with the spectral action to deduce the analogous results for non-minuscule $\mu$. 
\subsubsection{The non-minuscule case}
We recall that, if write the based-changed parameter $\phi_{E}$ as $\phi_{E} \simeq \bigoplus_{i = 1}^{r} \phi_{i}$, for $\phi_{i}$-distinct irreducibles, then we have an isomorphism: 
\[ S_{\phi} \simeq (\mathbb{Z}/2\mathbb{Z})^{r - 1} \times  \C^{\times}. \]
We now use the notation for $X^*(S^{\natural}_{\phi})$ established in \S\ref{sss: galoisrepcomp}. We write $b_{m}$ for the basic element of $B(\GU_n)$ with $\kappa(b_{m}) = m$. In what follows, we will, for $I \in \mc{P}(\{1,\ldots,r\})$ with image $[I] \in \mc{P}(\{1,\ldots,r\})/\sim$, abuse notation and conflate the representation $\pi_{[I],m}$ with the sheaf $j_{b_{m}!}(\pi_{[I],m}) \in \Dlis(\Bun_{G},\overline{\mathbb{Q}}_{\ell})$. We consider the irreducible characters $\ttau_{[I],m}$ of $S_{\phi}$, and let $\Act_{\ttau_{[I],m}}$ be the $\Act$-functor attached to it, as defined in Section \ref{sss: spectralaction}.  We have the following. 
\begin{proposition}{\label{ActgenGUn}}
For fixed $I,J \in \mc{P}(\{1,\dots,r\})$ and $m \in \Z$, we have isomorphisms of $J_{b_m}(\mathbb{Q}_{p})$-modules
\[ \Act_{\ttau_{[J],1}}(\pi_{[I],m+1}) \simeq \pi_{[I \oplus J],m} \]
\[\Act_{\ttau_{[J],0}}(\pi_{[I],m}) \simeq \pi_{[I \oplus J],m},   \]
and an isomorphism 
\[ \Act_{\ttau_{[J],-1}}(\pi_{[I],m}) \simeq \pi_{[I \oplus J],m+1} \]
of $J_{b_{m+1}}(\mathbb{Q}_{p})$-modules.
\end{proposition} 
\begin{proof}
By combining the first part of Theorem \ref{thm: kottwitzconjreps}, Theorem \ref{thm: RGammavsAct}, Lemma \ref{actirred}, and Schur's lemma (using that the $\phi_{i}$ are distinct irreducible representations), we deduce an isomorphism: 
\[ \Act_{\tilde{\tau}_{j,1}}(\pi_{[I],m + 1}) \simeq \pi_{[I \oplus \{j\}],m}, \]
for all $m \in \mathbb{Z}$ and $j \in \{1,\ldots,r \}$. We now use that the $\Act$-functors are monoidal. In particular, note that $\ttau_{j,1}^{-1} \simeq \ttau_{j,-1}$, and so we have an isomorphism
\[ \Act_{\ttau_{j,-1}}(\pi_{[I],m}) \simeq \Act_{\ttau_{j,-1}} \circ \Act_{\ttau_{j,1}}(\pi_{[I],m+1}) \simeq  \Act_{\mathbf{1}}(\pi_{[I \oplus \{j\}],m+1}) \simeq \pi_{[I \oplus \{j\}],m+1}, \] 
of $J_{b_{m + 1}}(\mathbb{Q}_{p})$-representations, where $\mathbf{1}$ denotes the trivial representation. It follows that it suffices to show that $\Act_{\ttau_{[J],0}}(\pi_{[I],m}) \simeq \pi_{[I \oplus J],m}$ for all $J \subset \{1,\ldots,r\}$, by applying $\Act_{\ttau_{[J],0}}(-)$ to the previous two isomorphisms, and using the monoidal property again.

To do this note that, given two elements $j_{1},j_{2} \in \{1,\ldots,r\}$, we can use that $\ttau_{[j_{1} \oplus j_{2}],0} \simeq \ttau_{j_{1},1} \otimes \ttau_{j_{2},-1}$ and the monoidal property to deduce that
\[ \Act_{\ttau_{[j_{1} \oplus j_{2}],0}}(\pi_{I,m}) \simeq \pi_{[I \oplus \{j_{1},j_{2}\}],m}, \]
for $j_{1},j_{2}$ distinct elements.
Similarly, for any set $J$ such that $|J|$ is even, we deduce that 
\[ \Act_{\ttau_{[J],0}}(\pi_{I,m}) \simeq \pi_{[I \oplus J],m}. \]
We write $I_{\mathrm{odd}} \subset \{1,\ldots,r\}$ for the subsets for which the $d_{i}$ are odd. As discussed in \S 2.2.4, $J$ and $J \oplus I_{\mathrm{odd}}$ give two representatives of $[J]$. Moreover, since $n$ is odd, $I_{\mathrm{odd}}$ has odd parity, and it follows that $J$ and $J \oplus I_{\mathrm{odd}}$ have different parities. Therefore, we have deduced that $\Act_{\ttau_{[J],0}}(\pi_{I,m}) \simeq \pi_{[I \oplus J],m}$ for all subsets $J \subset \{1,\ldots,r\}$ and $m \in \mathbb{Z}$, as desired.
\end{proof}
We now determine the values of $\Act_{W}$ for all $W \in \Rep_{\ol{\mathbb{Q}}_{\ell}}(S_{\phi})$.  
\begin{proposition}{\label{prop: Allactval}}
Let $I,J \subset \{1,\ldots,r\}$ be two index sets and $m_1,m_2 \in \mathbb{Z}$. Then we have an isomorphism
\[ \Act_{\ttau_{[J],m_1}}(\pi_{[I],m_2}) \simeq \pi_{[I \oplus J],m_2 - m_{1}} \]
of $J_{b_{m_{2} - m_{1}}}(\mathbb{Q}_{p})$-representations. 
\end{proposition}
\begin{proof}
The case where $m_{1} = 0$ follows from the previous Proposition. We assume $m_{1} > 0$, and proceed by induction on $m_{1}$. The base case of $m_{1} = 1$ is Proposition \ref{ActgenGUn}.  For the inductive step, using the monoidal property of $\Act$-functors we have an isomorphism  
\[ \Act_{\ttau_{[J],m_1}}(\pi_{[I],m_2}) \simeq \Act_{\ttau_{\emptyset,1}} \circ \Act_{\ttau_{[J],m_1 - 1}}(\pi_{[I],m_{2}}) \]
which we can rewrite using the inductive hypothesis as 
\[ \Act_{\ttau_{\emptyset,1}}(\pi_{[I \oplus J],m_{2} - m_{1} + 1}) \simeq \pi_{[I \oplus J],m_{2} - m_{1}} \]
as desired, where the previous isomorphism follows from Proposition \ref{ActgenGUn}. The case that $m_{1} < 0$ follows similarly, using the second part of Proposition \ref{ActgenGUn}. 
\end{proof}
We now summarize the consequences of Proposition \ref{prop: Allactval} for the non-minuscule shtuka spaces a little more explicitly. Let $(\GU^*_{n},b_{m_1},b_{m_2},\mu)$ be a basic local shtuka datum for $\GU^*_{n}$, where $\mu$ is a geometric dominant cocharacter with reflex field $E_{\mu}$ such that $\mu^{\flat} = m_1 - m_2$. For a finite index set $I \subset \{1,\ldots,r\}$ and a supercuspidal $L$-parameter $\phi$, we write $r_{\mu} \circ \phi|_{W_{E_{\mu}}} = \bigoplus_{j = 1}^{k} W^{\mu}_{j} \boxtimes \sigma_{j}^{\mu}$, viewed as a representation of $S_{\phi} \times W_{E_{\mu}}$. We can write the representations $W_{j}^{\mu} = \ttau_{[I^{\mu}_{j}],m_1 - m_2}$ for some finite index sets $I^{\mu}_{j} \subset \{1,\ldots,r\}$ and $j \in \{1,\ldots,k\}$. Then, by combining Lemma \ref{shimhecke}, Theorem \ref{thm: RGammavsAct}, and Proposition \ref{prop: Allactval}. We deduce the following generalization of Theorem \ref{thm: kottwitzconjreps}.
\begin{theorem}{\label{nonmindescr}}
For an arbitrary index set $I \subset \{1,\ldots,r\}$, we consider the representations $\pi_{[I],m_1} \in \Pi_{\phi}(J_{b_{m_1}}, \varrho_{b_{m_1}})$ and $\pi_{[I],m_2} \in \Pi_{\phi}(J_{b_{m_2}}, \varrho_{b_{m_2}})$. Then we have an isomorphism
\[ R\Gamma_{c}(\GU^*_n,b_{m_1},b_{m_2},\mu)[\pi_{[I],m_1}] \simeq  \bigoplus_{j = 1}^{k}  \pi_{[I \oplus I^{\mu}_{j}],m_2} \boxtimes {\sigma_{j}^{\mu}}  \]
of $J_{b_{m_2}}(\mathbb{Q}_{p}) \times W_{E_{\mu}}$-modules, and an isomorphism 
\[ R\Gamma_{c}(\GU^*_n,b_{m_1},b_{m_2},\mu)[\pi_{[I],m_2}] \simeq  \bigoplus_{j = 1}^{k}  \pi_{[I \oplus I^{\mu}_{j}], m_1} \boxtimes (\sigma_{j}^{\mu})^{\vee} \]
of $J_{b_{m_1}}(\mathbb{Q}_{p}) \times W_{E_{\mu}}$-modules. 
\end{theorem}
Thus, modulo understanding the representation-theoretic question of decomposing $r_{\mu} \circ \phi$ as a $S_{\phi} \times W_{E_{\mu}}$-representation, this completely describes the contributions of representations with supercuspidal $L$-parameters to the shtuka spaces. We now want to organize our result in a slightly nicer form. In particular, the datum of all these isomorphisms can be conveniently organized into an eigensheaf attached to the parameter $\phi$. 
\subsubsection{Fargues' Conjecture}
We start with the central definition. 
\begin{definition}{\label{defeigsheaf}}
For $G/\mathbb{Q}_{p}$ any connected reductive group, given an $L$-parameter $\phi: W_{\mathbb{Q}_{p}} \rightarrow \phantom{}^{L}G(\ol{\mathbb{Q}}_{\ell})$, we say a sheaf $\mathcal{G}_{\phi} \in \Dlis(\Bun_{G},\ol{\mathbb{Q}}_{\ell})$ is a Hecke eigensheaf with eigenvalue $\phi$ if, for all $V \in \Rep_{\ol{\mathbb{Q}}_{\ell}}(\phantom{}^{L}G^{I})$, we are given isomorphisms
\[ \eta_{V,I}: T_{V}(\mathcal{G}_{\phi}) \simeq \mathcal{G}_{\phi} \boxtimes r_{V} \circ \phi \]
of sheaves with continuous $W_{\mathbb{Q}_{p}}^{I}$-action, that are natural in $I$ and $V$, and compatible with compositions of Hecke operators.
\end{definition}
We will be interested in constructing for $G=\GU^*_{n}$ the eigensheaves attached to supercuspidal $L$-parameters $\phi$, as conjectured by \cite[Conjecture~4.4]{Fa}. In particular, we consider the following sheaf.
\begin{definition}
We set $\mathcal{G}_{\phi} \in \Dlis(\Bun_{G},\overline{\mathbb{Q}}_{\ell})$ to be the lisse-\'etale sheaf supported on $B(G)_{\bas}$, with stalk $\mathcal{G}_{\phi}|_{\Bun_{G}^{b}} \in \Dlis(\Bun_{G}^{b},\ol{\mathbb{Q}}_{\ell})$ isomorphic to  $\bigoplus_{\rho \in \Pi_{\phi}(J_{b}, \varrho_b)} \rho$, the members of the $L$-packet over the extended pure inner form $J_{b}$.
\end{definition}
We want to write this in a way that makes the Hecke eigensheaf property transparent from the work we did above. Let $k(\phi)_{\mathrm{reg}}$ be sheaf given by pulling back the sheaf on $[\Spec(\ol{\mathbb{Q}}_{\ell})/S_{\phi}]$ corresponding to the regular representation along $C_{\phi} \ra [\Spec(\ol{\mathbb{Q}}_{\ell})/S_{\phi}]$. We have the following proposition.
\begin{proposition}
Let $\pi_{0,\emptyset} \in \Pi_{\phi}(G,\id)$ be the representation corresponding to the trivial representation under the refined local Langlands correspondence. Then we have an isomorphism 
\[ k(\phi)_{\mathrm{reg}} \star \pi_{0,\emptyset} \simeq \mathcal{G}_{\phi} \] 
of sheaves on $\Dlis(\Bun_{G},\ol{\mathbb{Q}}_{\ell})$. 
\end{proposition}
\begin{proof}
The regular representation of $S_{\phi}$ breaks up as an infinite direct sum over $\ttau_{[I],m}$ for $[I] \in \mc{P}\{1,\ldots,r \}/\sim$ and $m \in \mathbb{Z}$. Therefore, we have an isomorphism:
\[ k(\phi)_{\mathrm{reg}} \star \pi_{0,\emptyset} \simeq \bigoplus_{m \in \mathbb{Z}} \quad \bigoplus_{[I] \in \mc{P}\{1,\ldots,r\}/\sim}  \Act_{\ttau_{[I],m}}(\pi_{0,\emptyset}) \]
However, applying Proposition \ref{prop: Allactval}, we get that this is isomorphic to 
\[ k(\phi)_{\mathrm{reg}} \star \pi_{0,\emptyset} \simeq \bigoplus_{m \in \mathbb{Z}} \quad \bigoplus_{[I] \in \mc{P}\{1,\ldots,r\}/\sim}  \pi_{[I],-m} \]
which is precisely the sheaf $\mathcal{G}_{\phi}$. 
\end{proof}
With this in hand, we can easily deduce the desired theorem.
\begin{theorem}{\label{thm: GUneigensheaf}}
For a supercuspidal $L$-parameter $\phi$, $\mathcal{G}_{\phi}$ is a Hecke eigensheaf on $\Bun_{\GU_{n}}$ with eigenvalue $\phi$.
\end{theorem}
\begin{proof}
The idea is the same as \cite[Proposition~7.4]{AL}. For any $V \in \Rep_{\ol{\mathbb{Q}}_{\ell}}(\phantom{}^{L}G^{I})$ for some finite index set $I$, the spectral action gives a natural isomorphism
\[ T_{V}(k(\phi)_{\mathrm{reg}} \star \pi_{0,\emptyset}) \simeq  C_{V} \star k(\phi)_{\mathrm{reg}} \star \pi_{0,\emptyset} \]
where $C_{V}$ is the vector bundle on $X_{\hat{G}}$ with $W_{\mathbb{Q}_{p}}^{I}$-action given by $r_{V}$, as in Section \ref{sss: spectralaction}. However, since the spectral action is monoidal and respects the $W_{\mathbb{Q}_{p}}^{I}$-action, this is the same as 
\[ C_{V} \star k(\phi)_{\mathrm{reg}} \star \pi_{0,\emptyset} \simeq (r_{V} \circ \phi \otimes_{C_{\phi}} k(\phi)_{\mathrm{reg}}) \star \pi_{0,\emptyset}  \simeq r_{V} \circ \phi \boxtimes k(\phi)_{\mathrm{reg}} \star \pi_{0,\emptyset} \simeq r_{V} \circ \phi \boxtimes \mathcal{G}_{\phi} \]
as desired, which will be natural in $V$ and $I$ and satisfy the desired compatibilities by the analogous compatibilities for the spectral action \cite[Corollary~IX.3]{FS}. Here for the second isomorphism, we have used that the Fargues-Scholze parameter of $\pi_{0,\emptyset}$ is $\phi$. In particular, this implies that the spectral action of $C_{V} \star k(\phi)_{\mathrm{reg}}$ factors over the base-change to the localization of $C_{\phi}$ at the closed point defined by $\phi$, since the endomorphisms of $\pi_{0,\emptyset} \simeq \mathcal{O}_{X_{\hat{G}}} \star \pi_{0,\emptyset}$ defined by $f \in \mathcal{O}_{X_{\hat{G}}} \setminus \mf{m}_{\phi}$ are isomorphisms, where $\mf{m}_{\phi}$ is the maximal ideal defined by $\phi$, allowing us to see that we get the correct $W_{\mathbb{Q}_{p}}$-action.
\end{proof}

\subsection{Unitary Groups}{\label{ss: unitarygroupapps}}
In this section, we deduce Fargues' conjecture and the full strength of the Kottwitz conjecture for unitary groups $\U_{n}$, using Theorem \ref{FSproperties} (7), and Corollary \ref{actcentralisog} with respect to the central isogeny $\U_{n} \ra \GU_{n}$. In general, it seems a bit tricky to propagate Fargues' conjecture along such central isogenies. However, in the case of odd unitary groups, we can take advantage of the special property that restriction along this map induces a bijection of $L$-packets.  

As for the similitude groups, our starting point is compatibility of the correspondences. Fix $(\U_n, \varrho, z)$ an extended pure inner twist of $\U^*_n$ and lift to an extended pure inner twist $(\GU^*_n, \tilde{\varrho}, \tilde{z})$ (on the level of cocycles, $\tilde{z}$ is given by composing $z$ with the inclusion $\U^*_n \to \GU^*_n$). Note that $(\GU^*_n, \tilde{\varrho}, \tilde{z})$ will give the trivial class in $B(\GU^*_n)$.
\begin{theorem}{\label{Uncompatibility}}
The diagram 
\begin{equation*}
\begin{tikzcd}[ampersand replacement=\&]
            \Pi(\U_{n})  \ar[rr, "\LLC_{\U_{n}}"] \arrow[drr, swap,"\LLC^{\mathrm{FS}}_{\U_{n}}"] \& \&   \Phi(\U_{n}) \ar[d,"(-)^{ss}"] \\
            \& \& \Phi^{ss}(\U_{n})
        \end{tikzcd}
\end{equation*}
commutes. 
\end{theorem}
\begin{proof}
Let $\pi \in \Pi(\U_n)$. We claim there is a lift of $\pi$ to some $\tilde{\pi} \in \Pi(\GU_n)$. Indeed, such a $\tilde{\pi}$ is given by a pair $(\pi, \chi)$ where $\pi$ is as before and $\chi$ is a character of $Z_{\GU_n}(\Q_p)$ lifting the central character, $\omega_{\pi}$, of $\pi$. The existence of such a $\chi$ follows from the exactness of the Pontryagin dual functor.

Now, it follows by \cite[\S2.3]{BMN} that the $L$-parameter of $\pi$ under $\LLC_{\mathrm{U}_{n}}$ is equal to the composition 
\[ W_{\mathbb{Q}_{p}} \times \SL(2,\ol{\mathbb{Q}}_{\ell}) \xrightarrow{\phi_{\tilde{\pi}}} \phantom{}^{L}\GU_{n}(\ol{\mathbb{Q}}_{\ell}) \ra \phantom{}^{L}\U_{n}(\ol{\mathbb{Q}}_{\ell}) \]
as a conjugacy class of parameters. Similarly, since the map $\U_{n} \ra \GU_{n}$ induces an isomorphism on adjoint groups, it follows by \cite[Theorem~I.9.6 (v)]{FS} that $\phi_{\pi}^{\mathrm{FS}}$ is given by $\phi_{\tilde{\pi}}^{\mathrm{FS}}$ composed with the induced map $\phantom{}^{L}\GU_{n}(\ol{\mathbb{Q}}_{\ell}) \ra \phantom{}^{L}\U_{n}(\ol{\mathbb{Q}}_{\ell})$. Therefore, the result immediately follows from Theorem \ref{thm: introcompatibility} for $G = \GU_{n}$. 
\end{proof}
Now, as in the previous section, we use this to deduce results for the cohomology of some local shtuka spaces attached to $\U_{n}$. 
\subsubsection{The Minuscule Case}
We continue with the fixed extended pure inner twist $(\U_n, \varrho, z)$ of $\U^*_n$. We note that in this case $X_*(Z(\hat{\U}_{n})^{\Gamma}) \simeq \mathbb{Z}/2\mathbb{Z}$. The relevant minuscule basic shtuka spaces are given by
\[ \Sht(\U_{n},b_{d},\mu_{d})_{\infty} \]
for $0 \leq d \leq n$, $\mu_{d} = (1^{d},0^{n -d})$, and $b_{d} \in B(\U_n,\mu_{d})$ the unique basic element. We note that, depending on if $d$ is even or odd, $b_{d} \in B(\U_n)$ is the trivial element or the unique non-trivial basic element which we denote simply by $b$. Using the same Hermitian form as the previous section, we can find a global unitary group $\mathbf{G} := \mathbf{U}(d,n - d)/\mathbb{Q}$ with signature $(d,n - d)$ over $\mathbb{R}$ and a Shimura datum $(\mathbf{G},X)$ such that
\begin{enumerate}
    \item $(\mathbf{G},X)$ is of abelian type. 
    \item We have an isomorphism: $\mathbf{G}_{\mathbb{Q}_{p}} \simeq \U_n$.
    \item The composition
    \[ \mathbb{G}_{m.\mathbb{C}} \hookrightarrow \prod_{\Gal(\mathbb{C}/\mathbb{R})} \mathbb{G}_{m,\mathbb{C}} \xrightarrow{X_{\mathbb{C}}} \mathbf{G}_{\mathbb{C}} \]
    is in the same conjugacy class as the geometric cocharacter $\mu_{d}$ after applying the isomorphism $\iota^{-1}_p$.
\end{enumerate}
We now assume that $p > 2$ for the rest of the section. Since the unitary group $\U_n$ was defined with respect to $E/\mathbb{Q}_{p}$ an unramified extension and $p > 2$  it follows that this satisfies the conditions necessary to apply the results of \cite[Theorem~1.2]{Shen} (See also \cite[Theorem~D]{LiHue}). This in particular allows us to deduce the relevant basic uniformization result necessary to apply Theorem \ref{concmiddledegree}, which when combined with Theorem \ref{Uncompatibility} gives us the following. 
\begin{proposition}{\label{unmiddledegree}}
For $\phi$ a supercuspidal $L$-parameter, let $\rho \in \Pi_{\phi}(J_{b_d}, \varrho_{b_d} \circ \varrho)$ be a smooth irreducible representation with $L$-parameter $\phi$. Then 
\[ R\Gamma_{c}(\U_{n},b_{d},\mu_{d})[\rho] \]
is concentrated in degree $0$.
\end{proposition}
We now want to combine this with the results of \cite{HKW}. Fix a supercuspidal parameter $\phi$ and consider the representation
\[ r_{\mu_{1}} \circ \phi|_{W_{E}} = \oplus_{i = 1}^{r} \phi_{i}. \]
Let $d_{i} := \dim(\phi_{i})$. Now, we recall that $S_{\phi} \simeq (\mathbb{Z}/2\mathbb{Z})^{r}$, where the center $Z(\hat{G})^{\Gamma} \simeq \mathbb{Z}/2\mathbb{Z}$ embeds diagonally. We consider the representations $\tau_{I}$ for $I \subset \{1,\ldots,r\}$, which is the sign representation on the factors indexed by $I$ and is trivial on the other factors. The refined local Langlands correspondence gives us a bijection between the representations $\Pi_{\phi}(\U_n, \varrho)$ (resp. $\Pi_{\phi}(\U_{n}, \varrho_{b_d} \circ \varrho)$) and the representations $\tau_{I}$ for $|I|  = \kappa(z) \mod 2$ (resp. $|I| = \kappa(z) + d \mod 2$). We write $\pi_{I}$ for the representation associated to $I \subset \{1,\ldots,r\}$. In what follows, we will implicitly regard $\pi_{I}$ as a sheaf on either $\Dc(\U_n(\mathbb{Q}_{p}),\ol{\mathbb{Q}}_{\ell}) \simeq \Dlis(\Bun_{\U_n}^{1},\ol{\mathbb{Q}}_{\ell}) \subset \Dlis(\Bun_{\U_n},\ol{\mathbb{Q}}_{\ell})$ if $|I|$ is even or $\Dc(J_{b}(\mathbb{Q}_{p}),\ol{\mathbb{Q}}_{\ell}) \simeq \Dlis(\Bun_{\U_n}^{b},\ol{\mathbb{Q}}_{\ell}) \subset \Dlis(\Bun_{\U_n},\ol{\mathbb{Q}}_{\ell})$ if $|I|$ is odd.

We note that the representation $\delta_{\pi,\pi_{I}} = \tau_{i}$ for $i = 1,\ldots,r$ if and only if $\pi = \pi_{I \oplus \{i\}}$. It follows that, if we apply the results of Hansen--Kaletha--Weinstein \cite[Theorem~1.0.2]{HKW} we have that 
\[ [R\Gamma_{c}(\U_n,b_1,\mu_{1})[\pi_{I}]] = \sum_{i = 1}^{r} d_{i}\pi_{I \oplus \{i\}},  \]
in $K_{0}(\U_n(\mathbb{Q}_{p}))$ the Grothendieck group of finite length admissible $\U_n(\mathbb{Q}_{p})$-representations, where we have used Theorem \ref{Uncompatibility} to deduce that the Fargues--Scholze parameter of $\pi_{I}$ is supercuspidal to see that the error term in \cite[Theorem~1.0.2]{HKW} is trivial. However, by Proposition \ref{unmiddledegree}, the complex  $R\Gamma_{c}(G,b_1,\mu_{1})[\pi_{I}]$ is concentrated in degree $0$ and is valued in smooth representations with fixed central character equal to $\omega$ the central character of $\pi_{I}$. Moreover, since the parameter $\phi$ is supercuspidal, we know that the representations $\pi_{I \oplus \{i\}}$ are supercuspidal. Therefore, since supercuspidal representations are injective/projective in the category of smooth representations with fixed central character, this extension splits, and we deduce the following. 
\begin{corollary}
For a supercuspidal $L$-parameter $\phi$ and all $\pi_{I} \in \Pi_{\phi}(J_{b_1}, \varrho_{b_1} \circ \varrho)$, we have an isomorphism 
\[ R\Gamma_{c}(\U_n,b_1,\mu_{1})[\pi_{I}] = \bigoplus_{i = 1}^{r} \pi_{I \oplus \{i\}}^{\oplus d_{i}} \]
of $G(\mathbb{Q}_{p})$-representations.
\end{corollary}
More generally, we can apply Corollary \ref{Unextmucor} together with \cite[Theorem~1.0.2]{HKW} and Theorem \ref{Uncompatibility} to argue exactly as above and deduce that the following is true.  
\begin{corollary}{\label{UnGdescr}}
Let $\phi$ be a supercuspidal $L$-parameter. For all $\pi_{I} \in \Pi_{\phi}(J_{b_{d}}, \varrho_{b_d} \circ \varrho)$, we have an isomorphism 
\[ R\Gamma_{c}(\U_n,b_{d},\mu_{d})[\pi_{I}] = (\bigoplus_{\substack{J \in \mc{P}(\{1,\ldots,r\}) \\ |J| = d}} \pi_{I \oplus J}^{\oplus d_{J}}) \oplus \bigoplus_{\substack{J \in \mc{P}(\{1,\ldots,r\}) \\ |J| < d}} \pi_{I \oplus J}^{\oplus \dim(\mathrm{Hom}_{S_{\phi}}(\tau_{J},r_{\mu_{d}} \circ \phi))}  \] 
of $J_{b_{d}}(\mathbb{Q}_{p})$-representations, where $d_{J} := \prod_{j \in J} d_{j}$ and $J$ ranges over all finite index sets satisfying $|J| = d \mod 2$. 
\end{corollary}
We will now combine this with the spectral action to describe the Weil group action on this complex and its non-minuscule analogues. Since this claim provides no information on the Weil group action, it will not give us an exact description of the $\Act$-functors a priori, just their values up to permuting the $L$-packet. This is enough to construct Fargues' conjectured eigensheaf, but it is not enough to completely prove the Kottwitz conjecture or compute the expected values of the $\Act$-functors. Nonetheless, we will show that we can still completely describe the $\Act$-functors by combining Proposition \ref{prop: Allactval} with Corollary \ref{actcentralisog}.

In particular, by applying Theorem \ref{thm: RGammavsAct}, we deduce that for a finite index set $I \subset \{1,\ldots,r\}$ such that $|I|= \kappa(z) \mod 2$, we have an isomorphism 
\[ R\Gamma_{c}(\U_n,b,\mu)[\pi_{I}] \simeq \bigoplus_{i = 1}^{r} \Act_{\tau_{i}}(\pi_{I}) \boxtimes \phi_{i} \]
for all $i = 1,\ldots,r$. This isomorphism combined with Corollary \ref{UnGdescr}, Lemma \ref{actirred}, and Schur's Lemma tells us that there must exist some bijection between the representations $\{\Act_{\tau_{i}}(\pi_{I})\}_i$ and the representations $\{\pi_{I \oplus \{i\}}\}_i$. In particular, by varying $(\GU_n ,\varrho, z)$, we deduce the existence of permutations 
\[ \sigma_{I}: \{1,\ldots,r\} \ra \{1,\ldots,r\} \]
for all finite index sets $I \subset \{1,\ldots,r\}$ defined by the property that 
\[ \Act_{\tau_{i}}(\pi_{I}) \simeq \pi_{I \oplus \sigma_{I}(i)}. \]
We record this now.
\begin{corollary}{\label{cor: Actperm1}}
Let $I \in \mc{P}(\{1,\ldots,r\})$ be an index set. Then there exist permutations $\sigma_{I}: \{1,\ldots,r\}$ defined by the property that there exists an isomorphism:
\[ \Act_{\tau_{i}}(\pi_{I}) \simeq \pi_{I \oplus \sigma_{I}(i)} \]
\end{corollary}
Naturally, we would expect $\pi_{I \oplus \{i\}}$ to correspond to $\Act_{\tau_{i}}(\pi_{I})$, or in other words that $\sigma_{I} = \mathrm{id}$ for all $I \in \mc{P}(\{1,\ldots,r\})$. In which case, we could deduce the Kottwitz conjecture for the non-minuscule shtuka spaces using the monoidal property of the $\Act$-functor, as in the previous section. We will study these permutations more carefully at the end of this section. For now, we will just directly compute the $\Act$-functors by combining Proposition \ref{prop: Allactval} and Corollary \ref{actcentralisog}. We first prove the following lemma.
\begin{lemma}{\label{actirred2}}
For any $I,J \in \mc{P}(\{1,\ldots,r\})$, $\Act_{\tau_{J}}(\pi_{I})$ is a smooth irreducible representation concentrated in degree $0$.
\end{lemma}
\begin{proof}
We prove this by induction on the cardinality of $J$. If $J = \emptyset$ then $\tau_{J}$ is the trivial representation and the claim is obvious. If $|J| = 1$, this follows from the previous corollary. In general, using the monoidal property, we write 
\[ \Act_{\tau_{J}}(\pi_{I}) \simeq \Act_{\tau_{j}} \circ \Act_{\tau_{J \setminus j}}(\pi_{I}) \]
for some $j \in J$, and apply the inductive hypothesis. 
\end{proof}
\begin{proposition}{\label{prop: Unallactval}}
Let $I,J \in \mc{P}(\{1,\ldots,r\})$ be two finite index sets. Then we have an isomorphism:
\[ \Act_{\tau_{J}}(\pi_{I}) \simeq \pi_{I \oplus J} \]
\end{proposition}
\begin{proof}
There exists a lift $\tilde{\phi}$ to a $\GU^*_{n}$-valued parameter such that precomposition with the natural map $S_{\tilde{\phi}} \ra S_{\phi}$ induced by the central isogeny takes $\tau_{J}$ to $\tilde{\tau}_{[J],0}$, for $J \in \mc{P}(\{1,\ldots,r\})$ with image $[J] \in \mc{P}(\{1,\ldots,r\})/\sim$. Choose also an extended pure inner twist $(\U_n, \varrho, z)$ representing the nontrivial element $b \in B(\U^*_n)_{\bas}$ and extend it to an extended pure inner twist $(\GU_n, \tilde{\varrho}, \tilde{z})$ of $\GU^*_n$. This twist has trivial class in $B(\GU^*_n)$ so there exists a $\Q_p$-isomorphism $t: \GU_n \to \GU_n^*$ trivializing it.

We recall from Section \ref{sss: galoisrepcomp} that if we write $I_{\mathrm{odd}} \subset \{1, \dots, r\}$ for the subset for which $d_{i}$ is odd then we can express the two representatives of $[I]$ in $\mc{P}(\{1,\ldots,r\})$ as $I$ and $I \oplus I_{\mathrm{odd}}$. However, we note that since $n$ is odd, $I_{\mathrm{odd}}$ must have odd cardinality and therefore $I$ and $I \oplus I_{\mathrm{odd}}$ have different parity. We write $I_{o}$ (resp. $I_{e}$) for the representative of $[I]$ with odd (resp. even cardinality). If we write $\varphi: \Bun_{\mathrm{U}^*_{n}} \ra \Bun_{\GU^*_{n}}$ for the natural map induced by the inclusion $\U^*_{n}$ to $\GU^*_{n}$, we note that, under the natural map $B(\U^*_{n}) \ra B(\GU^*_{n})$, the two basic elements both map to the trivial element $B(\GU^*_{n})_{\bas}$. Therefore, $\varphi$ induces a natural map of the HN-strata
\[ \Bun_{\U^{*}_{n}}^{\mathrm{ss}} := \Bun_{\U^*_{n}}^{b} \sqcup \Bun_{\U^*_{n}}^{1}  \ra  \Bun_{\GU^*_{n}}^{1}, \]
which is isomorphic to the map 
\[ [\ast/\underline{\U_{n}(\mathbb{Q}_{p})}] \sqcup  [\ast/\underline{\U^*_{n}(\mathbb{Q}_{p})}] \ra [\ast/\underline{\GU^*_{n}(\mathbb{Q}_{p})}] \]
of classifying $v$-stacks. The map $\varphi^{1}: \Bun_{\U^*_{n}}^{1}  \ra \Bun_{\GU^*_{n}}^{1}$ is given by the natural inclusion, and $\varphi^{b}$ is the composition $\U_n \rightarrow \GU_n \xrightarrow{t} \GU^*_n$. From the construction of the local Langlands correspondence for $\GU^*_n$ in terms of $\U^*_n$ (see \cite{BMN}), it follows that $(\varphi^{1})^{*}(\pi_{[I],0}) = \pi_{I_{e}}$, and that $(\varphi^{b})^{*}(\pi_{[I],0}) = \pi_{I_{o}}$. Here we note that, since $n$ is odd, we have $\GU^*_n(\Q_p) = \U^*_n(\Q_p) E^{\times}$, so these restrictions are irreducible. Now, for two finite index sets $I$ and $J$, Corollary \ref{actcentralisog}, tells us that we have isomorphisms
\[ \varphi_{\natural}\Act_{\tau_{J}}(\varphi^{*}(\pi_{[I],0})) \simeq \Act_{\ttau_{[J],0}}(\pi_{I,0} \otimes \varphi_{\natural}(\ol{\mathbb{Q}}_{\ell})|_{\Bun_{\GU_{n}^{*}}}) \simeq \pi_{[I \oplus J],0}^{\oplus 2} \]
where we have used Proposition \ref{prop: Allactval} for the first isomorphism. For the second isomorphism, first note that we can use \cite[Proposition~VII.3.1]{FS} to rewrite this as $\varphi_{\natural}(\ol{\mathbb{Q}}_{\ell}|_{\Bun_{\U_{n}^{*}}^{\mathrm{ss}}})$. Then using Frobenius reciprocity, the above description of $\varphi$ when restricted to $\Bun_{\U_{n}^{*}}^{\mathrm{ss}}$, and the fact that $\varphi_{\natural}$ is a left adjoint to $\varphi^{*}$, we deduce that $\varphi_{\natural}(\ol{\mathbb{Q}}_{\ell}|_{\Bun_{\U_{n}^{*}}^{\mathrm{ss}}}) \simeq \ol{\mathbb{Q}}_{\ell}^{\oplus 2}|_{\Bun_{\GU_{n}^{*}}^{1}}$, giving the second isomorphism. However, by the above discussion, we know that $\varphi^{*}(\pi_{[I],0}) = \varphi^{1*}(\pi_{[I],0}) \oplus \varphi^{b*}(\pi_{[I],0}) = \pi_{I_{e}} \oplus \pi_{I_{o}}$. Thus, we deduce that
\[ \varphi_{\natural}(\Act_{\tau_{J}}(\pi_{I_{e}})) \oplus \varphi_{\natural}(\Act_{\tau_{J}}(\pi_{I_{o}})) \simeq \pi_{[I \oplus J],0}^{\oplus 2} \]
We now claim that this implies that $\Act_{\tau_{J}}(\pi_{I}) \simeq \pi_{I \oplus J}$ for all $I$ and $J$. To see this, first note that $\Act_{\tau_{J}}(\pi_{I})$ is a smooth irreducible representation concentrated in degree $0$ by Lemma \ref{actirred2}, which will have Fargues--Scholze (= usual) parameter $\phi$. It follows that it must equal some $\pi_{K} \in \Pi_{\phi}(J_{b_{d}}, \varrho_{b_d})$, where $d = |I \oplus J| = |K| \mod 2$ (this parity condition can be deduced, for instance, from Theorem \ref{thm: RGammavsAct}). By the previous isomorphism and Schur's lemma, we have an equality:
\[ \ol{\mathbb{Q}}_{\ell}^{\oplus 2} = \mathrm{Hom}_{\Dlis(\Bun_{\GU^*_{n}},\ol{\mathbb{Q}}_{\ell})}(\varphi_{\natural}(\pi_{K} \oplus \pi_{K \oplus I_{\mathrm{odd}}}),\pi_{[I \oplus J],0}) \]
On the other hand, since $\varphi_{\natural}$ is by definition the left adjoint of $\varphi^{*}$, it follows that the RHS is equal to
\[ \mathrm{Hom}_{\Dlis(\Bun_{\U^*_{n}},\ol{\mathbb{Q}}_{\ell})}(\pi_{K} \oplus \pi_{K \oplus I_{\mathrm{odd}}},\varphi^{*}(\pi_{[I \oplus J],0})) = \mathrm{Hom}_{\Dlis(\Bun_{\U^*_{n}},\ol{\mathbb{Q}}_{\ell})}(\pi_{K} \oplus \pi_{K \oplus I_{\mathrm{odd}}},\pi_{I \oplus J} \oplus \pi_{I \oplus J \oplus I_{\mathrm{odd}}}). \]
Hence we must have $K= I \oplus J$ or $K = I \oplus J \oplus I_{\odd}$. The parity condition for $|K|$ then implies that $K= I \oplus J$ as desired.
\end{proof}
By applying Theorem \ref{thm: RGammavsAct}, this allows us to deduce the strongest form of the Kottwitz conjecture for all geometric dominant cocharacters $\mu$, as well as Fargues' conjecture for $\U_{n}$ by arguing as in the previous section; however, we have invoked the full strength of Theorem \ref{hardinput} to do this. Such a result is a bit overkill if all one is interested in is proving compatibility of Fargues--Scholze local Langlands with the refined local Langlands correspondence (See for example \cite{Ham} for a softer proof of compatibility in the case that $G = \mathrm{GSp}_{4}$). In particular, for the rest of the section, we want to explore what one can deduce just using Theorem \ref{Uncompatibility}. We already saw that it almost gives us the complete result, up to showing that the permutations $\sigma_{I}$ are trivial, using \cite[Theorem~1.0.2]{HKW}. However, since \cite[Theorem~1.0.2]{HKW}, always holds for any group (assuming the refined local Langlands with its various desiderata is known) this result is also something we can assume in the general case.  As we will see, compatibility will be sufficient to describe the stalks of Fargues' eigensheaf, and, using the monoidal property of $\Act$-functors, we can reduce showing the $\sigma_{I}$ are trivial to showing just $\sigma_{\emptyset}$ is trivial. We expect such implications of compatibility to persist for a general reductive group $G$, and it would be interesting to study this in complete generality. We now prove the following claim. 
\begin{proposition}{\label{unbijections}}
Let $\pi_{\emptyset}$ be the representation corresponding to the trivial representation with supercuspidal parameter $\phi$. Then we have an isomorphism 
\[ \bigoplus_{I \in \mc{P}(\{1,\ldots,r\})} \Act_{\tau_{I}}(\pi_{\emptyset}) \simeq \bigoplus_{I \in \mc{P}(\{1,\ldots,r\})} \pi_{I} \]
of sheaves on $\Dlis(\Bun_{\U^*_{n}},\ol{\mathbb{Q}}_{\ell})$. In other words, we have isomorphisms 
\[\bigoplus_{\substack{I \in \mc{P}(\{1,\ldots,r\}) \\ |I| = 1 \mod 2}} \Act_{\tau_{I}}(\pi_{\emptyset}) \simeq \bigoplus_{\rho \in \Pi_{\phi}(J_{b}, \varrho_b)} \rho  \]
of sheaves on $\Dlis(\Bun_{\U^*_n}^{b},\ol{\mathbb{Q}}_{\ell}) \simeq \Dc(J_{b}(\mathbb{Q}_{p}),\ol{\mathbb{Q}}_{\ell})$ and an isomorphism \[\bigoplus_{\substack{I \in \mc{P}(\{1,\ldots,r\}) \\ |I| = 0 \mod 2}} \Act_{\tau_{I}}(\pi_{\emptyset}) \simeq \bigoplus_{\pi \in \Pi_{\phi}(\U^*_n, \id)} \pi  \]
of sheaves on $\Dlis(\Bun_{\U^*_n}^{\mathbf{1}},\ol{\mathbb{Q}}_{\ell}) \simeq \Dc(\U^*_n(\mathbb{Q}_{p}),\ol{\mathbb{Q}}_{\ell})$. 
\end{proposition}
\begin{proof}
We will prove this by induction on the integer $d = 0,1,\ldots,r$ and the claim that there exists a bijection between the representations $\Act_{\tau_{I}}(\pi_{\emptyset})$ and $\pi_{I}$ for $I \in  \mc{P}(\{1,\ldots,r\})$ satisfying $|I| = d$. The case $d = 0$ is trivial, and the case $d = 1$ follows from Corollary \ref{cor: Actperm1}. Now, by Corollary \ref{UnGdescr}, Corollary \ref{actirred}, and Schur's Lemma, we obtain a bijection between the representations $\pi_{I}$ and $\Act_{\tau_{I}}(\pi_{\emptyset})$ for $\tau_{I}$ occurring as summands in $r_{\mu_{d}} \circ \phi|_{W_{E}}$ as an $S_{\phi}$-representation, where $I$ must satisfy $|I| \leq d$ and $|I| = d \mod 2$, and all the representations $\tau_{I}$ for $|I| = d$ occur. By the inductive hypothesis, the representations $\Act_{\tau_{I}}(\pi_{\emptyset})$ and $\pi_{I}$ for $0 \leq |I| < d$ are already in bijection with one another; therefore, we deduce that the representations $\Act_{\tau_{I}}(\pi_{\emptyset})$ and $\pi_{I}$ for $|I| = d$ are in bijection with one another.
\end{proof}
In particular, if we write $\mathcal{G}_{\phi}$ for the sheaf supported on the basic strata $\Bun_{\U^*_n}^{1}$ and $\Bun_{\U^*_n}^{b}$ with values given by $\bigoplus_{\pi \in \Pi_{\phi}(\U^*_n, \id)} \pi$ and $\bigoplus_{\rho \in \Pi_{\phi}(J_{b}, \varrho_b)} \rho$, respectively, we deduce the following.
\begin{corollary}{\label{Cor: Uneigsheaf}}
We have an isomorphism 
\[ k(\phi)_{\mathrm{reg}}\star \pi_{0,\emptyset} \simeq \mathcal{G}_{\phi} \]
of sheaves on $\Dlis(\Bun_{\U^*_n},\ol{\mathbb{Q}}_{\ell})$. In particular, $\mathcal{G}_{\phi}$ is an eigensheaf with eigenvalues $\phi$.
\end{corollary}
\begin{proof}
We have an identification
\[ k(\phi)_{\mathrm{reg}}\star \pi_{0,\emptyset} \simeq \bigoplus_{I \subset \{1,\ldots,r\}} \Act_{\tau_{I}}(\pi_{\emptyset}) \simeq \mathcal{G}_{\phi} \]
where the last isomorphism follows from Proposition \ref{unbijections}. The proof of the eigensheaf property is now the same as in Theorem \ref{thm: GUneigensheaf}.
\end{proof}
In particular, we have shown that the existence of an eigensheaf attached to $\phi$ with the conjectured stalks follows from compatibility of the Fargues--Scholze correspondence with the refined local Langlands correspondence. We suspect that this implication always holds. We now try to answer the question of how much about the values of the $\Act$-functors we can deduce from just compatibility. To do this, we first introduce the following definition.
\begin{definition}
We consider the function 
\[ d(-,-): \mc{P}(\{1,\ldots,r\}) \times \mc{P}(\{1,\ldots, r\}) \ra \mathbb{N}_{\geq 0} \]
which for $I,J \subset \{1,\ldots,r\}$ is given by $d(I,J) := |I \oplus J|$. 
\end{definition}
We have the following. 
\begin{proposition}{\label{prop: actbij}}
For all $I \subset \{1,\ldots,r\}$ and $1 \leq d \leq r$ there is a bijection between the set of $\pi_{J}$ such that $d(I,J) = d$ and $\Act_{\tau_{K}}(\pi_{I})$ for $|K| = d$. 
\end{proposition}
\begin{proof}
The proof is exactly the same as Proposition \ref{unbijections}. In particular, for a fixed $I$, one applies induction to $d$, with the base case being precisely Corollary \ref{cor: Actperm1}, and the inductive step following from Corollary \ref{UnGdescr}.
\end{proof}
Now, we recall that we have a permutation 
\[ \sigma_{\emptyset}: \{1,\ldots,r\} \ra \{1,\ldots,r\} \]
defined by the equation:
\[ \Act_{\tau_{i}}(\pi_{\emptyset}) \simeq \pi_{\sigma_{\emptyset}(i)} \]
The Kottwitz conjecture would predict that $\sigma_{\emptyset}$ is the identity permutation. This condition actually guarantees that the all the functors $\Act_{\tau_{I}}$ behave as expected. In particular, we have the following.
\begin{proposition}
Suppose that $\sigma_{\emptyset} = \mathrm{id}_{\{1,\ldots,r\}}$ then, for all $I \subset \{1,\ldots,r\}$, we have an isomorphism
\[ \Act_{\tau_{I}}(\pi_{\emptyset}) \simeq \pi_{I} \]
which, by the monoidal property, implies Proposition \ref{prop: Unallactval}. 
\end{proposition}
\begin{proof}
We argue by induction on the quantity $|I|$. For $|I| = 0$ this is the fact that $\Act_{\mathbf{1}}(-)$ is the identity functor, and for $|I| = 1$ this is by assumption. Now let $I$ be an index set with $1 < |I| = d \leq r$. We want to compute the value of $\Act_{\tau_{I}}(\pi_{\emptyset})$.

We enumerate the elements of $I$ by $\{i_{1},i_{2},\ldots,i_{d}\}$. For all $j = 1,\ldots,d$, we have by the monoidal property that $\Act_{\tau_{I}}(\pi_{\emptyset}) \simeq \Act_{\tau_{i_{j}}} \circ \Act_{\tau_{I \setminus \{i_{j}\}}}(\pi_{\emptyset}) \simeq \Act_{\tau_{i_{j}}}(\pi_{I \setminus \{i_{j}\}})$ for varying $j = 1,\ldots,d$, where we have used the inductive hypothesis for the last isomorphism. Now we can apply Proposition \ref{prop: actbij} to the representations $\pi_{I \setminus \{i_{j}\}}$. It tells us that we have a bijection between $\Act_{\tau_{k}}(\pi_{I \setminus \{i_{j}\}})$ for $k = 1,\ldots,r$ and the representations $\pi_{(I \setminus \{i_{j}\}) \oplus \{k\}}$ where we have added or subtracted $k$ from $I \setminus \{i_{j}\}$. However, if $k \in I \setminus \{i_j\}$ then by the inductive assumption, we have $\Act_{\tau_k}(\pi_{I \setminus \{i_j\}}) = \pi_{(I \setminus \{i_j\}) \oplus k}$. Hence, for $k = 1,\ldots,r$ such that $k \notin I \setminus \{i_{j}\}$, we have a bijection between $\Act_{\tau_{(I \setminus \{i_{j}\}) \cup \{k\}}}(\pi_{\emptyset})$ and the representations $\pi_{(I \setminus \{i_{j}\}) \cup \{k\}}$. For varying $j = 1,\ldots,d$ the representations $\pi_I$ and $\Act_I(\pi_{\emptyset})$ are the only ones appearing in every bijection for varying $j = 1,\ldots,d$. Hence, we must have $\Act_I(\pi_{\emptyset})=\pi_I$, as desired. 
\end{proof}
\begin{remark}
We note that the ideas explained in this section illustrate how, knowing compatibility of the Fargues--Scholze local Langlands correspondence with the refined local Langlands correspondence, is enough to show the Hecke eigensheaf attached to supercuspidal $\phi$ has the correct form, as predicted by \cite[Conjecture~4.13]{Fa}, as well as show the Kottwitz conjecture and compute the value of the $\Act$-functors up to the ambiguity of the permutations matching the irreducible summands of $r_{\mu} \circ \phi$ with the representations appearing in $R\Gamma_{c}(G,b,\mu)[\pi]$, using \cite[Theorem~1.0.2]{HKW}. The previous proposition tells us that for $G = \U_{n}$ it is sufficient to answer this question for the complex $R\Gamma_{c}(G,b,\mu_{1})[\pi_{\emptyset}]$. In particular, Kottwitz's and Fargues' conjecture for any quasi-split group $G$ and a supercuspidal $L$-parameter $\phi$ should ultimately be reduced to computing $\Act_{W_{i}}(\pi_{\mathbf{1}})$, where $W_{i} \in \Rep_{\ol{\mathbb{Q}}_{\ell}}(S_{\phi})$ are a set of irreducible representations such that any irreducible $W \in \Rep_{\ol{\mathbb{Q}}_{\ell}}(S_{\phi})$ can be written as a tensor product of some $W_{i}$, and $\pi_{\mathbf{1}} \in \Pi_{\phi}(G)$ is the unique $\mf{w}$-generic representation for some Whittaker datum $\mf{w}$. The key point is that for $G = \U_{n}$ the $W_{i}$ can be chosen to be $\tau_{i}$ for $i \in \{1,\ldots,r\}$, and all these $S_{\phi}$-representations are realized in $r_{\mu_{1}}$. Therefore, it is interesting to wonder if, for any $G$, one can find a nice list of dominant geometric cocharacters $\mu$ such that $r_{\mu} \circ \phi$ realizes some set of generators $W_{i}$, as well as whether the analysis in the end of this section generalizes.
\end{remark}
\section{Conflict of interest statement}
On behalf of all authors, the corresponding author states that there is no conflict of interest.
\section{Data availability statement}
Data sharing not applicable to this article as no datasets were generated or analysed during the current study.
\bibliographystyle{amsalpha}
\bibliography{publications}
\end{document}